\documentclass[11pt]{article}
\usepackage[top=0.9in, bottom=0.8in, left=1in, right=1in]{geometry}
\usepackage{amsmath,amssymb,amsthm,enumerate,bbm}
\usepackage{xcolor}
\usepackage[unicode,breaklinks=true,colorlinks=true]{hyperref}

\usepackage{theoremref}
\usepackage{mathrsfs}
\usepackage[]{authblk}
\usepackage{comment}
\usepackage{cancel}

\allowdisplaybreaks

\usepackage{titletoc}
\titlecontents{section}[1.5em]{\vspace{-2pt}}%
    {\thecontentslabel\enspace\enspace}
    {}
    {\titlerule*[0.5pc]{.}\contentspage}%

\titlecontents{subsection}[3em]{\vspace{-2.5pt}}%
    {\thecontentslabel\enspace\enspace}
    {}
    {\titlerule*[0.5pc]{.}\contentspage}%

\numberwithin{equation}{section}
\newtheorem{thm}{Theorem}[section]
\newtheorem{cor}[thm]{Corollary}
\newtheorem{lem}[thm]{Lemma}
\newtheorem{prop}[thm]{Proposition}

\theoremstyle{remark}
\newtheorem{remark}{Remark}[section]
\theoremstyle{definition}
\newtheorem{example}[thm]{Example}
\newcommand{\bke}[1]{\left ( #1 \right )}
\newcommand{\bkt}[1]{\left [ #1 \right ]}
\newcommand{\bket}[1]{\left \{ #1 \right \}}
\newcommand{\norm}[1]{\left \| #1 \right \|}

\newcommand{\R}{\mathbb{R}}
\newcommand{\N}{\mathbb{N}}
\renewcommand{\div}{\mathop{\rm div}\nolimits}
\newcommand{\curl} {\mathop{\rm curl}}
\newcommand{\pd}{\partial}
\newcommand\Ga{\Gamma}

\newcommand\La{\Lambda}

\newcommand{\si}{\sigma}
\newcommand\De{\Delta}
\newcommand\de{\delta}
\newcommand{\nb}{\nabla}
\newcommand{\lec}{{\ \lesssim \ }}
\newcommand{\gec}{{\ \gtrsim \ }}

\newcommand{\bka}[1]{{\langle #1 \rangle}}
\newcommand{\abs}[1]{\left | #1 \right |}

\newcommand\al{\alpha}
\newcommand\be{\beta}
\newcommand\ep{\epsilon}

\newcommand\e {\varepsilon}

\newcommand\la{\lambda}

\newcommand\Si{\Sigma}

\newcommand{\ZZ}{\mathbb{Z}}
\newcommand{\NN}{\mathbb{N}}

\newcommand{\LN}  {\mathrm{L\mkern-0.5mu N}}

\newcommand{\na}{\nabla}
\newcommand{\td}{\tilde}

\newcommand{\I}{\infty}

\newcommand{\EQ}[1]{\begin{equation} #1 \end{equation}}
\newcommand{\EQS}[1]{\begin{equation}\begin{split} #1 \end{split}\end{equation}}

\newcommand{\EQN}[1]{\begin{equation*}\begin{split} #1 \end{split}\end{equation*}}
\newcommand{\EN}[1]{\begin{enumerate} #1 \end{enumerate}}
\newcommand{\uloc}{{\mathrm{uloc}}}
\newcommand{\loc}{{\mathrm{loc}}}

\newcommand{\bP}{{\mathbf{P}}}
\newcommand{\tsum}{\textstyle \sum}
\newcommand{\one}{\mathbbm{1}}

\begin{document}
\title{The Green tensor of the nonstationary Stokes system in the half space}

\author[1]{\rm Kyungkeun Kang}
\author[2]{\rm Baishun Lai}
\author[3]{\rm Chen-Chih Lai \thanks{current affiliation: Department of Mathematics, Columbia University, New York, NY 10027, USA}} 
\author[4]{\rm Tai-Peng Tsai}
\affil[1]{\footnotesize Department of Mathematics, Yonsei University, Seoul 120-749, South Korea}
\affil[2]{\footnotesize LCSM (MOE) and School of Mathematics and Statistics, Hunan Normal University, Changsha 410081, Hunan, China}
\affil[3,4]{\footnotesize Department of Mathematics, University of British Columbia, Vancouver, BC V6T 1Z2, Canada}

\date{}
\maketitle

\vspace{-1.09cm}
\begin{abstract}
We prove the first ever pointwise estimates of the (unrestricted) Green tensor and the associated pressure tensor of the nonstationary Stokes system in the half-space, for every space dimension greater than one. The force field is not necessarily assumed to be solenoidal. The key is to find a suitable Green tensor formula which maximizes the tangential decay, showing in particular the integrability of Green tensor derivatives. With its pointwise estimates, we show the symmetry of the Green tensor, which in turn improves pointwise estimates. We also study how the solutions converge to the initial data, and the (infinitely many) restricted Green tensors acting on solenoidal vector fields. As applications, we give new proofs of existence of mild solutions of the Navier-Stokes equations in $L^q$, pointwise decay, and uniformly local $L^q$ spaces in the half-space.
\end{abstract}

\tableofcontents

\section{Introduction}

This paper considers the Green tensor of the nonstationary Stokes system in the half space. A major goal is
to derive its pointwise estimates.
%
Denote $x=(x_1,\ldots,x_{n-1},x_n) = (x',x_n)$ and $x^*=(x',-x_n)$ for $x \in \R^n$, $n \ge 2$, and the half space
 $\R^n_+ = \{  (x',x_n)\in \R^n \ | \ x_n>0\}$ with boundary $\Si=\pd\R^n_+$.

\subsection{Background}
The nonstationary Stokes system in the half-space $\R_{+}^{n}$, $n\geq2$, reads
\begin{align}\label{E1.1}
\begin{split}
\left.
\begin{aligned}
u_{t}-\Delta u+\nabla \pi=f \\
 \div u=0
\end{aligned}\ \right\}\ \ \mbox{in}\ \ \R_{+}^{n}\times (0,\infty),
\end{split}
\end{align}
with initial and boundary conditions
\begin{align}\label{E1.2}
u(\cdot,0)=u_0; \qquad  %
u(x',0,t)=0\ \ \mbox{on}\ \ \Si\times (0,\infty).
\end{align}
Here  $u=(u_{1},\ldots,u_{n})$ is the velocity, $\pi$ is the pressure, and $f=(f_{1},\ldots,f_{n})$ is the  external force. They are defined for $(x,t)\in \R_{+}^{n}\times (0,\infty)$.
The \emph{Green tensor} $ G_{ij}(x,y,t)$
and its associated \emph{pressure tensor}  $g_j(x,y,t)$
are defined for $(x,y,t) \in\R^n _+ \times \R^n_+ \times \R$ and
 $1\le i,j\le n$ so that,
for suitable $f$ and $u_0$,
the solution of \eqref{E1.1} is given by
\begin{equation}
\label{E1.3}
u_i(x,t) = \sum_{j=1}^n\int_{\R^n_+}G_{ij}(x,y,t)u_{0,j}(y)\,dy+
\sum_{j=1}^n\int_0^t \int_{\R^n_+} G_{ij}(x,y,t-s) f_j(y,s)\,dy\,ds.
\end{equation}
Another way to write a solution of \eqref{E1.1} uses the \emph{Stokes semigroup} $e^{-t\mathbf{A}}$,
where $\mathbf{A}= - \bP \De$ is the \emph{Stokes operator}, and $\bP$ is the \emph{Helmholtz projection} (see Remark \ref{rem2.6})
\begin{equation}
\label{E1.3a}
u(t) = e^{-t\mathbf{A}}\bP u_0 + \int_0^t e^{-(t-s)\mathbf{A}}\bP f(s)\,ds.
\end{equation}
We may regard the Green tensor $G_{ij}$ as the kernel of $e^{-t\mathbf{A}}\bP$. In using \eqref{E1.3} and
\eqref{E1.3a}, we already exclude weird solutions of \eqref{E1.1} that are unbounded at spatial infinity, and can talk about ``the'' unique solution in suitable classes. For applications to Navier-Stokes equations,
\EQ{\label{NS} \tag{\textsc{NS}}
u_{t}-\Delta u+ \nabla \pi=-u \cdot \nb u ,\quad
 \div u=0, \quad \mbox{in}\ \ \R_{+}^{n}\times (0,\infty),
}
with zero boundary condition, a solution of \eqref{NS} is called a \emph{mild solution} if it satisfies \eqref{E1.3} or \eqref{E1.3a} with
$f = -u \cdot \nb u$ and suitable estimates.

The Stokes semigroup $e^{-t\mathbf{A}}$ and the Helmholtz projection $\bP$ are only defined in suitable functional spaces. When defined, the image of $\bP$ is solenoidal. A vector field $u=(u_1,\ldots,u_n)$ in $\R^n_+$ is called
 {\bf solenoidal} if
\begin{align}\label{solenoidal}
\div u=0,\ \ \ u_{n}|_{\Si}=0.
\end{align}
An equivalent condition for $u \in L^1_\loc(\overline{\R^n_+})$ is
\EQ{
\int_{\R^n_+} u \cdot \nb \phi \,dx=0, \quad \forall \phi \in C^\infty_c(\overline{\R^n_+}).
}
For applications to Navier-Stokes equations, although we may assume $u_0$ is solenoidal, we do not have $\div f=0$ for $f= -u\cdot \nb u$. Hence we cannot omit $\bP$ in the integral of \eqref{E1.3a}.

The initial condition $u(\cdot,0)=u_0$ in \eqref{E1.2} is understood by the weak limit
\EQ{\label{Gij-initial0}
\lim_{t\to 0_+} (u(t),w) = (u_0,w) , \quad \forall w \in C^\infty_{c,\si}(\R^n_+),
}
where
$C^\infty_{c,\si}(\R^n_+)=\{ w \in C^\infty_c(\R^n_+;\R^n): \div w=0\}.$
A strong limit is unavailable unless we further assume
$u_0$ is solenoidal, see Theorem \ref{Convergence-to-initial-data}.
This agrees with the expectation that
\[
\lim_{t\to 0_+} e^{-t\mathbf{A}}\bP u_0 = \bP u_0.
\]

There are many results for \eqref{E1.1} in Lebesgue and Sobolev spaces because the Stokes semigroup and the Helmholtz projection are bounded in $L^q(\R^n_+)$, $1<q<\infty$.
Solonnikov \cite{MR0171094} expressed the solution $u$ in terms of Oseen and Golovkin tensors (see \S\ref{sec2}) and proved estimates of $u_t, \nb^2 u, \nb p$ in $L^q$ %
in $\R^3_+\times \R_+$,
extending the 2D work by Golovkin \cite{MR0128219}. %
Ukai \cite{MR896770} derived
an explicit solution formula to \eqref{E1.1} when $f=0$ in $\R^n_+$, expressed in terms of Riesz operators and the solution operators for the heat and Laplace equations in $\R^n_+$. It is simpler and different from that of \cite{MR0171094} and gives estimates in $L^q$ spaces trivially.
Cannone-Planchon-Schonbek \cite{MR1759797} extended \cite{MR896770} for nonzero $f$ using pseudo-differential operators. Estimates in borderline $L^1$ and $L^\infty$ spaces are studied by Desch, Hieber, and Pr\"{u}ss \cite{DHP}.
Koch and Solonnikov \cite{KS-2002} derived gradient estimates of $u$ in $L^q_{x,t}$ for $q>1$ when $f$ is a divergence of some tensor field.
These results are applied to the study of \eqref{NS} in Lebesgue spaces.

The pointwise behavior of the solutions of \eqref{NS} is less studied, as the Helmholtz projection is not  bounded in $L^\infty$, and there have been no pointwise estimates for $G_{ij}$ except for two special cases to be explained below. To circumvent this difficulty, many researchers
expand explicitly
\[
e^{-t\mathbf{A}}\bP \pd_k (u_k u)
\]
to sums of estimable terms
for the study of \eqref{NS}. See also the literature review for mild solutions later, in particular \eqref{eq:1015}. The drawback of this approach is that it does not apply to general nonlinearities $f=f_0(u,\nb u)$.

The pointwise estimates for $G_{ij}$ and its derivatives will be useful in the following situations:
\EN{
\item It gives direct estimates of the Navier-Stokes nonlinearity without expanding its Helmholtz projection.
\item It works for general nonlinearities, for example, those considered in Koba \cite{MR3302113}, and those from the coupling of the fluid velocity with another physical quantity such as
\[
f_j  = \tsum_k  \pd_k (b_k b_j),\qquad g_j =  -  \tsum_k \pd_k  (\pd_k d \cdot \pd_j d),
\]
where $f$ is the coupling with the magnetic field $b: \R^3_+ \times (0,\I) \to \R^3 $ in the \emph{magnetohydrodynamic equations} in the half space $\R^3_+$ with boundary conditions $b _3=0$ and $(\nabla \times b) \times e_3=0$ (see \cite{MR2921213, MR2957530, MR3121723, MR3366805}), and $g$ is the coupling with the orientation field $d: \R^3_+ \times (0,\I) \to \mathbb{S}^2 $ in the \emph{nematic liquid crystal flows} with boundary conditions $\pd _3 d|_\Si=0$ and $\lim_{|x|\to \I} d=e_3$ (see \cite{MR3945617}).
\item It allows to estimate the contribution from a non-solenoidal initial data, e.g., $u_0 \in L^q$ and in particular when $q=1$, as done by Maremonti \cite{Maremonti-2011} for bounded domains.
\item Pointwise estimates are very useful for
the study of the local and asymptotic behavior of the solutions of \eqref{NS}, see e.g.~\cite{Korobkov-Tsai} and our companion papers \cite{KLLTbs,Greenapp}.
}
In contrast to the absence in the time-dependent case, pointwise estimates for \emph{stationary} Stokes system in the half-space have been known;
See \cite{KMT18} for the literature and the most recent refinement.

We now describe the  two special cases of known pointwise estimates for $G_{ij}$.
For the special case of solenoidal vector fields $f$ satisfying \eqref{solenoidal},
by using the Fourier transform in $x'$ and the Laplace transform in $t$ of the system \eqref{E1.1},
Solonnikov \cite[(3.12)]{MR0415097} derived an explicit formula of the \emph{restricted Green tensor} and their pointwise estimates  for $n=3$ (also see \cite{MR0460931, MR1992567} for $n \ge 2$; The same method is used in \cite{MMP2}). Specifically, he showed that for $u_0=0$,
and $f$ satisfying \eqref{solenoidal},
\EQS{
\label{E1.5}
u_i(x,t) &= \sum_{j=1}^n\int_0^t \int_{\R^n_+} \breve G_{ij}(x,y,t-s) f_j(y,s)dy\,ds,\\
\pi(x,t)&=\sum_{j=1}^n\int_0^t \int_{\R^n_+}\breve g_j(x,y,t-s)f_j(y,s)dy\,ds,
}
with
\EQS{\label{E1.6}
\breve G_{ij}(x,y,t)& = \de_{ij}\Ga(x-y,t) + G_{ij}^*(x,y,t),
\\
G_{ij}^*(x,y,t) &=  -\de_{ij}\Ga(x-y^*,t)
 \\
&\quad
- 4(1-\de_{jn})\frac {\pd}{\pd x_j}
 \int_{\Si \times [0,x_n]} \frac {\pd}{\pd x_i} E(x-z) \Ga(z-y^*,t)\,dz,
\\
\breve g_j(x,y,t)&=4(1-\delta_{jn})\pd_{x_j}\Big[\int_{\Si} E(x-\xi')\partial_n\Gamma(\xi'-y,t)d\xi'\\&
\quad+\int_{\Si}\Ga(x'-y'-\xi',y_{n},t)\pd_nE(\xi',x_n)\,d\xi'\Big],
}
where $y^*=(y',-y_n)$ for $y=(y',y_n)$, and $E(x) $ and $\Ga(x,t)$ are the fundamental solutions of the Laplace
and heat equations in $\R^n$, respectively.  (See \S\ref{sec2}. Our $E(x)$ differs from \cite{MR0415097} by a sign.)  Moreover, $G_{ij}^*$ and $\breve g_j$
satisfy the pointwise bound (\cite[(2.38), (2.32)]{MR1992567}) for $n \ge2$,
\EQS{
\label{Solonnikov.est}
|\pd_{x',y'}^l\pd_{x_n}^k \pd_{y_n}^q\pd_t^m G_{ij}^*(x,y,t)| &\lesssim \frac{e^{-\frac{cy_n^2}t}}{t^{m+\frac{q}2}(|x^*-y|^2+t)^{\frac{l+n}2}(x_n^2+t)^{\frac{k}2}};\\
|\pd_{x,y'}^l \pd_{y_n}^q\pd_t^m\breve g_j(x,y,t)| &\lesssim t^{-1-m-\frac{q}{2}}(|x-y^{*}|^{2}+t)^{-\frac{n-1+l}{2}}e^{-\frac{cy_n^2}{t}}.
}
His argument is also valid for $n=2$ since the fundamental solution $E$ in \eqref{E1.6} has a derivative, thus has the scaling property.

Another special case is the pointwise estimate of the Green tensor by Kang \cite{MR2097573},
 but only when the second variable $y$ is zero, or equivalently $y_n=0$,
\begin{equation}
\label{eq_est_Kang}
|\pd_x^l\pd_t^mG_{ij}(x,y',t)|\lesssim \frac1{t^{m+\frac{1+\alpha}2}(|x-y'|^2+t)^{\frac{l+n-2}2}x_n^{1-\alpha}},
\end{equation}
where $\alpha$ is any number with $0<\alpha<1$, and we identify $y'$ with $(y',0)$. Even for $y=0$, this estimate does not seem optimal because we anticipate the symmetry of the Green tensor (see \thref{prop1}).

\subsection{Results}

The following is our first and key
pointwise estimates of the (unrestricted) Green tensor and its derivatives. Even when restricted to $y=0$, it is better than \eqref{eq_est_Kang} by removing the singularity at $x_n=0$. It will be further improved in \thref{thm3} after we show symmetry.

\begin{prop}[First estimates]\thlabel{prop2}
Let $n\ge 2$, $x,y\in\R^n_+$, $t>0$, $i,j=1,\ldots,n$, and $l,k,q,m \in \N_0$. Let $G_{ij}$ be the Green tensor for the time-dependent Stokes system \eqref{E1.1} in the half-space $\R^n_+$, and $g_j$ be the associated pressure tensor. We have
\EQS{\label{Green_est}
|\pd_{x',y'}^l \pd_{x_n}^k\pd_{y_n}^q\pd_t^m & G_{ij}(x,y,t)| \lec
\frac1{\left(|x-y|^2+t\right)^{\frac{l+k+q+n}2+m}}
\\
&+\frac{\LN_{ijkq}^{mn}}{t^{m}(|x^*-y|^2+t)^{\frac{l+k-k_i+n}2 }(x_n^2+t)^{ \frac{k_i}2 }(y_n^2+t)^{\frac{q}2}},
}
where $k_i=(k-\de_{in})_+$,
\EQ{\label{LNijkqmn.def}
\LN_{ijkq}^{mn} := 1+\de_{n2}\mu_{ik}^m\bkt{\log(\nu_{ijkq}^m|x'-y'|+x_n+y_n+\sqrt{t}) - \log(\sqrt{t})},
}
with $\mu_{ik}^m= 1-(\de_{k0}+\de_{k1}\de_{in})\de_{m0}$, and $\nu_{ijkq}^m =  \de_{q0} \de_{jn} \de_{k(1+\de_{in})} \de_{m0}+\de_{m>0}$.
Also,
\EQ{\label{pressure_est}
|\pd_{x',y'}^l\pd_{x_n}^k\pd_{y_n}^q g_j(x,y,t)|
\lec  t^{-\frac12} \bkt{ \frac1{R^{l+q+n}} \bke{\frac1{ x_n^{k}}+ \de_{k0} \log{\frac R{x_n}} } +  \frac 1{R^{k+n-1}y_n^{l+q+1}}},
}
where $R=|x'-y'|+x_n+y_n+\sqrt{t}\sim|x-y^*|+\sqrt{t}$.%
\end{prop}

\emph{Comments on \thref{prop2}:}
\EN{
\item The numerator $\LN_{ijkq}^{mn}$ is a log correction for $n=2$,
and equals 1 if $n \ge 3$.
The parameters $\mu_{ik}^m,\nu_{ijkq}^m \in\{0,1\}$. For simplicity we may take $\mu_{ik}^m=\nu_{ijkq}^m=1$ for most cases.


%
%

\item As we will see in \thref{gj-decomp}, the pressure tensor $g$ contains a delta function supported at $t=0$. It is not in \eqref{pressure_est} where $t>0$.

\item The estimate \eqref{Green_est} of $\pd_t G_{ij}$ is not integrable for $0<t<1$. It can be improved using the Green tensor equation \eqref{Green-def-distribution} and estimates of $\De_x G_{ij}$ and $\nb_x g_j$.

}

\medskip

With the first estimates, we are able to prove the following theorems on restricted Green tensors, convergence to initial data, and symmetry of the Green tensor.
We say a tensor $\bar G_{ij}(x,y,t)$ is a \emph{restricted Green tensor} if for any solenoidal $u_0$, the vector field $u_i(x,t) =\sum_{j=1}^n \int_{\R^n_+} \bar G_{ij}(x,y,t)u_{0,j}(y)\,dy$ is a solution of the Stokes system \eqref{E1.1}-\eqref{E1.2}.

\begin{thm}[Restricted Green tensors]\label{th6.1}
Let $u_0 \in C^1_{c,\si}(\overline {\R^n_+})$, i.e., it is a vector field in $C^1_c(\overline {\R^n_+};\R^n)$ with $\div u_0=0$ and $u_{0,n}|_\Si=0$. Then
\[
\sum_{j=1}^n \int_{\R^n_+} G_{ij}(x,y,t)u_{0,j}(y)\,dy = \sum_{j=1}^n \int_{\R^n_+} \breve G_{ij}(x,y,t)u_{0,j}(y)\,dy
= \sum_{j=1}^n \int_{\R^n_+} \widehat G_{ij}(x,y,t)u_{0,j}(y)\,dy
\]
as continuous functions in $x\in \R^n_+$ and $t>0$,
where $ \breve G_{ij}(x,y,t)$ is the restricted Green tensor of Solonnikov given in \eqref{E1.6}, and
\EQS{ \label{0827a}
\widehat G_{ij}(x,y,t)& = \de_{ij}\bkt{ \Ga(x-y,t) - \Ga(x-y^*,t) } - 4 \de_{jn}C_i(x,y,t),
}
with $C_i(x,y,t)=\int_0^{x_n}\int_\Si\pd_n\Ga(x-y^*-z,t)\,\pd_iE(z)\,dz'\,dz_n$.
\end{thm}

{\it Comments on Theorem \ref{th6.1}:}
\begin{enumerate}
\item The last term of $\breve G_{ij}$ in \eqref{E1.6} only acts on the \emph{tangential} components $u_{0,j}$, $j<n$. In contrast, the last term of $\widehat G_{ij}$ in \eqref{0827a}
only acts on the \emph{normal} component $u_{0,n}$. We do not know whether \eqref{0827a} has appeared in literature. We will use both $\breve G_{ij}$ and $\widehat G_{ij}$ in the proof of
Lemma \ref{th6.2}. $C_i$ will be defined in \eqref{eq_def_Ci} with estimates in Remark \ref{D_estimate-rmk}.

\item We can get infinitely many restricted Green tensors by adding to $\breve G_{ij}$ any tensor $T_{ij}$ that vanishes on all solenoidal vector fields $f=(f_j)$,  $\int_{\R^n_+} T_{ij}(x,y,t) f_j(y) dy = 0$, for example, a tensor of the form $T_{ij} = \pd_{y_j} T_i(x,y,t)$ with suitable regularity and decay.
We do not need $\sum_i\pd_{x_i} T_{ij}(x,y,t)=0$ nor $T_{ij}|_{x_n=0}=0$ since $\int_{\R^n_+} T_{ij}(x,y,t) f_j(y) dy = 0$.
In fact, if we denote
\[
C_{i}^\sharp(x,y,t) := \int_0^{x_n}\int_{\Si} \Ga(x-y^*-z,t) \pd_iE(z)\, dz'dz_n,
\]
then we have the (more symmetric) alternative forms:
\EQS{\label{0902a}
\breve G_{ij}(x,y,t)& = \de_{ij}\bkt{ \Ga(x-y,t) - \Ga(x-y^*,t) } + 4(1-\de_{jn}) \pd_{y_j} C_{i}^\sharp(x,y,t),
\\
\widehat G_{ij}(x,y,t)& = \de_{ij}\bkt{ \Ga(x-y,t) - \Ga(x-y^*,t) } - 4 \de_{jn} \pd_{y_j} C_{i}^\sharp(x,y,t)
\\
&= \breve G_{ij}(x,y,t) + \pd_{y_j} 4C_{i}^\sharp(x,y,t).
}

\item In contrast, the unrestricted Green tensor $G_{ij}$ is \emph{unique}: We require it to satisfy
the equation \eqref{Green-def-distribution}$_1$, the boundary condition \eqref{Green-def-distribution}$_2$, and the initial condition that the vector field $u_i(x,t)=\sum_{j=1}^n\int_{\R^n_+}G_{ij}(x,y,t)(u_0)_j(y)\,dy$ satisfies $\lim_{t\to 0_+}u(\cdot,t) = \bP u_0$ for any initial data $u_0$ not necessarily solenoidal.
Suppose $(\bar G_{ij}(x,y,t), \bar g_j(x,y,t))$ is another pair of unrestricted Green tensor and pressure tensor. For fixed $j$ and $y$, the difference $u_i(x,t) = (G_{ij}-\bar G_{ij})(x,y,t)$ and its companion pressure $p(x,t)=(g_j- \bar g_j)(x,y,t)$ satisfy the Stokes system \eqref{E1.1} with zero boundary and initial values. Under bounds such as
\[
|u(x,t)| \lec \frac 1{(y_n+|x|+\sqrt t)^n}, \quad |p(x,t)| \lec \frac 1{\sqrt t(y_n+ |x|+\sqrt t)^{n-1}}, \quad 
\]
suggested by Proposition \ref{prop2}, we can show $u=0$ by energy estimate: Testing \eqref{E1.1} by $u \phi_R$ for some cut-off function $\phi_R(x)=\Phi(x/R)$ and integrating over $t_0<t<t_1$, sending $R\to \infty$, and then sending $t_0 \to 0_+$.
(Also see \cite[Theorem 5]{MMP1}). Hence $G_{ij}=\bar G_{ij}$.

\item Theorem \ref{th6.1} is extended to $u_0\in L^p_\si$ in Remark \ref{rem9.2} for $1\le p\le \infty$. When $p=\infty$ we can only show the first equality, and we need $u_0$  in the $L^\infty$-closure of $C^1_c$.
\end{enumerate}

\begin{thm}[Convergence to initial data]  \label{Convergence-to-initial-data}
Let $u(x,t)=\sum_{j=1}^n \int_{\R^n_+} G_{ij}(x,y,t)u_{0,j}(y)\,dy$ for a vector field $u_0 $ in $\R^n_+$. Let ${\bf P}$ be the Helmholtz projection in $\R^n_+$ to be given in Remark \ref{rem2.6}.

\textup{(a)} If $u_0 \in C^1_c(\R^n_+)$,
 then $u(x,t)\to (\bP u_0)(x)$ for all $x\in \R^n_+$, and uniformly for all $x$ with $x_n\ge\de$ for any $\de>0$.

 \textup{(b)} If $u_0 \in L^q(\R^n_+)$, $1<q<\infty$,
 then $u(x,t)\to (\bP u_0)(x)$ in $L^q(\R^n_+)$.

  \textup{(c)} If $u_0 \in C^1_{c,\si}(\overline {\R^n_+})$, i.e., it is a vector field in $C^1_c(\overline {\R^n_+};\R^n)$ with $\div u_0=0$ and $u_{0,n}|_\Si=0$,
 then $u_0=\bP u_0$ and $u(x,t)\to u_0(x)$ in $L^q(\R^n_+)$  for $1<q\le \infty$.
\end{thm}

In Part (a), the support of $u_0$ is away from the boundary. In Part (c), the tangential part of $u_0$ may be nonzero on $\Si$, and $q=\infty$ is allowed.

\medskip

\begin{prop}[Symmetry of Green tensor]\thlabel{prop1}
Let $G_{ij}$ be the Green tensor for the Stokes system in the half-space $\R^n_+$, $n\ge2$. Then for $x,y\in\R^n_+$ and $t\neq 0$ we have
\EQ{\label{Green.symmetry0}
G_{ij}(x,y,t)=G_{ji}(y,x,t),\quad
\forall
x\not=y\in\R^n_+.
}
\end{prop}

For the stationary case, the symmetry is known by Odqvist \cite[p.358]{MR1545170} for $n=3$ and \cite[Lemma 2.1, (2.29)]{KMT18} for $n \ge 2$. We do not know \eqref{Green.symmetry0} for the nonstationary case in the literature.
We will prove \thref{prop1} in Section \ref{sec7}, after we have shown \thref{prop2}.
It gives an alternative proof of  the stationary case for $n\ge3$, see
Remark \ref{stationary-nonstationary}.

Although $G_{ij}$ is symmetric by \thref{prop1}, the restricted Green tensors in \eqref{E1.6} and \eqref{0827a}
are not. For example, if $i<n$ and $j=n$,
\[\breve G_{in}(x,y,t)=0,\ \ \ \breve G_{ni}(y,x,t)= -4
 \int_{\Si \times [0,y_n]} \pd_i\pd_n E(y-z) \Ga(z-x^*,t)\,dz,\]
\[\widehat G_{in}(x,y,t)=-4 C_i(x,y,t),\ \ \ \widehat G_{ni}(y,x,t)= 0.\]

\medskip

By the symmetry of the Green tensor in \thref{prop1}, the estimates in \thref{prop2} can be improved. Our main estimates are the following:

\begin{thm}[Main estimates]\thlabel{thm3}
Let $n\ge 2$, $x,y\in\R^n_+$, $t>0$, $i,j=1,\ldots,n$, and $l,k,q,m \in \N_0$.
We have
\EQS{\label{eq_Green_estimate}
|\pd_{x',y'}^l \pd_{x_n}^k \pd_{y_n}^q \pd_t^m G_{ij}(x,y,t)|\lec&\frac1{(|x-y|^2+t)^{\frac{l+k+q+n}2+m}}
\\&
+\frac{\LN_{ijkq}^{mn}+\LN_{jiqk}^{mn}}
{t^{m}(|x^*-y|^2+t)^{\frac{l+k-k_i+q-q_j+n}2}(x_n^2+t)^{\frac{k_i}2}(y_n^2+t)^{\frac{q_j}2} },
}
where $\LN_{ijkq}^{mn}$ is given in \eqref{LNijkqmn.def}, $k_i=(k-\de_{in})_+$, and $q_j=(q-\de_{jn})_+$.
\end{thm}

{\it Comments on \thref{thm3}:}

\EN{

\item
Assume $ l + k + q + n \ge 3$. For the cases when $k_i = q_j = 0$ and $m=0$,
the time integrals of the above estimates coincide with the well-known estimates of the stationary Green tensor given in \cite[IV.3.52]{GaldiBook2011}. We lose tangential spatial decay in other cases.

\item The estimates of the stationary Green tensor mentioned above have been improved by \cite{KMT18}. For example, when there is no normal derivative and $n+l\ge 3$, \cite[Theorems 2.4, 2.5]{KMT18} show %
\EQ{
|\pd_{x',y'}^l G_{ij}^{0}(x,y)| \lec \frac {x_n y_n^{1+\de_{jn}}}  {|x-y|^{n-2+l}\, |x-y^*|^{2+\de_{jn}}}.
}
(It can be improved using symmetry, but \cite{KMT18} does not have $i=j=n$ case.)
The tangential decay rate is better than the normal decay and the whole space case, probably because of the zero boundary condition. Thus \eqref{eq_Green_estimate} may have room for improvement. Compare \thref{thm4}.

}

\medskip

The following estimates quantify the boundary vanishing of the Green tensor and its derivatives at $x_n=0$ or $y_n=0$.
\begin{thm}[Boundary vanishing]\thlabel{thm4}
Let $n\ge 2$, $x,y\in\R^n_+$, $t>0$, $i,j=1,\ldots,n$, and $l,k,q,m \in \N_0$.
Let $0\le\alpha\le1$.
If $k=0$, we have
\EQS{\label{eq_thm3_al_yn}
\left|\pd_{x',y'}^l\pd_{y_n}^q\pd_t^mG_{ij}(x,y,t)\right|
&\lesssim\frac{x_n^\al}{(|x-y|^2+t)^{\frac{l+q+n}2+m}(|x-y^*|^2+t)^{\frac\al2}}
\\
&\quad +\frac{x_n^\al\,\LN}
{t^{m+\frac{\al}2}(|x-y^*|^2+t)^{\frac{l+q-q_j+n}2} (y_n^2+t)^{\frac{q_j}2}} ,
}
with $\LN= {\textstyle\sum_{k=0}^1}( \LN_{ijkq}^{mn} + \LN_{jiqk}^{mn})(x,y,t) $.
If $q=0$, we have
\EQS{\label{eq_thm3_al_xn}
\left|\pd_{x',y'}^l\pd_{x_n}^k\pd_t^mG_{ij}(x,y,t)\right|
&\lesssim\frac{y_n^\al}{(|x-y|^2+t)^{\frac{l+k+n}2+m}(|x-y^*|^2+t)^{\frac\al2}}
\\
&\quad +\frac{y_n^\al\,\LN}
{t^{m+\frac{\al}2}(|x-y^*|^2+t)^{\frac{l+k-k_i+n}2 } (x_n^2+t)^{\frac{k_i}2}} ,
}
with $\LN= {\textstyle\sum_{q=0}^1}( \LN_{ijkq}^{mn} + \LN_{jiqk}^{mn})(x,y,t)$.
\end{thm}

\subsection{Key ideas and the structure of the proof}

Let us explain the idea for our key result, \thref{prop2}: The major difficulty is to find a formula for the Green tensor in which each term has good estimates. Our first formula \eqref{eq_def_Gij} with the correction term $W_{ij}$ given by
\eqref{eq_Wij_formula0} is obtained from the definition using the Oseen and Golovkin tensors.
The second formula for $W_{ij}$ in \thref{lem_Wij_formula} is obtained using Poisson's formula for the heat equation to \emph{remove the time integration}. The idea of using Poisson's formula is already in the stationary case of \cite{MR548252,KMT18}. Our final formula for the Green tensor in Lemma \ref{Green-formula} is obtained by identifying the cancellation of terms in \thref{lem_Wij_formula}, maximizing the tangential decay.
We further transform the term $\widehat H_{ij}$ in Lemma \ref{Green-formula} in terms of $D_{ijm}$ in Lemma \ref{cancel_C_H}, which are integrals over $\Si \times [0,x_n]$. For $D_{ijm}$, we do space partition and integration by parts to estimate their tangential derivatives, and we explore their algebraic properties, e.g., computing their divergence, to \emph{move normal derivatives to tangential derivatives}. These enable us to
prove \thref{prop2}.

Maximizing the tangential decay is essential: As seen in \thref{prop2}, normal derivatives do not increase tangential decay, and maximal tangential decay allows us to prove the integrability in $y$ of all derivatives of the Green tensor  (uniformly in $x$). This is used in the proofs of \eqref{E10.6} of Lemma \ref{th9.1} and \eqref{S93eq2} of Lemma \ref{th9.6}, both relying on the function $H_1 \in L^1$ for $H_1$ defined in \eqref{1015b}, for the construction of mild solutions of Navier-Stokes equations.  The maximal tangential decay is also used to prove that the Green tensor itself is integrable in $y$, but with an $x_n$-dependent constant,
\EQ{ \int_{\R^n_+} |G_{ij}(x,y,t)|\,dy \lec  \ln (e+\frac {x_n}{\sqrt t}).
}
This is proved in \eqref{C1.def} of  Remark \ref{rem9.2} using \thref{thm4}, and used to prove
an extension of Theorem \ref{th6.1} to the $L^\infty$-setting, see Remark \ref{rem9.2}. In this sense, the Green tensor in the half space has a stronger decay than the whole space case. This phenomenon is well known in the stationary case.

Having the first estimates of both Green tensor and its associated pressure tensor in hand, we can investigate restricted Green tensors and initial values, and prove \thref{prop1} on the symmetry of the Green tensor. Our main estimate \thref{thm3} is proved using  \thref{prop2} and \thref{prop1}. We then prove the  boundary vanishing \thref{thm4} using the normal derivative estimates of \thref{thm3}.

\subsection{Applications}

As an application, we will construct mild solutions of the Navier-Stokes equations in the half space in various functional spaces.
We will provide other applications in forth coming papers \cite{KLLTbs} and \cite{Greenapp}.
Since it is only for illustration, we only consider local-in-time solutions with zero external force.
Fujita-Kato \cite{KF1962, FK1964} and Sobolevskii \cite{Sobolevskii1964} transformed \eqref{NS} into an abstract initial value problem using the Stokes semigroup
\EQ{\label{Duhamel}
u(t) = e^{-t\mathbf{A}} u_0 %
 -\int_0^t e^{-(t-s)\mathbf{A}}\bP\pd_k (u_ku)(s)\, ds,
}
whose solution $u(t)$ lies in some Banach spaces and is called a mild solution of \eqref{NS}.
In the whole space setting,
there is an extensive literature on the unique existence of mild solutions of \eqref{NS}. See e.g.~\cite{FJR, Kato1984, GM1985, Giga1986, Yamazaki2000, Miyakawa2000, KT2001, Brandolese2005, MaTe} for the most relevant to our study.

For mild solutions of \eqref{NS} in the half-space, the unique local and global existence in $L^q(\R^n_+)$ were established by Weissler \cite{Weissler1980} for $3\le n<q<\infty$, by Ukai \cite{MR896770} for $2\le n\le q<\infty$, and by Kozono \cite{Kozono1989} for $2\le n=q$. Canaone-Planchon-Schonbek \cite{MR1759797} established unique existence of solutions in $L^\infty L^3$ with initial data in the homogeneous Besov space $\dot B_{q,\infty}^{3/q-1}(\R^3_+)$. For mild solutions in weighted $L^q$ spaces, we refer the reader to \cite{KK2015, KK2016}.

For solutions with pointwise decay, Crispo-Maremonti \cite{CM2004} proved the local existence of solutions controlled by $(1+|x|)^{-\al}(1+t)^{-\be/2}$, $\al+\be=a\in(1/2,n)$ when $u_0\in L^\infty(\R^n_+, (1+|x|)^a dx)$ and $n\ge3$. If $a\in[1,n)$, they further showed the existence is global in time when $u_0$ is small enough in $L^\infty(\R^n_+, (1+|x|)^a dx)$. The constraints imposed in \cite{CM2004} on $a$ and $n$ are relaxed by Chang-Jin \cite{CJ-NA2017} to $a\in(0,n]$ and $n\ge2$. They proved the existence of mild solutions to \eqref{NS} having the same weighted decay estimate as the Stokes solutions if $a\in(0,n]$. Note that for the case $a=n$, the mild solution is local in time because the weighted estimate of solutions to the Stokes system has an additional log factor. They also obtained the weighted decay estimates for $n<a<n+1$ in \cite{CJ-JDE2017} with an additional condition that $R_j'u_0\in L^\infty(\R^n_+, (1+|x|)^a dx)$.  Regarding solutions whose initial data has no spatial decay, the local existence and uniqueness of strong mild solutions with initial data in $L^\infty$ were established by Bae-Jin \cite{BJ2012}, improving Solonnikov \cite{MR1981302} and
  Maremonti \cite{MR2760555}
for continuous initial data. Recently, Maekawa-Miura-Prange \cite{MMP1} studied the analyticity of Stokes semigroup in uniformly local $L^q$ space via the Stokes resolvent problem and constructed mild solutions in such spaces for $q\ge n$.

In the following, Theorems \ref{thm5}, \ref{thm6} and \ref{thm7} are already known, while Theorems \ref{thm1.8} is new. We will provide new proofs using the following solution formula of \eqref{NS} with the Green tensor
\EQS{
u_i(x,t) &= \sum_{j=1}^n \int_{\R^n_+} \breve G_{ij}(x,y,t) u_{0,j}(y)dy\\
&\quad+ \sum_{j,k=1}^n\int_0^t \int_{\R^n_+} \pd_{y_k} G_{ij}(x,y,t-s) (u_k u_j)(y,s)dy\,ds.
}
We use the restricted Green tensor $\breve G_{ij}$ for the first term and the (unrestricted) Green tensor $G_{ij}$ for the second term.
Note that the second term is written as
\EQS{\label{eq:1015}
-\sum_{j,k=1}^n\int_0^t \int_{\R^n_+} \breve G_{ij}(x,y,t-s) (\bP \pd_k u_ku )_j(y,s)dy\,ds
}
 and
explicitly computed in \cite{CM2004,BJ2012}, as the Green tensor $G_{ij}$ was unknown.

For $1\le q\le \infty$, let
\EQ{\label{Lpsi}
L^p_\si(\R^n_+) = \bket{ f \in L^p(\R^n_+;\R^n) :
\div f=0, \ f_n(x',0)=0}.
}

\begin{thm}\label{thm5}
Let $2\le n \le q\le \infty$ and $u_0 \in L^q_{\si}(\R^n_+)$. If $q=\infty$, we also assume $u_0$ in the $L^\infty$-closure of $C^1_{c,\si}(\overline{\R^n_+})$ and $n\ge3$. There are $T= T(n,q, u_0)>0$ and a unique mild solution $u(t)\in C([0,T]; L^q)$
of \eqref{NS} in the class
\[
\sup_{0<t<T} \bke{ \norm{u(t)}_{L^q} + t^{\frac n{2q}} \norm{u(t)}_{L^\infty}
+ t^{1/2}\norm{\nb u(t)}_{L^q}} \le C_* \norm{u_0}_{L^q}.
\]
We can take $T= T(n,q, \norm{u_0}_{L^q})$ if $n<q \le\infty$.
\end{thm}

This is known in \cite{Weissler1980,MR896770,Kozono1989} for $2\le n\le q<\infty$, and in \cite{BJ2012} for $q=\infty$. (We do not have $n=2$ and $q=\infty$ case of \cite{BJ2012}).

For $a\ge 0$, denote
\EQ{
Y_a = \bigg\{ f \in L^\infty_\loc(\R^n_+) \, \bigg|\, \norm{f}_{Y_a} = \sup_{x \in \R^n_+} |f(x)|\bka{x}^a<\infty
\bigg\},
}
and
\begin{equation}
Z_a = \bigg\{ f \in L^\infty_\loc(\R^n_+) \, \bigg|\, \norm{f}_{Z_a} = \sup_{x \in \R^n_+} |f(x)|\bka{x_n}^a<\infty
\bigg\}.
\end{equation}

\begin{thm}\label{thm6}
Let $n\ge 2$ and $0< a \le n$. For any vector field $u_0 \in Y_a$ with $\div u_0=0$ and $u_{0,n}|_\Si=0$,
there is a strong mild solution $u\in L^\infty(0,T;Y_a)$ of \eqref{NS} for some time interval $(0,T)$.
Moreover, the mild solution is unique in the class $L^\infty(\R^n_+\times (0,T))$.
\end{thm}

\begin{thm}\label{thm1.8}
Let $n\ge 2$ and $0< a \le 1$. For any vector field $u_0 \in Z_a$ with $\div u_0=0$ and $u_{0,n}|_\Si=0$,
there is a strong mild solution $u\in L^\infty(0,T;Z_a)$ of \eqref{NS} for some time interval $(0,T)$.
Moreover, the mild solution is unique in the class $L^\infty(\R^n_+\times (0,T))$.
\end{thm}

Theorem \ref{thm6} corresponds to \cite[Theorem 1]{CJ-NA2017} and \cite[Theorem 2.1]{CM2004}.
Theorem \ref{thm1.8} is new. Its upper bound $a\le 1$ is less than that in Theorem \ref{thm6}.

\medskip

For $1\le q\le \infty$, denote
\EQN{
L^q_{\uloc}(\R^n_+) &= \bigg\{u \in L^q_{\loc}(\R^n_+)\ \big|\
\sup_{x \in \R^n_+} \norm{u}_{L^q(B_1(x) \cap \R^n_+ )} <\infty\bigg\},
\\
L^q_{\uloc,\si}(\R^n_+) &= \bket{ u \in L^q_{\uloc}(\R^n_+;\R^n) \ \big|\
\div u=0, \ u_{0,n}|_\Si=0}.
}

\begin{thm}\label{thm7}
Let $2\le n \le q\le \infty$ and $u_0 \in L^q_{\uloc,\si}(\R^n_+)$.

\textup{(a)} If $n<q\le\I$, and suppose $n\ge3$ if $q=\infty$, there are $T= T(n,q, \norm{u_0}_{L^q_\uloc})>0$ and a unique mild solution $u$ of \eqref{NS} with
\EQS{\label{eq1.25}
u(t)\in L^\infty(0,T; L^q_{\uloc,\si})\cap C((0,T); W^{1,q}_{\uloc,0}(\R^n_+)^n\cap \mathrm{BUC}_\si (\R^n_+)),
\\
\sup_{0<t<T} \bke{ \norm{u(t)}_{L^q_\uloc} + t^{\frac n{2q}} \norm{u(t)}_{L^\infty}
+ t^{1/2}\norm{\nb u(t)}_{L^q_\uloc}} \le C_* \norm{u_0}_{L^q_\uloc}.
}

\textup{(b)} If $n=q$, for any $0<T<\infty$, there are $\ep(T),C_*(T)>0$ such that if $\norm{u_0}_{L^n_\uloc} \le \ep(T)$,
then there is a unique mild solution $u(t)$ of \eqref{NS} in the class \eqref{eq1.25}.
\end{thm}

This theorem is \cite[Propositions 7.1 and 7.2]{MMP1}. Continuity at time zero requires further restrictions on $u_0$.

\medskip

In addition to the existence of mild solutions in various spaces, pointwise estimates of the Green tensor is useful for the study of local and asymptotic behavior of solutions. In a forth coming paper \cite{KLLTbs}, we will use the Stokes flows of \cite{Kang2005} as the profile to construct
solutions of the Navier-Stokes equations in $\R^3_+ \times (0,2)$
with \emph{finite global energy} such that they are globally bounded with spatial decay but their normal derivatives are unbounded near the boundary due to H\"older continuous boundary fluxes which are not $C^1$ in time.
Note that  \cite[Theorem 2.4]{KLLTbs} should be revised according to Theorem \ref{thm3}. In its application in 
 \cite[(4.6)]{KLLTbs}, the first line should be revised, but the conclusion remains the same. We will collect other applications in \cite{Greenapp}.
\medskip

The rest of this paper is organized as follows.
In Sect.~2, we give a few preliminaries and recall the Oseen tensor and the Golovkin tensor.
In Sect.~3, we consider the Green tensor and its associated pressure tensor, and derive their first formulas.
In Sect.~4, we derive a second formula  for the Green tensor which has better estimates.
In Sect.~5, we give the first estimates in \thref{prop2} of the Green tensor and the pressure tensor.
In Sect.~6, we study the restricted Green tensors, and how the solutions converge to the initial values.
In Sect.~7, we prove the symmetry of the Green tensor in \thref{prop1}.
In Sect.~8, the ultimate estimate in \thref{thm3} is derived from \thref{prop2} using the symmetry of the Green tensor and the divergence-free condition. We also estimate their vanishing at the boundary, proving \thref{thm4}.
In Sect.~9, we prove the key estimates for the construction of mild solutions in various spaces for
Theorems \ref{thm5},  \ref{thm6}, \ref{thm1.8} and  \ref{thm7}.

\medskip

\emph{Notation.} We denote $\bka{\xi}=(|\xi|^2+2)^{1/2}$ for any $\xi \in \R^m$, $m \in \NN$.
We denote $f\lesssim g$ if there is a constant $C$ such that $|f|\leq Cg$.
\[
\begin{array}{ccc}
\text{Green tensor} & \cdots & G_{ij} ,\ g_j \\
\text{Oseen tensor} & \cdots & S_{ij},\ s_j  \\
\text{Golovkin tensor} & \cdots &  K_{ij},\ k_j  \\
\text{Fundamental solution of }-\Delta & \cdots & E \\
\text{Heat kernel} & \cdots & \Ga \\
\text{Poisson kernel for heat equation} & \cdots & P
\end{array}
\]

\section{Preliminaries, Oseen and Golovkin tensors}
\label{sec2}

In this section, we first recall a few definitions and estimates from  \cite{MR0171094}. We then give two integral estimates.
We next recall in \S\ref{S2.1} the Oseen tensor \cite{Oseen}, which is the fundamental solution of the nonstationary Stokes system  in $\R^n$. We finally recall in \S\ref{S2.2} the Golovkin tensor
\cite{MR0134083}, which is the Poisson kernel of the nonstationary Stokes system  in $\R^n_+$.

The heat kernel $\Ga$ and the fundamental solution $E$ of $-\De$ are given by
$$
\Gamma(x,t)=\left\{\begin{array}{ll}(4\pi t)^{-\frac{n}{2}}e^{\frac{-x^{2}}{4t}}&\ \text{ for }t>0,\\0&\ \text{ for }t\le0,\end{array}\right.\ \text{ and }\ E(x)=\left\{\begin{array}{ll}\frac1{n\,(n-2)|B_1|}\,\frac1{|x|^{n-2}}&\ \text{ for }n\ge3,\\-\frac1{2\pi}\,\log|x|&\ \text{ if }n=2.\end{array}\right.
$$

The Poisson kernel of $-\De$ in $\R^n_+$ is $P_0 (x)= -2\pd_n E(x)$. We will use %
 \cite[(2.32)]{KMT18} for $n \ge 2$,
\EQ{\label{KMT2.32}
\int_\Si E(\xi'-y)P_0(x-\xi')\,d\xi'=E(x-y^*), \quad P_0 (x)= -2\pd_n E(x).
}
It is because the integral is a harmonic function in $x$ that equals $E(x-y^*)$ when $x_n=0$, and was first used in
Maz$'$ja-Plamenevski{\u\i}-Stupjalis \cite [Appendix 1]{MR548252} to study the stationary Green tensor for $n=2,3$.

We will use the following functions defined in \cite[(60)-(61)]{MR0171094}:
\begin{equation}\label{eq_def_A}A(x,t)=\int_\Si\Ga(z',0,t)E(x-z')\,dz'=\int_\Si\Ga(x'-z',0,t)E(z',x_n)\,dz'\end{equation} and
\begin{equation}\label{eq_def_B}B(x,t)=\int_\Si\Ga(x-z',t)E(z',0)\,dz'=\int_\Si\Ga(z',x_n,t)E(x'-z',0)\,dz'.\end{equation}
They are defined only for $n=3$ in \cite{MR0171094} and differ from \eqref{eq_def_A}-\eqref{eq_def_B} by a factor of $4\pi$.
The estimates for $A, B$, and their derivatives are given in \cite[(62, 63)]{MR0171094} for $n=3$. For general case, we can use the same approach and derive the following estimates for $l+n\ge3$:
\begin{equation}\label{eq_estA}|\pd_x^l\pd_t^mA(x,t)|\lesssim\frac1{t^{m+\frac12}(x^2+t)^{\frac{l+n-2}2}}\end{equation}
 and
\begin{equation}\label{eq_estB}|\pd_{x'}^l\pd_{x_n}^k\pd_t^mB(x,t)|\lesssim\frac1{(x^2+t)^{\frac{l+n-2}2}(x_n^2+t)^{\frac{k+1}2+m}}.\end{equation}
In fact, the last line of \cite[page 39]{MR0171094} gives
\begin{equation}\label{eq_estB2}|\pd_{x'}^l\pd_{x_n}^k B(x,t)|\lesssim\frac1{(x^2+t)^{\frac{l+n-2}2}t^{\frac{k+1}2}}\,e^{-\frac{x_n^2}{10t}}.\end{equation}

\begin{remark}\label{remark_est_AB}
For $n=2$, the condition $l\ge1$ is needed as $A(x,t)$ and $B(x,t)$ grow logarithmically as $|x|\to \infty$.
In fact, one may prove for $n=2$
\[
|A(x,t)|+|B(x,t)|\lesssim\frac{1 + |\log(|x_2|+\sqrt{t})|+|\log(|x_1|+|x_2|+\sqrt{t})|}{\sqrt{t}}.
\]

\end{remark}

\subsection{Integral estimates}
We now give a few useful integral estimates.

\begin{lem}\thlabel{lem6-1}%
For positive $L,a,d$, and $k$ we have
\[\int_0^L\frac{r^{d-1}\,dr}{(r+a)^k}\lesssim\left\{\begin{array}{ll}L^d(a+L)^{-k}& \text{ if } k<d,\\[3pt] L^d(a+L)^{-d}(1+\log_+\frac{L}a) & \text{ if } k=d,\\[3pt] %
L^d(a+L)^{-d}a^{-(k-d)}
& \text{ if } k>d.\end{array}\right.\]
\end{lem}

\begin{proof}
Denote the integral by $I$. If $a\ge\frac{L}2$, then
\[I\lesssim a^{-k}\int_0^L r^{d-1}\,dr\sim L^da^{-k}.\]
If $a<\frac{L}2$, then
\EQN{
I=\int_0^a\frac{r^{d-1}\,dr}{(r+a)^k}+\int_a^L\frac{r^{d-1}\,dr}{(r+a)^k}\lesssim&~a^{d-k}+\int_a^Lr^{d-k-1}\,dr\\
\lesssim&~a^{d-k}+\left\{\begin{array}{ll}L^{d-k}& \text{ if } k<d,\\\log\frac{L}a& \text{ if } k=d,\\a^{d-k}& \text{ if } k>d.\end{array}\right.
}
For $k<d$,
\[I\lesssim\left\{\begin{array}{ll}L^da^{-k}& \text{ if } a\ge\frac{L}2\\[3pt] L^{d-k}& \text{ if } a<\frac{L}2\end{array}\right.\lesssim L^d\max(a,L)^{-k}\lesssim L^d(a+L)^{-k},\]
where we used the fact $2\max(a,L)\ge a+L$. Next, for $k=d$,
\[I\lesssim\left\{\begin{array}{ll}\left(\frac{L}a\right)^d& \text{ if } a\ge\frac{L}2\\[3pt] 1+\log\frac{L}a& \text{ if } a<\frac{L}2\end{array}\right.\lesssim \frac{L^d}{(a+L)^d}(1+ \log_+\frac La),
\]
because $a\ge\frac{L}2$ implies that $\frac{L}a\lesssim1$. Finally, for $k>d$ we get
\[I\lesssim\left\{\begin{array}{ll}L^da^{-k}& \text{ if } a\ge\frac{L}2\\[3pt] a^{d-k}& \text{ if } a<\frac{L}2\end{array}\right.\lesssim a^{-k}\min(a,L)^d  \sim \frac{ L^d}{(a+L)^{d}a^{k-d}}.
\qedhere \]
\end{proof}

\begin{lem}\thlabel{lemma2.2}
Let $a>0$, $b>0$, $k>0$, $m>0$ and $k+m>d$. Let $0\not=x\in \R^d$ and
\[I:=\int_{\R^d}\frac{dz}{(|z|+a)^k(|z-x|+b)^m}.\]
Then, with $R=\max\{|x|,\,a,\,b\}\sim|x|+a+b$,
\EQN{I\lesssim R^{d-k-m} + \de_{kd} R^{-m} \log \frac Ra
+ \de_{md} R^{-k} \log \frac Rb
+ \mathbbm 1_{k>d} R^{-m}a^{d-k}
+ \mathbbm 1_{m>d} R^{-k}b^{d-m}.
}
\end{lem}
\begin{proof}
Decompose $I$ into
\[I=\left(\int_{|z|<2R}+\int_{|z|>2R}\right)\frac{dz}{(|z|+a)^k(|z-x|+b)^m}:=I_1+I_2.\]
For $I_2$ we have
\[I_2 \lec \int_{|z|>2R}\frac{dz}{|z|^k|z|^m}\sim R^{d-k-m}.\]
For $I_1$ we consider the three cases concerning $R$: $R=|x|$, $R=a$, and $R=b$.
\begin{itemize}
\item If $R=|x|$, we split $I_1$ into
\[I_1=\left(\int_{|z|<\frac{R}2}+\int_{|z-x|<\frac{R}2}+\int_{\substack{\frac{R}2<|z|<2R\\|z-x|>\frac{R}2}}\right)\frac{dz}{(|z|+a)^k(|z-x|+b)^m}=:I_{1,1}+I_{1,2}+I_{1,3}.\]
By \thref{lem6-1} we obtain
\EQN{
I_{1,1}\lesssim&\int_{|z|<\frac{R}2}\frac{dz}{(|z|+a)^kR^m}\\\sim&~R^{-m}\int_0^{\frac{R}2}\frac{r^{d-1}\,dr}{(r+a)^k}\lec\left\{\begin{array}{ll}R^{d-m}(a+R)^{-k}& \text{ if } k<d,\\R^{-m}\left(1+\log_+\frac{R}a\right)& \text{ if } k=d,\\R^{-m}a^{-k}\min(a,R)^d& \text{ if } k>d\end{array}\right.
}
since $|z-x|\ge|x|-|z|=R-|z|\ge\frac{R}2$. Also by \thref{lem6-1},
\EQN{
I_{1,2}\lec\int_{|z-x|<\frac{R}2}\frac{dz}{R^k(|z-x|+b)^m}\sim&~R^{-k}\int_0^{\frac{R}2}\frac{r^{d-1}\,dr}{(r+b)^m}\\
\lec&\left\{\begin{array}{ll}R^{d-k}(b+R)^{-m}& \text{ if } m<d,\\R^{-k}\left(1+\log_+\frac{R}b\right)& \text{ if } m=d,\\R^{-k}b^{-m}\min(b,R)^d& \text{ if }m>d\end{array}\right.
}
since $|z|+a\ge|x|-|z-x|=R-|z-x|>\frac{R}2$, and
\[I_{1,3}\lec\int_{\substack{\frac{R}2<|z|<2R\\|z-x|>\frac{R}2}}\frac{dz}{|z|^k|z-x|^m}\lec R^{-k-m}\int_{\frac{R}2}^{2R}r^{d-1}\,dr\sim R^{d-k-m}.\]
\item If $R=a>|x|$,
\EQN{
I_1\le\int_{|z|<2R}\frac{dz}{a^k(|z-x|+b)^m}\le a^{-k}\int_{|z-x|<3R}\frac{dz}{(|z-x|+b)^m}=&~R^{-k}\int_0^{3R}\frac{r^{d-1}\,dr}{(r+b)^m}\\
\sim&~I_{1,2}.
}
\item If $R=b>|x|$
\EQN{
I_1\le\int_{|z|<2R}\frac{dz}{(|z|+a)^kb^m}=R^{-m}\int_0^{2R}\frac{r^{d-1}\,dr}{(r+a)^k}\sim I_{1,1}.
}
\end{itemize}
Combining the above cases, the proof is complete.
\end{proof}

\subsection{Oseen tensor}\label{S2.1}
We first recall the Oseen tensor $S_{ij}(x,y,t) = S_{ij}(x-y,t)$, derived by Oseen in \cite{Oseen}. For the Stokes system in $\R^{n}$:
\begin{align}\label{NE1.4}
\begin{split}
\left.
\begin{aligned}
v_{t}-\Delta v+\nabla q=f\\
v(x,0)=0, \ \ \div v=0
\end{aligned}\ \right\}\ \ \mbox{in}\ \ \R^{n}\times (0,+\infty),
\end{split}
\end{align}
with $f(\cdot,t)=0$ for $t<0$,
the unknown $v$ and $q$ are given by (see e.g.~\cite{FJR} or \cite[(46)]{MR0171094}):
\begin{align*}
v_{i}(x,t)
=&~\sum_{j=1}^n\int_{0}^{t}\int_{\R^{n}}S_{ij}(x-y,t-s)f_{j}(y,s)dyds,
\end{align*}
and
\begin{align*}
q(x,t)=\sum_{j=1}^n\int_{-\infty}^{\infty}\int_{\R^{n}}s_j(x-y,t-s)f_{j}(y,s)dyds
=-\sum_{j=1}^n \int_{\R^{n}}\pd_j E(x-y)f_{j}(y,t)dy.
\end{align*}
Here $(S_{ij},s_j)$, the Oseen tensor, is the fundamental solution of the non-stationary Stokes system in $\R^{n}$, and
\EQS{\label{eq_def_Sij}
S_{ij}(x,t)&=\de_{ij}\Ga(x,t)+\Ga_{ij}(x,t),\\
\Ga_{ij}(x,t)&=\pd_i\pd_j\int_{\R^{n}}\Gamma(x-z,t)E(z)dz,
}
\EQ{\label{sj.def}
s_j(x,t)=-\pd_jE(x)\de(t).
}
In \cite[(41), (42), (44)]{MR0171094} it is shown that (for $n=3$, but the general case can be treated in the same way)
\begin{equation}\label{eq_estSij}
|\pd_x^l\pd_t^m\Ga(x,t)| + |\pd_x^l\pd_t^m\Ga_{ij}(x,t)| + \left|\pd_x^l\pd_t^mS_{ij}(x,t)\right|\lesssim\frac1{\left(x^2+t\right)^{\frac{l+n}2+m}}
\end{equation}
for $n\ge2$. It holds for $n=2$ since we can apply one derivative on $E$ to remove the $\log$.

\begin{remark}
\label{rem2.1} Formally taking the zero time limit of \eqref{eq_def_Sij}, we get
\EQ{\label{Oseen-initial1}
S_{ij}(x,0_+)=\de_{ij}\de(x)+\pd_i\pd_jE(x).
}
An exact meaning of \eqref{Oseen-initial1} is given by Lemma \ref{Oseen-initial}.
In other words, the zero time limit of the Oseen tensor is the kernel of the Helmholtz projection $\bP_{\R^n}$ in $\R^n$,
\EQ{
(\bP_{\R^n}  u)_i = u_i + \pd_i (-\De)^{-1} \nb \cdot u.
}
\end{remark}

\begin{lem}\label{Oseen-initial}
Fix $i,j\in \{1,\ldots,n\}$, $n \ge 2$.
Suppose $f\in C^1_c(\R^n)$. Let $v(x,t)=\int_{\R^n} S_{ij}(x-y,t)  f(y)\, dy$ and $v_0(x)= \de_{ij} f(x) +\pd_i \int_{\R^n} \pd_j E(x-y)    f(y) \,dy$. Then
\[
\lim_{t\to0_+} \sup_{x \in \R^n} \bka{x}^{n}\,|v(x,t)-v_0(x)|=0.
\]
\end{lem}

Some regularity of $f$ is needed to ensure $L^\infty$ convergence because $v_0$ may not be continuous if we only assume $f\in C^0_c$. By Lemma \ref{Oseen-initial} and approximation,
the convergence $v(\cdot,t)\to v_0$ is also valid in $L^q(\R^n)$, $1< q<\infty$, for $f \in L^q(\R^n)$.

\begin{proof}
 We first consider $u(x,t)=\int_{\R^n} \Ga(x-y,t)a(y)\,dy$ for a bounded and uniformly continuous function $a$. Let $M=\sup |a|$. For any $\e>0$, by uniform continuity, there is $r>0$ such that $|a(x)-a(y)|\le \e$ if $|x-y|\le r$. Using $\int_{\R^n} \Ga(x-y,t) \,dy=1$,
\EQN{
\abs{u(x,t)-a(x)} &= \abs{\bke{\int _{B_r(x)} + \int_{B_r^c(x)}} \Ga(x-y,t)[a(y) - a(x)]\,dy}
\\& \le \int _{B_r(x)}\Ga(x-y,t) \e\,dy + \int_{B_r^c(x)} \Ga(x-y,t) 2 M\,dy
\\& \le \e + CM \int_{|z|>r} t^{-n/2} e^{-z^2/4t}\,dz
 \le \e + CM e^{-r^2/8t}.
}
This shows $\norm{u(x,t)-a(x)}_{L^\infty(\R^n)} \to 0$ as $t\to 0_+$. Suppose furthermore $a\in C^0_c(\R^n)$, $a(y)=0$ if $|y|>R\ge1$. Then for $|x|>2R$,
\EQN{
\abs{u(x,t)-a(x)} &=\abs{u(x,t)}\le \int _{B_R}  \Ga(x-y,t)M\,dy
\\& \le  CM e^{-|x|^2/32t} \int_{\R^n} t^{-n/2} e^{-|x-y|^2/8t}\,dy = CM e^{-|x|^2/32t}
. %
}
We conclude for any $\al\ge 0$
\EQ{\label{0813a}
\abs{u(x,t)-a(x)}  \le \frac{o(1)}{(|x|+R)^\al }, \quad \forall x \in \R^n,
}
where $o(1)\to0$ as $t\to 0_+$, uniformly in $x$. (Estimate \eqref{0813a} is valid for $n\ge1$.)

Recall the definition \eqref{eq_def_Sij} of $S_{ij} = \de_{ij}\Ga + \Ga_{ij}$.
For $f\in C^1_c(\R^n)$, by \eqref{0813a} with $a=f$,
\[
v(x,t)-v_0(x)  = \frac{o(1)}{(|x|+R)^n}  +  v_1(x,t),
\]
where
\EQN{
v_1(x,t) &= \int_{\R^n}\int_{\R^n}  \Ga (x-y-z,t)\pd_j E(z)dz \,\pd_i f(y) \,dy - \int_{\R^n} \pd_j E(x-w)\pd_i  f(w) \,dw
\\ &= \int _{\R^n} \bke{\int_{\R^n} \Ga (w-y,t)\pd_i f(y) \,dy - \pd_i  f(w)} \pd_j E(x-w) \,dw  .
}
For the second equality we used $z=x-w$ and Fubini theorem.
By \eqref{0813a} again with $a=\pd_i f$,
\[
|v_1(x,t) | \le  \int_{\R^n} \frac{o(1)}{(|w|+R)^{n+1} } |x-w|^{1-n}\,dw  \le \frac{o(1)}{R(|x|+R)^{n-1} }.
\]
We have used \thref{lemma2.2} for the second inequality.

We now improve its decay in $|x|$ and assume $|x|>R+1$. Decompose $\R^n = U+V$ where $U=\{w: |w-x|<|x|/2\}$ and $V=U^c$. Integrating by parts in $w_i$ in $V$, we get
\EQN{
v_1(x,t) &=
 \int _{U} \bke{\int_{\R^n} \Ga (w-y,t)\pd_i f(y) \,dy - \pd_i  f(w)} \pd_j E(x-w) \,dw   \\& \quad+
 \int _{V} \bke{\int_{\R^n} \Ga (w-y,t) f(y) \,dy -   f(w)} \pd_i \pd_j E(x-w) \,dw
 \\
 &\quad+ \int_{\pd V}  \bke{\int \Ga (w-y,t) f(y) \,dy -   f(w)} \pd_j E(x-w) \,n_i \,dS_w = I_1+I_2+I_3.
}

By \eqref{0813a} with $a=\pd_i f$,
\[
|I_1 | \le  \int_{U} \frac{o(1)}{(|x|+R)^{n+2} } |x-w|^{1-n}\,dw \le \frac{o(1)}{(|x|+R)^{n+1} }.
\]
By \eqref{0813a} with $a=f$,
\[
|I_2 | \le  \int_{V} \frac{o(1)}{(|w|+R)^{n+1} } |x|^{-n}\,dw \le \frac{o(1)}{R(|x|+R)^{n} },
\]
\[
|I_3 | \le  \int_{\pd V} \frac{o(1)}{(|x|+R)^{n+1} } |x|^{1-n}\,dS_w  \le \frac{o(1)}{(|x|+R)^{n+1} }.
\]
The main term is $I_2$.
This shows the lemma.
\end{proof}

\subsection{Golovkin tensor}\label{S2.2}
The Golovkin tensor $K_{ij}(x,t): \R_{+}^{n}\times \R\to \R$  is the Poisson kernel of the nonstationary Stokes system  in $\R^n_+$, first constructed by Golovkin
\cite{MR0134083} for $\R^3_+$. Consider the boundary value problem of the Stokes system in the half-space:
\begin{equation}\label{0903a}
\begin{split}
\left.
\begin{aligned}
\hat{v}_{t}-\Delta \hat{v}+\nabla p=0\\
\div \hat{v}=0
\end{aligned}\ \right\}\ \ \mbox{in}\ \ \R_{+}^{n}\times (0,\infty),
\\[1mm]
\hat{v}(x',0,t)=\phi(x',t), \ \ \mbox{on}\ \ \Si\times (0,\infty).
\end{split}
\end{equation}
We extend $\phi(x',t)=0$ for $t<0$. By Solonnikov \cite[(82)]{MR0171094}, the Golovkin tensor $K_{ij}(x,t)$ and its associated pressure tensor $k_j$ are explicitly given by
\EQS{\label{eq_def_Kij}
K_{ij}(x,t)&=-2\,\de_{ij}\,\pd_n\Ga(x,t)-4\,\pd_j\int_0^{x_n}\int_\Si\pd_n\Ga(z,t)\,\pd_iE(x-z)\,dz'\,dz_n\\
&\quad -2\,\de_{nj}\pd_iE(x)\de(t),}
\EQ{\label{eq_def_kj}
k_j(x,t)=2\,\pd_j\pd_nE(x)\de(t)+2\,\de_{nj}E(x)\de'(t)+\frac2t\,\pd_jA(x,t),
}
where $A(x,t)$ is defined in \eqref{eq_def_A}.
A solution $(\hat{v},p)$ of \eqref{0903a} is represented by (\cite[(84)]{MR0171094}):
\EQ{\label{Golovkin}
\hat{v}_i(x,t)=\sum_{j=1}^n\int_{-\infty}^{\infty}\int_\Si K_{ij}(x-\xi',t-s)\phi_j(\xi',s)\,d\xi'\,ds
}
and
\EQS{\label{Golovkin_p}
p(x,t)=&~2\sum_{i=1}^n\pd_i\pd_n\int_{\Si}E(x-\xi')\phi_i(\xi',t)\,d\xi'
+2\int_{\Si}E(x-\xi')\pd_t\phi_n(\xi',t)\,d\xi'
\\
&+ \sum_{j=1}^n \pd_{i}  \int_{-\infty}^{\infty}\int_\Si \frac2{t-s} A(x-\xi',t-s) [\phi_i(\xi',s) - \phi_i(\xi',t)]\,d\xi'\,ds.
}
Note that $ \phi_i(\xi',t)$ is subtracted from the last integral to make it integrable.  Alternatively, using $(\pd_t-\De_{x'})A=(-1/(2t))A$ (since
$(\pd_t-\De_{x'}) \Ga(x',0,t) = (\pd_t-\De_{x'})(4\pi t)^{-1/2} \Ga_{\R^{n-1}}(x',t) = -(2t)^{-1} \Ga(x',0,t) $),
$p(x,t)$ can also be expressed as \cite[(85)]{MR0171094}
\EQS{\label{eq84_Solo1968}
p(x,t)=&~2\sum_{i=1}^n\pd_i\pd_n\int_{\Si}E(x-\xi')\phi_i(\xi',t)\,d\xi'
+2\int_{\Si}E(x-\xi')\pd_t\phi_n(\xi',t)\,d\xi'\\
&-4\sum_{i=1}^n(\pd_t-\De_{x'})\int_{-\infty}^\infty\int_{\Si}\pd_iA(x-\xi',t-\tau)\phi_i(\xi',\tau)\,d\xi'd\tau.
}
The last term of \eqref{eq_def_kj} is not integrable (hence not a distribution) and has to be understood in the sense of \eqref{Golovkin_p} or \eqref{eq84_Solo1968}.
By \cite{MR0171094} (for $n=3$, but the general case can be treated in the same manner), for $n \ge 2$, the Golovkin tensor satisfies,
for $i,j=1,\ldots,n$ and $t>0$,
\begin{equation}\label{eq_estKij}
\left|\pd_{x'}^l\pd_{x_n}^k\pd_t^mK_{ij}(x,t)\right|\lesssim\frac1{t^{m+\frac12}\left(x^2+t\right)^{\frac{l+n-\si}2}(x_n^2+t)^{\frac{k+\si}2}}, \quad \si=\de_{i<n} \de_{jn}.
\end{equation}
Here $\si=1$ if $i<n=j$ and $\si=0$ otherwise.
Specifically,
the case $j<n$ is \cite[(73)]{MR0171094}, the case $j=n$ uses $j<n$ case, the formulas for $K_{in}$ on \cite[page 47]{MR0171094},
 and \cite[(69)]{MR0171094}.

%

\begin{remark}
(i)
In the proof of \cite[(73)]{MR0171094}, in the equation after \cite[(72)]{MR0171094},  there is at least one $x'$-derivative acting on $B$ (defined in \eqref{eq_def_B}) even if $l=0$. The same is true for formulas  for $K_{in}$ on \cite[page 47]{MR0171094}.
Hence we have estimate \eqref{eq_estKij} for all $n \ge 2$ and do not have a log factor for $n=2$. Compare \eqref{eq_estB} and Remark \ref{remark_est_AB}.

\smallskip

(ii) Solonnikov  \cite[pp.46-48]{MR0171094} decomposes $\hat v = w + w'$ where
\EQN{
w_i (x,t)&= \sum_{j<n} \iint K_{ij}(x-\xi',t-s) \phi_j(\xi',s)d\xi'ds, \\
w_i' (x,t)&=  \iint K_{in}(x-\xi',t-s) \phi_n(\xi',s)d\xi'ds,
}
and shows that $w_i(x',0,t)=(1-\de_{in})\phi_i(x',t)$ and
$w_i'(x',0,t)=\de_{in}\phi_i(x',t)$.

\smallskip

(iii) The limit of $\hat v(\cdot,t)$ as $t \to 0_+$ depends on the $\lim_{t \to 0_+} \phi(\cdot,t)$. It is in general nonzero unless $\phi(\cdot,t)=0$ for $0<t<\de$. See the following example.
\end{remark}

\begin{example}
Let $\rho(\xi',t)$ be any continuous function defined on $\Si \times \R$ with suitable decay. Let
\[
u(x,t) = \nb_x h(x,t),\quad
h(x,t) = \int_\Si -2E(x-\xi') \rho(\xi',t)d\xi'.
\]
Let $\hat v(x,t)$ be defined by \eqref{Golovkin} with $\phi(x',t) = u(x',0,t)$. We claim that
$\hat v(x,t)=u(x,t)$.
Note that $h$ is harmonic in $x$ and $u_n|_\Si = \rho$ as $-2\pd_n E(x)$ is the Poisson kernel of $-\De$ in $\R^n_+$.
Since $\div u=0$ and $\curl u=0$, by Stein \cite{Stein70} Theorem III.3 on page 65,  we have
\[
u_n|_\Si = \rho, \quad u_i |_{\Si}  = R_i' \rho \quad (i<n),
\]
where $R'_j$ is the $j$-th Riesz transform on $\R^{n-1}$, $\widehat{R'_j f}(\xi') = \frac{i \xi_j}{|\xi'|} \hat f(\xi')$.
By \eqref{eq_def_Kij} and \eqref{Golovkin},
\EQN{
\hat v_i(x,t)=&-2\int_{-\infty}^\infty\int_\Si\pd_n\Ga(x-\xi',t-s)\phi_i(\xi',s)\,d\xi'ds\\
&-4\sum_{j=1}^{n-1}\int_{-\infty}^\infty \int_\Si\pd_{x_j}\bke{ \int_0^{x_n}\int_\Si\pd_n\Ga(z,t-s)\,\pd_iE(x-\xi'-z)\,dz'\,dz_n } \phi_j(\xi',s)\,d\xi'ds\\
&-4\int_{-\infty}^\infty \int_\Si\pd_{x_n}\bke{ \int_0^{x_n}\int_\Si\pd_n\Ga(z,t-s)\,\pd_iE(x-\xi'-z)\,dz'\,dz_n } \phi_n(\xi',s)\,d\xi'ds\\
&-2\int_\Si\pd_iE(x-\xi')\phi_n(\xi',t)\,d\xi' = : I_1+I_2+I_3+I_4.
}

As $\phi_n= \rho$, $I_4=u_i(x,t)$ by definition. If $i<n$, since $\phi_j  = R_j' \rho$, we can switch derivatives
\EQN{
I_2 &= -4\sum_{j=1}^{n-1}\int_{-\infty}^\infty \int_\Si\pd_{x_j}\bke{ \int_0^{x_n}\int_\Si\pd_n\Ga(z,t-s)\,\pd_jE(x-\xi'-z)\,dz'\,dz_n } \phi_i(\xi',s)\,d\xi'ds\\
&=4\int_{-\infty}^\infty \int_\Si\pd_{x_n}\bke{ \int_0^{x_n}\int_\Si\pd_n\Ga(z,t-s)\,\pd_nE(x-\xi'-z)\,dz'\,dz_n } \phi_i(\xi',s)\,d\xi'ds \\
&\quad + 2\int_{-\infty}^\infty \int_\Si\pd_{n}\Ga(x-\xi',t-s) \phi_i(\xi',s)\,d\xi'ds = I_{2a}+I_{2b}.
}
The second equality is \cite[(68)]{MR0171094}.
Note that $I_{2b}$ cancels $I_1$, and $I_{2a}+I_3=0$ because
\[
-2\int_\Si \pd_iE(x-\xi'-z)\phi_n(\xi',s)\,d\xi' =  u_i(x-z,s)
=-2 \int_\Si \pd_nE(x-\xi'-z)\phi_i(\xi',s)\,d\xi'.
\]
The first equality is by definition of $u_i$. The second is because $-2\pd_n E$ is the Poisson kernel.
Thus $\hat v_i(x,t)=u_i(x,t)$ for $i<n$. As they are harmonic conjugates of $\hat v_n$ and $u_n$, and
$\hat v_n$ and $u_n$ have the same boundary value $\rho$,
we also have $\hat v_n(x,t)=u_n(x,t)$.
\hfill $\square$
\end{example}

\section{First formula for the Green tensor}\label{S2.3}

In this section, we derive a formula of the Green tensor $G_{ij}$ of the non-stationary Stokes system in the half-space. We decompose $G_{ij}=\tilde{G}_{ij}+W_{ij}$ with explicit $\tilde G_{ij}$ given by \eqref{tdG.def}, and derive a formula for the remainder term $W_{ij}$. %

For the nonstationary Stokes system in the half-space $\R^n_+$, $n\ge2$, the Green tensor $G_{ij}(x,y,t)$ and its associated pressure tensor $g_j(x,y,t)$, for each fixed $j=1,\ldots,n$ and $y\in\R^n_+$, satisfy
\EQS{\label{Green-def-distribution}
\pd_t G_{ij}-\De_xG_{ij}+\pd_{x_i}g_j=\de_{ij}\de_y(x)\de(t),\quad \sum_{i=1}^n\pd_{x_i}G_{ij}=0,\quad \text{ for }x\in\R^n_+\text{ and } t\in\R,
}
\EQN{
G_{ij}(x,y,t)|_{x_n=0}=0.
}
Recall the defining property that
solution $(u,\pi)$ of \eqref{E1.1}-\eqref{E1.2} with zero boundary condition is given by \eqref{E1.3}
and
\begin{equation}\label{eq2.12}
\pi(x,t) = \int_{\R^n_+}g(x,y,t)\cdot u_0(y)\,dy+\int_{-\infty}^\infty \int_{\R^n_+} g(x,y,t-s)\cdot f(y,s)\,dy\,ds.
\end{equation}
The time interval in \eqref{eq2.12} is the entire $\R$ as we will see in \thref{gj-decomp} that $g$ contains a delta function in time, cf.~\eqref{sj.def}. In contrast, $G_{ij}$ is a function and we can define $G_{ij}(x,y,t)=0$ for $t \le 0$ in view of \eqref{E1.3}. Note that $G_{ij}(x,y,0_+)\not=0$, see \thref{th:Gij-initial}.

\medskip

We now proceed to find a formula for $G_{ij}$.
Let $u,\pi$ solve \eqref{E1.1}-\eqref{E1.2} with zero external force $f=0$, and non-zero initial data $u(x,0)=u_0(x)$, in the sense of \eqref{Gij-initial0}. Then
\EQ{\label{ui.def}
u_i(x,t)=\sum_{j=1}^n\int_{\R^n_+}G_{ij}(x,y,t)(u_0)_j(y)\,dy,
}
and $\pi$ is given by \eqref{eq2.12} with $f=0$.
Let ${\bf E}u_0$ be an extension of $u_0$ to $\R^n$ by
\EQ{\label{div0-extension}
{\bf E}u_0(x',x_n)=(-u_0',u_0^n)(x',-x_n)\ \text{ for }x_n<0.
}
Then $\div{\bf E}u_0(x',x_n)=-\div u_0(x',-x_n)$ for $x_n<0$.

\medskip

\begin{remark}
If $\div u_0=0$ and $u_0^n(x',0)=0$, then $\div {\bf E}u_0=0$ in $\mathcal{D}'(\R^n)$.
\end{remark}

\medskip

Let $\tilde{u}$ be the solution to the homogeneous Stokes system in $\R^n$ with initial data ${\bf E}u_0$. Then
\EQN{
\tilde{u}_i(x,t)=&\sum_{j=1}^n\int_{\R^n}S_{ij}(x-y,t)({\bf E}u_0)_j(y)\,dy\\
=&\sum_{j=1}^n\int_{\R^n_+}(S_{ij}(x-y,t)-\ep_jS_{ij}(x-y^*))(u_0)_j(y)\,dy\\
=&\sum_{j=1}^n\int_{\R^n_+}\tilde{G}_{ij}(x,y,t)(u_0)_j(y)\,dy,
}
where
\EQS{\label{tdG.def}
\tilde{G}_{ij}(x,y,t)=S_{ij}(x-y,t)-\epsilon_jS_{ij}(x-y^*,t),\quad \epsilon_j=1-2\de_{nj}.
}
Note that the factor $\epsilon_j$ is absent in the second term of Solonnikov's restricted Green tensor \eqref{E1.6}. Eqn.~\eqref{tdG.def} is closer to \cite[(2.22)]{KMT18}.

\begin{lem}\thlabel{thSijodd}
We have
\EQS{\label{Sijstar}
S_{ij}(x^*,t)=\epsilon_i\epsilon_jS_{ij}(x,t),
}
\EQS{\label{eq_tildeGij}
\tilde{G}_{ij}(x,y,t)\big|_{x_n=0}=
2\,\de_{in} \,S_{nj}(x'-y,t).
}
\end{lem}
\begin{proof}
 If $i=j$, then $S_{ii}(x,t)$ is even in all $x_k$. If $i\neq j$, then $S_{ij}(x,t)$ is odd in $x_i$ and $x_j$, but even in $x_k$ if $k\neq i,j$. In particular, with $x_k=x_n$, we get \eqref{Sijstar} for all $i,j=1,\ldots,n$. By \eqref{Sijstar},
\[\left.\tilde{G}_{ij}(x,y,t)\right|_{x_n=0}=S_{ij}(x'-y,t)-\epsilon_jS_{ij}(x'-y^*,t)= S_{ij}(x'-y,t)-\epsilon_iS_{ij}(x'-y,t)\]
which gives \eqref{eq_tildeGij}.
\end{proof}

\medskip

Let $\hat{u}=u-\tilde{u}|_{\R^n_+}$. Then $\hat{u}$ solves the boundary value problem \eqref{0903a} with boundary data $\hat{u}|_{x_n=0}=-\tilde{u}(x,t)|_{x_n=0}$. 
By the Golovkin formula \eqref{Golovkin},
\EQN{
\hat{u}_i(x,t)=&~\sum_{k=1}^n{\int_{-\infty}^{\infty}}\int_{\Si}K_{ik}(x-\xi',t-s)(-\tilde{u}_k(\xi',0,s))\,d\xi' ds\\
=&~\sum_{k=1}^n{\int_{-\infty}^{\infty}}\int_{\Si}K_{ik}(x-\xi',t-s)\left(-\sum_{j=1}^n\int_{\R^n_+}\tilde{G}_{kj}(\xi',y,s)(u_0)_j(y)\,dy\right)\,d\xi' ds\\
=&~\sum_{j=1}^n\int_{\R^n_+}\left(-\sum_{k=1}^n{\int_{-\infty}^{\infty}}\int_{\Si}K_{ik}(x-\xi',t-s)\tilde{G}_{kj}(\xi',y,s)\,d\xi'ds\right)(u_0)_j(y)dy\\
=&~\sum_{j=1}^n\int_{\R^n_+}W_{ij}(x,y,t)(u_0)_j(y)dy,
}
where
\EQS{\label{def_Wij}
W_{ij}(x,y,t)&=-\sum_{k=1}^n{\int_{-\infty}^{\infty}}\int_{\Si}K_{ik}(x-\xi',t-s)\tilde{G}_{kj}(\xi',y,s)\,d\xi'ds.
}

By \thref{thSijodd}, we have the following first formula of $W_{ij}$.
\begin{lem}[The first formula of $W_{ij}$]
For $x,y\in\R^n_+$, $t>0$, and $i,j=1,\ldots,n$
\begin{equation}
\label{eq_Wij_formula0}
W_{ij}(x,y,t)=-2{\int_{-\infty}^{\infty}}\int_\Si K_{in}(x-\xi',t-s)S_{nj}(\xi'-y,s)\,d\xi'ds.\end{equation}
\end{lem}

\begin{remark}\label{rem2.4}
Because of $\de(t)$ in the last term of the formula \eqref{eq_def_Kij} of $K_{ij}$, when we substitute \eqref{eq_def_Kij} into the right side of \eqref{eq_Wij_formula0}, one of the resulting integrals is spatial only.
\end{remark}

\medskip

As $u=\tilde{u}|_{\R^n_+}+\hat{u}$, the Green tensor $G_{ij}$ has the decomposition
\EQS{
\label{eq_def_Gij}
G_{ij}(x,y,t)&=\tilde{G}_{ij}(x,y,t)+W_{ij}(x,y,t)\\
&=S_{ij}(x-y,t)-\epsilon_jS_{ij}(x-y^*,t) + W_{ij}(x,y,t).
}
All of them are zero for $t \le 0$.

With the first formula of $W_{ij}$, we have the scaling property of the Green tensor.
\begin{cor}
For $n\ge2$ the Green tensor $G_{ij}$ obeys the following scaling property
\[G_{ij}(x,y,t)=\la^nG_{ij}(\la x,\la y, \la^2t).\]
\end{cor}

\begin{proof}
Note that $\Ga(\la x,\la^2 t)=\la^{-n}\Ga(x,t)$ and $\de(\la^2 t)=\la^{-2}\de(t)$. It follows directly from \eqref{eq_def_Gij}, \eqref{eq_Wij_formula0} and the scaling properties of $K_{ij}$ and $S_{ij}$.
\end{proof}

\begin{remark}
In Lemma 2.1 of the stationary case of \cite{KMT18}, the condition $n\ge3$ is needed for showing the scaling property of $G_{ij}(x,y)$ because the 2D fundamental solution $E$ does not have the scaling property. However, in the nonstationary case we do not have this issue. So the scaling property of the nonstationary Green tensor holds for all dimension $n\ge2$.
\end{remark}

\medskip

Before we consider the zero time limit of $G_{ij}$, we consider the Helmholtz projection.

\begin{remark}[Helmholtz projection in $\R^n_+$]\label{rem2.6}
For a vector field $u$ in $\R^n_+$, its Helmholtz projection $\bP u$ is given by
\EQ{%
(\bP u)_i = u_i - \pd_i p,
}
where $p$ satisfies $-\De p =- \div u$, and $\pd_n p=u_n$ on $x_n=0$. Using the
Green function of the Laplace equation with Neumann boundary condition, $N(x,y)=E(x-y)+E(x-y^*)$, we have
\begin{equation}
\label{Helmholtz-p0}
p(x) = -\int_{\R^n_+} N(x,y) \div u(y)\,dy - \int_\Si u_n (y)N(x,y)\,dS_y.
\end{equation}
Note the unit outer normal $\nu=-e_n$ and $\frac{\pd p}{\pd \nu}  = - \pd_n p=-u_n$.
The second term is absent in \cite[Appendix]{MR605431}, \cite[(III.1.18)]{GaldiBook2011}, and \cite[Lemma A.3]{MMP2} because they are concerned with $L^q$ bounds of $\bP \tilde u$ with $\tilde u\in L^q$, for which \eqref{Helmholtz-p0} is undefined, and they approximate $\tilde u$ in $L^q$ by $u\in C^\infty_c(\R^n_+)$,  for which the second term in \eqref{Helmholtz-p0} is zero. For our purpose, we want pointwise bounds and hence we need to keep the boundary term.
Integrating by parts, %
\EQ{\label{Helmholtz-p1}
p(x)= \int_{\R^n_+} \pd_{y_j} N(x,y) u_j(y)\,dy .
}
The boundary terms on $\Si$ cancel. Using the definition of $N(x,y)$,
\EQ{
\label{Fj.def}
\pd_{y_j} N(x,y) = - F_j^y(x), \quad   F_j^y(x):= \pd_j E(x-y)+\ep_j \pd_j  E(x-y^*) .
}
Thus
\EQ{\label{Helmholtz.formula}
(\bP u)_i (x)= u_i(x) + \pd_i  \int_{\R^n_+} F_j^y(x) u_j(y)\,dy .
}
\end{remark}

We now consider the zero time limit of $G_{ij}$.

\begin{lem}\thlabel{th:Gij-initial}
(a) For $x,y\in\R^n_+$,
we have
\begin{equation}\label{Gij.t=0}
G_{ij}(x,y,0_+)=\de_{ij}\de(x-y) + \pd_{x_i}F^y_j(x),
\end{equation}
where $F^y_j(x)$ is defined in \eqref{Fj.def},
in the sense that, for any $i,j\in \{1,\ldots,n\}$ and $f\in C^1_c(\R^n_+)$, we have
\begin{equation}\label{Gij.t=0-wk}
\lim _{t \to 0_+} \bke{ \int_{\R^n_+}  G_{ij}(x,y,t)  f(y) dy- \de_{ij} f(x) - \pd_{x_i} \int_{\R^n_+}  F^y_j(x)  f(y) \,dy}=0,
\end{equation}
for all $x \in \R^n_+$, and uniformly for $x_{n}\ge\delta$, for any $\de>0$.

(b) Let $u_0 \in C^1_c(\R^n_+;\R^n)$ be a vector field in $\R^n_+$ and let $u(x,t)$ be given by \eqref{ui.def}.
 Then $u(x,t)\to (\bP u_0)(x)$ for all $x\in \R^n_+$, and uniformly for all $x$ with $x_n\ge\de$ for any $\de>0$.
\end{lem}

Note that $\pd_{x_i}F^y_j(x)$ is a distribution since it may produce delta function at $y$. This lemma shows that the zero time limit of the Green tensor is exactly the Helmholtz projection in $\R^n_+$, given in \eqref{Helmholtz.formula}.
We will show uniform convergence in Lemma \ref{th6.2} where we assume $\bP u_0\in C^1_c(\overline{\R^n_+}) $, allowing nonzero tangential components of $u_0|_\Si$, and show $L^q$ convergence in Lemma \ref{th:Gij-initial-Lq} where we assume $ u_0\in L^q(\R^n_+) $ but do not assume $u_0=\bP u_0$.

\begin{proof} (a) We may extend $f$ to $\R^n$ by setting $f(y)=0$ for $y_n \le 0$.
Recall that
$$
G_{ij}(x,y,t)=\tilde{G}(x,y,t)+W_{ij}(x,y,t), \ \ \mbox{with}\ \ \tilde{G}(x,y,t)=S_{ij}(x-y,t)-\epsilon_jS_{ij}(x-y^*,t).
$$
By \eqref{Sijstar} and Lemma \ref{Oseen-initial},
\begin{equation}\label{0817a}
\lim _{t \to 0_+} \int_{\R^n_+} \tilde G_{ij}(x,y,t)  f(y) dy=\de_{ij} f(x)+\pd_i\int_{\R^n_+}\pd_j\big[E(x-y)-\epsilon_{j}E(x-y^{*})\big]f(y)\,dy,
\end{equation}
uniformly in $x\in \R^n_+$. Now we consider the contribution from $W_{ij}(x,y,t)$. By \eqref{eq_Wij_formula0} and \eqref{eq_def_Kij},
\[
W_{ij}(x,y,t)=-2\int_{-\infty}^\infty \int_{\Si} K_{in}(x-\xi',t-s)S_{nj}(\xi'-y,s)\,d\xi'ds = W_{ij,1}(x,y,t) + W_{ij,2}(x,y,t),
\]
where
\EQN{
W_{ij,1}(x,y,t)&=-2\int_{-\infty}^\infty \int_{\Si} \tilde K_{in}(x-\xi',t-s)S_{nj}(\xi'-y,s)\,d\xi'ds ,
\\
W_{ij,2}(x,y,t)&=4 \int_{\Si}\pd_iE(x-\xi')S_{nj}(\xi'-y,t)\,d\xi',
}
and $\tilde K_{ij}$ is the sum of the first two terms in the definition \eqref{eq_def_Kij} of $K_{ij}$. By \eqref{eq_estKij}, \eqref{eq_estSij}, change of variable $s=u^2$ and Lemma \ref{lemma2.2},
\EQN{
|W_{ij,1}|
&\lesssim\int_0^t\int_{\Si}\frac1{\sqrt{s}(|x-\xi'|+\sqrt{s})^{n-1} (x_n+\sqrt{s})}\,\frac1{(|\xi'-y|+\sqrt{t-s})^n}\,d\xi'ds\\
&\le\int_0^t\int_{\Si}\frac1{\sqrt{s} (x_n+\sqrt{s}) |x-\xi'|^{n-1}}\,\frac1{|\xi'-y|^n}\,d\xi'ds
=2 \log\bke{1+\frac{\sqrt{t}}{x_n}} \int_{\Si}\frac1{|x-\xi'|^{n-1}}\,\frac1{|\xi'-y|^n}\,d\xi'\\
&\lec \log\bke{1+\frac{\sqrt{t}}{x_n}} \bket{|x-y^*|^{-n} + |x-y^*|^{-n} \log \frac{|x-y^*|}{x_n} + |x-y^*|^{-(n-1)}y_n^{-1}}.
}
From this, one has
$$
\lim _{t \to 0_+}\int_{\R^n_+}  W_{ij,1}(x,y,t)  f(y) dy=0
$$
for all $x \in \R^n_+$, and uniformly for $x_{n}\ge \delta>0$.
On the other hand, by Remark \ref{rem2.1}, $W_{ij,2}(x,y,t)$ for $x_n,y_n>0$ as $t\to 0_+$ formally tends to
\EQS{\label{0722b}
4\int_{\Si}\pd_iE(x-\xi')\pd_n \pd_j E(\xi'-y)\,d\xi'&=
-4\pd_{x_i}\pd_{y_j}\int_{\Si}E(x-\xi')\pd_nE(\xi'-y)\,d\xi'\\&=
-2\pd_{x_i}\pd_{y_j}\int_{\Si}E(\xi'-x)P_0(y-\xi')\,d\xi'\\&=
-2\frac{\pd}{\pd x_i}\frac{\pd}{\pd y_j}  E(x-y^*)=
2\ep_j\pd_i\pd_j E(x-y^*),
}
where $P_0=-2\pd_nE$ and we've used \eqref{KMT2.32} for the third equality. It is in the sense of functions since its singularity is at $y=x^* \not \in \R^n_+$. Thus
\EQN{
&\int_{\R^n_+}  W_{ij,2}(x,y,t)  f(y) dy -
\int_{\R_+^{n}}2\ep_j\pd_i\pd_j E(x-y^{*})
f(y)\,dy
\\ & =\int_{\R^n}\!  4 \int_{\Si} \pd_iE(x-\xi')S_{nj}(\xi'-y,t) f(y) dy
-\int_{\R^n}\! 4 \int_{\Si}\pd_iE(x-\xi') \pd_j E(\xi'-y) d\xi' \pd_n  f(y) dy
\\ &= 4\int_{\Si}  \bket{ \int_{\R^n} S_{nj}(\xi'-y,t)  f(y) dy -   \int_{\R^n} \pd_j E(\xi'-y) \pd_n f(y) dy}  \pd_iE(x-\xi')d\xi'.
}
For the first equality we used \eqref{0722b} and integrated by parts in $y_n$ in the second integral using $f\in C^1_c(\R^n_+) $.
For the second equality we used the Fubini theorem.
By Lemma \ref{Oseen-initial} and \thref{lemma2.2}, the above is bounded by
\[
\lec \int_{\Si} \frac {o(1)} {\bka{\xi'}^{n}}\frac 1{|x-\xi'|^{n-1}}  d\xi' \lec \frac{o(1)}{\bka{x}^{n-1}}.
\]
The combination of the above and \eqref{0817a} give Part (a).

Part (b) is a consequence of Part (a) and Remark \ref{rem2.6}.
\end{proof}

\medskip

Finally we derive a formula for the pressure tensor $g_j$, to be used to estimate $g_j$ in \S\ref{S5}, and show symmetry of $G_{ij}$ in \S\ref{sec7}.

\begin{prop}[The pressure tensor $g_j$]\thlabel{gj-decomp}
For $x,y\in\R^n_+$, $t\in \R$, and $j=1,\ldots,n$ we have
\EQS{\label{eq_formula_gj}
g_j(x,y,t) =  \widehat w_j(x,y,t) - F^y_j(x) \de(t),
}
where $\widehat w_j(x,y,t) $ is a function with $\widehat w_j(x,y,t) =0$ for $t \le 0$ and, for $t>0$,
\EQS{\label{w-formula}
\widehat w_j(x,y,t)
=& - \sum_{i<n} 8\int_0^t\int_{\Si}\pd_i \pd_nA(\xi',x_n,\tau)\pd_nS_{ij}(x'-y'-\xi',-y_n,t-\tau)\,d\xi'd\tau\\
&+ \sum_{i<n} 4\int_{\Si}\pd_i E(x-\xi') \pd_nS_{ij}(\xi'-y,t)\,d\xi'\\
&+ 8\int_{\Si} \pd_nA(\xi',x_n,t) \pd_n\pd_jE(x'-y'-\xi',-y_n)\, d\xi'.
}
\end{prop}

\begin{proof}
For fixed $j$, the Green tensor  $({G}_{ij}, g_j)$ satisfies \eqref{Green-def-distribution} in $\R^n_+$.
Let
\EQS{
\tilde{g}_j(x,y,t) &=s_j(x-y,t)-\epsilon_j s_j(x-y^*,t)\\
&= - \bkt{  \pd_j E(x-y)-\ep_j \pd_j  E(x-y^*)}\delta(t).
}
The pair $(\tilde{G}_{ij},\tilde g_j)$ satisfies in $\R^n$
\EQS{\label{0903-0}
(\pd_t-\De_x)\tilde{G}_{ij}(x,y,t)+\pd_{x_i}\tilde g_j(x,y,t)&=\de_{ij}\de_y(x)\de(t)  - \ep_j \de_{ij} \de_{y^*}(x) \de(t),
\\
\textstyle \sum_{i=1}^n \pd_{x_i}\tilde G_{ij} &=0.
}
Thus the difference
$(W_{ij}, w_j)= ({G}_{ij}, g_j) - (\tilde{G}_{ij},\tilde g_j)$ solves
 in $\R^n_+$
\EQ{\label{0904a}
\left\{\begin{array}{l}
(\pd_t-\De_x)W_{ij}(x,y,t)+\pd_{x_i}w_j(x,y,t)=0,\quad \sum_{i=1}^n \pd_{x_i}W_{ij}=0,
\\[2mm]
W_{ij}(x,y,t)|_{x_n=0}=-2\,\de_{in}S_{nj}(x'-y,t).
\end{array}\right.
}
By \eqref{eq84_Solo1968}, we have
\EQS{\label{0904e}
w_j(x,y,t)
&= -4\int_{\Si}\pd_n^2E(x-\xi')S_{nj}(\xi'-y,t)\,d\xi' -4\int_{\Si}E(x-\xi')\pd_tS_{nj}(\xi'-y,t)\,d\xi'\\
&\quad +8(\pd_t-\De_{x'})\int_{-\infty}^\infty\int_{\Si}\pd_nA(x-\xi',t-\tau)S_{nj}(\xi'-y,\tau)\,d\xi'd\tau\\
&= I_1+I_2+I_3.
}
Using $(\pd_t -\De) S_{ij} + \pd_i s_j = \de_{ij}\de(x) \de(t)$, we have
\EQS{\label{0905a}
I_2 =& -4\int_{\Si}E(x-\xi')[\De S_{nj}(\xi'-y,t) - \pd_n s_j(\xi'-y,t)]\, d\xi'\\
=& -4\int_{\Si} \De_{x'} E(x-\xi')S_{nj}(\xi'-y,t)\, d\xi' -4\int_{\Si} E(x-\xi') \pd_n^2 S_{nj}(\xi'-y,t)\, d\xi'\\
&~ -4 \de(t) \int_{\Si} E(x-\xi') \pd_n\pd_j E(\xi'-y)\, d\xi'
}
The first term of $I_2$ in \eqref{0905a} cancels $I_1$ since $\De E(x-\xi')=0$, and the last term of \eqref{0905a} is
$\overline w_j(x,y) \de(t)$
with
\EQN{
\overline w_j(x,y) %
&= \pd_{y_j} 4\int_{\Si}E(x-\xi')\pd_n  E(\xi'-y)\,d\xi'
 = \pd_{y_j} 2 \int_{\Si}E(\xi'-x)P_0(y-\xi')\,d\xi' \\
 &= 2 \pd_{y_j} E(x-y^*) = -2 \ep_j \pd_j E(x-y^*)
}
using \eqref{KMT2.32}. Note that
\EQ{\label{3.29}
\tilde g_j(x,y,t) + \overline w_j(x,y) \de(t)  = - F_j^y(x) \de(t).
}

Using $(\pd_t -\De) S_{ij} + \pd_i s_j = \de_{ij}\de(x) \de(t)$ again, we have
\EQN{
I_3 =&~ 8(\pd_t-\De_{x'})\int_{-\infty}^\infty\int_{\Si}\pd_nA(\xi',x_n,\tau)S_{nj}(x'-y'-\xi',-y_n,t-\tau)\,d\xi'd\tau\\
=&~ 8\int_{-\infty}^\infty\int_{\Si}\pd_nA(\xi',x_n,\tau)\bkt{\pd_n^2 S_{nj} - \pd_n s_j}(x'-y'-\xi',-y_n,t-\tau)\, d\xi'd\tau\\
=&~ 8\int_{-\infty}^\infty\int_{\Si}\pd_nA(\xi',x_n,\tau)\pd_n^2 S_{nj}(x'-y'-\xi',-y_n,t-\tau)\,d\xi'd\tau \\
& \qquad + 8\int_{\Si} \pd_nA(\xi',x_n,t) \pd_n\pd_jE(x'-y'-\xi',-y_n)\, d\xi'.
}

Denote
$\widehat w_j(x,y,t) = w_j(x,y,t) - \overline w_j(x,y) \de(t) $.
We conclude
\EQS{\label{0715a}
&\widehat w_j(x,y,t)
= 8\int_{-\infty}^\infty\int_{\Si}\pd_nA(\xi',x_n,\tau)\pd_n^2 S_{nj}(x'-y'-\xi',-y_n,t-\tau)\,d\xi'd\tau
\\
& -4\int_{\Si} E(x-\xi') \pd_n^2 S_{nj}(\xi'-y,t) + 8\int_{\Si} \pd_nA(\xi',x_n,t) \pd_n\pd_jE(x'-y'-\xi',-y_n)\, d\xi' .
}
Using $\pd_n^2 S_{nj} = -\sum_{i<n} \pd_i \pd_n S_{ij}$ and integrating by parts in $\xi_i$ the first two terms,
we get \eqref{w-formula} for $\widehat w_j(x,y,t)$.
Integration by parts is justified since the singularities of the integrands are outside of $\Si$, and the integrands have sufficient decay as $|\xi'| \to \infty$ by \eqref{eq_estSij} and \eqref{eq_estA}
even for $n=2$.
This and \eqref {3.29} prove the proposition.
\end{proof}

\begin{remark}\label{rem3.5}
(i) Eq.~\eqref{w-formula} is better than \eqref{0715a} because its estimate %
allows more decay in $|x-y^*|+\sqrt t$, i.e., in tangential direction. However, it has a boundary singularity at $x_n=0$; see Remark \ref{rem7.1}.

(ii)
With \thref{gj-decomp},
the pressure formula \eqref{eq2.12} in the case $u_0=0$ becomes
\EQS{\label{0906a}
\pi(x,t) &= %
\int_{-\infty}^\infty \int_{\R^n_+} g(x,y,t-s)\cdot f(y,s)\,dy\,ds
\\
&= %
\int_{0}^t \int_{\R^n_+} \widehat w(x,y,t-s)\cdot f(y,s)\,dy\,ds
- \int_{\R^n_+}F_j^y(x) \cdot f_j(y,t)\,dy.
}
The last term comes from the Helmholtz projection of $f$ at time $t$ (see \eqref{Helmholtz-p1}-\eqref{Fj.def}), and corresponds to  the pressure formula above \eqref{eq_def_Sij} in the whole space case.  The first term of \eqref{0906a} shows that $\pi(\cdot,t)$ also depends on the value of $f$ at times $s<t$. There is no such term in the whole space case. This history-dependence property of the pressure in the half space case is well known, see e.g.~\cite{MR896770}.
\end{remark}

\begin{remark}[Kernel of Green tensor]\label{Gij.kernel}
Consider
\[
\mathbf G = \bket{u= \nb h\in C^0( \R^n_+ ; \R^n), \, \lim_{|x|\to \infty}h(x)=0}.
\]
If $u_0 \in \mathbf G$, then $u(x,t)$ given by \eqref{ui.def} is identically zero, using integration by parts in \eqref{ui.def}. The whole thing vanishes because $\sum_j \pd_{y_j}G_{ij}=0$ and $G_{in}|_{y_n=0}=0$.
Thus $\mathbf G$ is contained in the kernel of the Green tensor.
In fact, it is also inside the kernel of the Helmholtz projection in $L^q(\R^n_+)$, $1<q<\infty$, if we impose suitable spatial decay on functions in $\mathbf G$.
\end{remark}

\begin{remark}[Relation between stationary and nonstationary Green tensors]\label{stationary-nonstationary}
Denote the Green tensor of the stationary Stokes system
in the half space as $G_{ij}^0(x,y)$. For $n \ge 3$ we can show
\EQ{
\int_\R G_{ij}(x,y,t)\,dt  = G_{ij}^0(x,y).
}
The integral does not converge for $n=2$.
The idea is to decompose $G_{ij}(x,y,t)=\tilde G_{ij}(x,y,t)+W_{ij}(x,y,t)$ and
show their time integrations converge to corresponding terms in \cite[(2.25)]{KMT18}.
This relation gives an alternative proof of symmetry $G_{ij}^0(x,y)=G_{ji}^0(y,x)$ for $n\ge3$ using \thref{prop1}.
\end{remark}

\section{Revised formula for the Green tensor}
\label{sec3}
In this section we derive a second formula for the remainder term  $W_{ij}$ which is suitable for pointwise estimate. We also use it to get a new formula for the Green tensor in Lemma \ref{Green-formula}.

We first recall the Poisson kernel $P(x,\xi',t)$ for $\pd_t-\De$ in the half-space $\R^n_+$
for $x\in\R^n_+$ and $\xi'\in \Si$,
\begin{equation}\label{eq_P=-2dnG}P(x,\xi',t)=-2\,\pd_n\Ga(x-\xi',t).\end{equation}
The following lemma is based on Poisson's formula, and can be used to \emph{remove the time integration} in the first formula \eqref{eq_Wij_formula0}. It is the time-dependent version of \eqref{KMT2.32}.

\begin{lem}\thlabel{Poisson_integral}
Let $n \ge 2$. For $x\in\mathbb{R}^n_+$, $y\in\R^n$ and $t>0$,
\EQ{\label{eq3.2}
\int_0^t\int_\Si\Gamma(\xi'-y,s)P(x,\xi',t-s)\,d\xi'\,ds=\Gamma(x-y^\sharp,t),\quad y^\sharp=(y',-|y_n|).}
Note that $y^\sharp=y^*$ if $\R^n_+$, and $y^\sharp=y$ if $y\in\R^n_-$.
\end{lem}
\begin{proof} First we consider $y\in\R^n_+$.
Since $u(x,t)=\Gamma(x-y^*,t)$ satisfies \[\left\{\begin{array}{ll}(\pd_t-\Delta)u(x,t)=0&\  \text{ for }(x,t)\in\mathbb{R}^n_+\times(0,\infty),\\[1mm]
u(x',t)=\Gamma(x'-y^*,t)=\Gamma(x'-y,t) &\  \text{ for }(x',t)\in\pd\mathbb{R}^n_+\times(0,\infty),\\[1mm]
u(x,0)=\Gamma(x-y^*,0)=\delta(x-y^*)=0&\  \text{ for }x\in\mathbb{R}^n_+,\end{array}\right.\] by Poisson's formula for $\pd_t-\Delta$ in $\mathbb{R}^n_+$, we have
\[
\int_0^t\int_\Si\Gamma(\xi'-y,s)P(x,\xi',t-s)\,d\xi'\,ds=\Gamma(x-y^*,t) .
\]
For $y\in\R^n_-$, $y^*\in\R^n_+$. Since $\Ga(\xi'-y,s)=\Ga(\xi'-y^*,s)$,
\EQN{
\int_0^t\int_\Si\Gamma(\xi'-y,s)P(x,\xi',t-s)\,d\xi'\,ds&=\int_0^t\int_\Si\Gamma(\xi'-y^*,s)P(x,\xi',t-s)\,d\xi'\,ds\\
&=\Gamma(x-y^{**},t)=\Gamma(x-y,t).
}
The combination of the two cases $y\in\R^n_+$ and $y\in\R^n_-$ gives \eqref{eq3.2}.
\end{proof}

\medskip

With \thref{Poisson_integral} in hand, we are able to derive the second formula for $W_{ij}$.
\begin{lem}[The second formula for $W_{ij}$]\thlabel{lem_Wij_formula} For $x,y\in\R^n_+$ and $i,j=1,\ldots,n$,
\begin{equation}\label{eq_Wij_formula}
\begin{aligned}W_{ij}(x,y,t)&=-2\de_{in}\de_{nj}\Ga(x-y^*,t)+2\de_{in}\ep_j\Ga_{nj}(x-y^*,t)-4\de_{nj}C_i(x,y,t)-4H_{ij}(x,y,t) + V_{ij}(x,y,t),\end{aligned}\end{equation}
where
\begin{equation}\label{eq_def_Ci}
C_i(x,y,t)=\int_0^{x_n}\int_\Si\pd_n\Ga(x-y^*-z,t)\,\pd_iE(z)\,dz'\,dz_n,\end{equation}
\begin{equation}\label{eq_def_Hij}H_{ij}(x,y,t)=
-\int_{\mathbb{R}^n}\pd_{y_j}C_i(x,y+w,t)\pd_nE(w)dw,
\end{equation}
and
\EQ{\label{eq-Vij-def}
V_{ij}(x,y,t) = -2\de_{in} \La_j(x,y,t) -4\int_0^{x_n} \int_{\Si} \pd_{x_n}\La_j(x-z,y,t) \pd_iE(z)\, dz'dz_n.
}
Here
\EQ{\label{eq-La-def}
\La_j(x,y,t) = \pd_{y_n}\pd_{y_j} \int_{w_n<-y_n}G^{ht}(x,y+w,t) E(w)\, dw,
}
where $G^{ht}(x,y,t) = \Ga(x-y,t) - \Ga(x-y^*,t)$ is the Green function of heat equation in $\R^n_+\times(0,\infty)$.
Note that $C_i(x,y,t)$ is defined in $\R^n_+ \times \R^n \times (0,\infty)$, and $y_n $ is allowed to be negative.
\end{lem}

\medskip

\begin{remark}\thlabel{rem_3.1}
We can show that $C_i$, $H_{ij}$, and $V_{ij}$ are well defined using Lemma \ref{lemma2.2}. The $x'$- and $y'$-derivatives are interchangeable for $C_i$, $H_{ij}$, and $V_{ij}$: $\pd_{x'}^lC_i(x,y,t)=(-1)^l\pd_{y'}^lC_i(x,y,t)$ and   similarly for $H_{ij}$, and $V_{ij}$.

\end{remark}

\medskip

\begin{remark}
The formula \eqref{eq_Wij_formula} is better than \eqref{eq_Wij_formula0} because the definitions of the terms on the right side do not involve integration in time.
If an integration in time was involved,
there might be singularities at $s=0,t$ when we use the estimates of $K_{ij}$ and $S_{ij}$ in \eqref{eq_estKij} and \eqref{eq_estSij}, respectively. Their estimates would be worse and contain, for example, singularities in $x_n$ for $x_n$ small. The quantity $C_i(x,t)$ studied by Solonnikov \cite[(66)]{MR0171094} corresponds to our $C_i(x,0,t)$ with $y=0$ and he did not study full $C_i(x,y,t)$ with $y\not=0$ nor $H_{ij}(x,y,t)$.
\end{remark}

\medskip

\begin{remark}
The formula \eqref{eq_Wij_formula} corresponds to that of the stationary case in \cite[(2.36)]{KMT18}:
\[W_{ij}(x,y)=-\left(\de_{in}-x_n\pd_{x_i}\right)\left(\de_{nj}-y_n\pd_{y_j}\right)E(x-y^*).\]
\end{remark}

\medskip

\begin{proof}[Proof of \thref{lem_Wij_formula}]
 To obtain \eqref{eq_Wij_formula}, we use the formulae \eqref{eq_def_Sij} and \eqref{eq_def_Kij} and split the integral of \eqref{eq_Wij_formula0} into six parts as \[ \int_{-\infty}^{\infty}\int_\Si K_{in}(x-\xi',t-s)S_{nj}(\xi'-y,s)\,d\xi'\,ds=I_1+I_2+I_3+I_4+I_5+I_6,\]
where
\begin{align*}
I_1&=-2\,\de_{in}\de_{nj} \int_{-\infty}^{\infty}\int_\Si \pd_n\Ga(x-\xi',t-s)\Ga(\xi'-y,s)\,d\xi'ds,\\
I_2&=-4\,\de_{nj} \int_{-\infty}^{\infty}\int_\Si\pd_{x_n}\left[\int_0^{x_n}\!\int_\Si\pd_n\Ga(z,t-s)\,\pd_iE(x-\xi'-z)\,dz'\,dz_n\right]\Ga(\xi'-y,s)\,d\xi'ds,\\
I_3&=-2\,\de_{nj} \int_{-\infty}^{\infty}\int_\Si \pd_iE(x-\xi')\de(t-s)\Ga(\xi'-y,s)\,d\xi'ds,\\
I_4&=-2\,\de_{in} \int_{-\infty}^{\infty}\int_\Si \pd_n\Ga(x-\xi',t-s)\int_{\R^n}\pd_n\pd_j\Ga(\xi'-y-w,s)E(w)\,dwd\xi'ds,\\
I_5&=-4 \int_{-\infty}^{\infty}\int_\Si\pd_{x_n}\left[\int_0^{x_n}\int_\Si\pd_n\Ga(z,t-s)\,\pd_iE(x-\xi'-z)\,dz'\,dz_n\right] \\
&\hspace{70mm}\cdot\left[\int_{\R^n}\pd_n\pd_j\Ga(\xi'-y-w,s)E(w)\,dw\right]d\xi'ds,\\
I_6&=-2 \int_{-\infty}^{\infty}\int_\Si\pd_iE(x-\xi')\de(t-s)\int_{\R^n}\pd_n\pd_j\Ga(\xi'-y-w,s)E(w)\,dw\,d\xi'ds.
\end{align*}

We use \thref{Poisson_integral} to compute $I_1,I_2,I_4,I_5$. Indeed, we have
\[\begin{aligned}I_1=&-2\de_{in}\de_{nj}\int_0^t\int_\Si \pd_n\Ga(x-\xi',t-s)\Ga(\xi'-y,s)\,d\xi'ds\\=&~\de_{in}\de_{nj}\int_0^t\int_\Si P(x,\xi',t-s)\Ga(\xi'-y,s)\,d\xi'ds\\=&~\de_{in}\de_{nj}\Ga(x-y^\sharp,t) = \de_{in}\de_{nj} \Ga(x-y^*,t),\end{aligned}\]
where we used \eqref{eq_P=-2dnG}, \thref{Poisson_integral} and $y\in\R^n_+$. And, by changing the variables and Fubini's theorem, we have
\[\begin{aligned}%
I_2=&-4\,\de_{nj}\int_0^t\int_\Si\pd_{x_n}\left[\int_0^{x_n}\int_\Si\pd_n\Ga(x-\xi'-z,t-s)\,\pd_iE(z)\,dz'\,dz_n\right]\Ga(\xi'-y,s)\,d\xi'ds\\
=&-4\,\de_{nj}\pd_{x_n}\left[\int_0^{x_n}\int_\Si\int_0^t\left(\int_\Si\pd_n\Ga(x-\xi'-z,t-s)\Ga(\xi'-y,s)\,d\xi'ds\right)\pd_iE(z)\,dz'dz_n\right].\end{aligned}\]
With the aid of \eqref{eq_P=-2dnG} and \thref{Poisson_integral}, we actually get
\[\begin{aligned}
I_2=&~2\,\de_{nj}\pd_{x_n}\left[\int_0^{x_n}\int_\Si\left(\int_0^t\int_\Si P(x-z,\xi',t-s)\Ga(\xi'-y,s)\,d\xi'ds\right)\pd_iE(z)\,dz'dz_n\right]\\
=&~2\,\de_{nj}\pd_{x_n}\left[\int_0^{x_n}\int_\Si\Ga(x-z-y^\sharp,t)\,\pd_iE(z)\,dz'dz_n\right]\\
=&~ 2\,\de_{nj}\pd_{x_n}\left[\int_0^{x_n}\int_\Si\Ga(x-z-y^*,t)\,\pd_iE(z)\,dz'dz_n\right]\quad \text{(since $y\in\R^n_+$)} \\
=&~2\,\de_{nj}\int_\Si\Ga(x'-z'-y^*,t)\,\pd_iE(z',x_n)\,dz'
+2\,\de_{nj} C_i(x,y,t),
\end{aligned}\]
where $C_i(x,y,t)$ is as defined in \eqref{eq_def_Ci}.

Moreover, rearranging the integrals and derivatives and using \eqref{eq_P=-2dnG}, we obtain
\[\begin{aligned}I_4=&-2\,\de_{in}\int_0^t\int_\Si \pd_n\Ga(x-\xi',t-s)\int_{\R^n}\pd_n\pd_j\Ga(\xi'-y-w,s)E(w)\,dwd\xi'ds\\
=&-2\,\de_{in}\pd_{y_n}\pd_{y_j}\left[\int_0^t\int_\Si \pd_n\Ga(x-\xi',t-s)\int_{\R^n}\Ga(\xi'-y-w,s)E(w)\,dwd\xi'ds\right]\\
=&~\de_{in}\pd_{y_n}\pd_{y_j}\left[\int_{\R^n}\left(\int_0^t\int_\Si P(x,\xi',t-s)\Ga(\xi'-(y+w),s)d\xi'\,ds\right)E(w)\,dw\right].\end{aligned}\]
Hence, applying Fubini's theorem and \thref{Poisson_integral}, we have
\[\begin{aligned}I_4=&~\de_{in}\pd_{y_n}\pd_{y_j}\left[\int_{\R^n}\Ga(x-(y+w)^\sharp,t)E(w)\,dw\right]\\
=&~\de_{in}\pd_{y_n}\pd_{y_j}\left[\int_{w_n>-y_n}\Ga(x-(y+w)^*,t)E(w)\,dw + \int_{w_n<-y_n}\Ga(x-y-w,t)E(w)\,dw\right]\\
=&~\de_{in}\pd_{y_n}\pd_{y_j}\left[\int_{\R^n}\Ga(x-(y+w)^*,t)E(w)\,dw\right.\\
&\qquad\qquad\qquad\qquad\qquad\qquad\qquad\quad \left.+ \int_{w_n<-y_n}(\Ga(x-y-w,t) - \Ga(x-(y+w)^*,t)) E(w)\,dw\right]\\
=&~-\de_{in}\ep_j \Ga_{nj}(x-y^*,t) + \de_{in}\pd_{y_n}\pd_{y_j} \int_{w_n<-y_n} G^{ht}(x,y+w,t) E(w)\,dw\\
=&~-\de_{in}\ep_j \Ga_{nj}(x-y^*,t) + \de_{in} \La_j(x,y,t),\end{aligned}\]
where $G^{ht}(x,y,t)=\Ga(x-y,t) - \Ga(x-y^*,t)$ is the Green function of heat equation in $\R^n_+\times(0,\infty)$ and $\La_j(x,y,t)$ is as defined in \eqref{eq-La-def}.

In addition, by changing the variables, Fubini's theorem and \eqref{eq_P=-2dnG}, we get
\[\begin{aligned}I_5
=&-4\int_0^t\int_\Si\pd_{x_n}\left[\int_0^{x_n}\int_\Si\pd_n\Ga(x-\xi'-z,t-s)\,\pd_iE(z)\,dz'dz_n\right]\\
&~~~~~~~~~~~~~~~~~~~~~~~~~~~~~~~~~~~~~~~~~~~~~~~~~~~~~~~~~\cdot\left[\int_{\R^n}\pd_n\pd_j\Ga(\xi'-y-w,s)E(w)\,dw\right]d\xi'ds\\
=&~2\,\pd_{x_n}\left[\int_0^{x_n}\int_\Si\int_{\R^n}\pd_{y_n}\pd_{y_j}\left(\int_0^t\int_\Si P(x-z,\xi',t-s)\Ga(\xi'-(y+w),s)\,d\xi'ds\right)\right.\\&~~~~~~~~~~~~~~~~~~~~~~~~~~~~~~~~~~~~~~~~~~~~~~~~~~~~~~~~~~~~~~~~~~~~~~~~~~~~~~~~\cdot\pd_iE(z)E(w)\,dwdz'dz_n\Big].\end{aligned}\]
Thus, \thref{Poisson_integral} implies
\EQN{ %
I_5=&~2\,\pd_{x_n}\left[\int_0^{x_n}\int_\Si\int_{\R^n}\pd_{y_n}\pd_{y_j}\Ga((x-z)-(y+w)^\sharp,t)\,\pd_iE(z)E(w)\,dwdz'dz_n\right]\\
=&~ 2 \int_\Si\int_{\R^n}\pd_{y_n}\pd_{y_j}\Ga(x'-z'-(y+w)^\sharp,t)\,\pd_iE(z',x_n)E(w)\,dwdz'\\
& \quad +  2\,\int_0^{x_n}\int_\Si \pd_{x_n} \left(\int_{\R^n}\pd_{y_n}\pd_{y_j}\Ga((x-z)-(y-w)^\sharp,t)E(w)\,dw\right)\pd_iE(z)\,dz'dz_n\\
=&~ 2\int_\Si\Ga_{nj}(x'-z'-y,t)\,\pd_iE(z',x_n)\,dz' + 2H_{ij}^\sharp(x,y,t),
} %
where we've used $\Ga(x'-z'-(y+w)^\sharp,t)=\Ga((x'-z'-y)-w,t)$, the functions $\Ga_{ij}$ is defined in \eqref{eq_def_Sij}, and $H_{ij}^\sharp$ is expanded as
\EQN{
H_{ij}^\sharp(x,y,t) &= \int_0^{x_n}\int_{\Si} \pd_{x_n} \pd_{y_n} \pd_{y_j} \left( \int_{w_n>-y_n} \Ga((x-z)-(y+w)^*,t) E(w)\, dw\right.\\
&\qquad\qquad\qquad\qquad\qquad\left. +  \int_{w_n<-y_n} \Ga((x-z)-y-w,t) E(w)\, dw \right) \pd_iE(z)\, dz'dz_n\\
&= \int_0^{x_n}\int_{\Si} \pd_{x_n} \pd_{y_n} \pd_{y_j} \left( \int_{\R^n} \Ga((x^*-z^*-y)-w,t) E(w)\, dw\right.\\
&\qquad\qquad\qquad\qquad\qquad\left. + \int_{w_n<-y_n} G^{ht}(x-z,y+w,t)E(w)\, dw\right) \pd_iE(z)\, dz'dz_n\\
&= H_{ij}(x,y,t) + \int_0^{x_n} \int_{\Si} \pd_{x_n}\La_j(x-z,y,t) \pd_iE(z)\, dz'dz_n,
}
where $H_{ij}$ is defined in \eqref{eq_def_Hij}.

For $I_3$ and $I_6$, a direct computation gives
\[\begin{aligned}I_3
=&-2\,\de_{nj}\int_\Si \pd_iE(x-\xi')\Ga(\xi'-y,t)\,d\xi'\\
=&-2\,\de_{nj}\int_\Si\Ga(x'-z'-y^*,t)\,\pd_iE(z',x_n)\,dz',\end{aligned}\]
and
\[\begin{aligned}I_6
=&-2\int_\Si\pd_iE(x-\xi')\int_{\R^n}\pd_n\pd_j\Ga(\xi'-y-w,t)E(w)\,dwd\xi'\\
=&-2\int_\Si\pd_iE(x-\xi')\Ga_{nj}(\xi'-y,t)\,d\xi'\\=&-2\int_\Si\Ga_{nj}(x'-z'-y,t)\,\pd_iE(z',x_n)\,dz'.
\end{aligned}\]
Combining the above computations of $I_1,\cdots,I_6$, and noting that $I_3$ cancels the first term of $I_2$ while $I_6$ cancels the first term of $I_5$, we get
\EQN{
\sum_{k=1}^6 I_k
=\de_{in}\de_{nj}\Ga(x-y^*,t)+2\,\de_{nj}C_i(x,y,t)-\de_{in}\ep_j\Ga_{nj}(x-y^*,t)  + \de_{in} \La_j(x,y,t) + 2H_{ij}^\sharp(x,y,t).
}
This completes the proof.
\end{proof}

We now  explore a cancellation between $C_i$ and $H_{ij}$ in \eqref{eq_Wij_formula}, and define
\EQS{\label{Hhat_def}
\widehat  H_{ij}(x,y,t)  = H_{ij}(x,y,t) + \de_{nj} C_i (x,y,t).
}
Then \eqref{eq_Wij_formula} becomes
\EQS{\label{new_W}
W_{ij}(x,y,t) = -2\de_{in} \de_{nj} \Ga(x-y^*,t) + 2 \de_{in} \ep_j \Ga_{nj}(x-y^*,t) - 4\widehat H_{ij}(x,y,t) + V_{ij}(x,y,t).
}
This formula will provide better estimates than summing estimates of individual terms in \eqref{eq_Wij_formula}. See Remark \ref{D_estimate-rmk} after \thref{D_estimate}.

We conclude a second formula for the Green tensor.
\begin{lem}\label{Green-formula} The Green tensor satisfies
\EQS{
G_{ij}(x,y,t) &= \de_{ij} \bkt{ \Ga(x-y,t) - \Ga(x-y^*,t) } + \bkt{ \Ga_{ij}(x-y,t) - \ep_i \ep_j \Ga_{ij}(x-y^*,t) }\\
&\quad - 4\widehat H_{ij}(x,y,t) + V_{ij}(x,y,t).
}
\end{lem}
\begin{proof}
Recall \eqref{eq_def_Gij} that $G_{ij}(x,y,t)=S_{ij}(x-y,t)-\ep_jS_{ij}(x-y^*,t)+ W_{ij}(x,y,t)$. By \eqref{eq_def_Sij} and \eqref{new_W}, we get the lemma.
\end{proof}

\section{First estimates of the Green tensor}\label{S5}
In this section, we first estimate $\widehat H_{ij}$, 
then estimate $V_{ij}$, and finally
prove the Green tensor estimates in \thref{prop2}.

\subsection{Estimates of $\widehat H_{ij}$}

\begin{lem}\thlabel{cancel_C_H}
For $i,j=1,\ldots,n$, we have
\EQS{\label{Hhat_D}
\left \{
\begin{split}
\widehat H_{ij}(x,y,t) &= - D_{ijn}(x,y,t) \quad\text{ if } j < n,   \quad\\
\widehat  H_{in}(x,y,t) &= \textstyle \sum_{\be<n} D_{i\be\be}(x,y,t)  \quad\text{ if }j = n,
\end{split}
\right .
}
where for $m = 1,\ldots, n$,
\EQ{\label{D_def}
D_{i\be m}(x,y,t) = \int_0^{x_n}\int_{\Si}\pd_\be\Ga_{m n}(x^*-y-z^*,t)\,\pd_iE(z)\,dz'dz_n.
}

\end{lem}

\begin{proof}
By definition,
\EQN{
H_{ij}(x,y,t)&=-\int_{\R^n}\pd_{y_j}C_i(x,y+w,t)\pd_nE(w)\,dw\\
&=-\int_{\R^n}\pd_{y_j}\bke{\int_0^{x_n}\!\int_{\Si}\pd_n\Ga(x-(y+w)^*-z,t)\pd_iE(z)\,dz'dz_n}\pd_nE(w)\,dw.
}
Integrating by parts in $w_n$ and applying Fubini's theorem give, for $j=1,\ldots,n$,
\EQN{
H_{ij}(x,y,t)
=&\int_0^{x_n}\int_{\Si}\pd_{y_j}\int_{\R^n}\pd_n^2\Ga((x^*-y-z^*)-w,t)E(w)\,dw\,\pd_iE(z)\,dz'dz_n\\
=&-\int_0^{x_n}\int_{\Si}\pd_j\Ga_{nn}(x^*-y-z^*,t)\pd_iE(z)\,dz'dz_n=-D_{ijn}(x,y,t).
}
This proves \eqref{Hhat_D} when $j<n$.
For $j=n$, we use the fact that $-\De E=\de$ to obtain
\EQN{
H_{in}(x,y,t)=&-\int_{\mathbb{R}^n}\pd_{y_n}C_i(x,y+w,t)\pd_nE(w)dw\\
=& - C_i(x,y,t)+\sum_{\beta=1}^{n-1}\int_{\mathbb{R}^n}\pd_{y_\beta}C_i(x,y+w,t)\pd_\beta E(w)\,dw.
}
Using the same argument above, we get
\EQN{
H_{in}(x,y,t)= - C_i(x,y,t)+\sum_{\beta=1}^{n-1}D_{i\be\be}(x,y,t).
}
This proves \eqref{Hhat_D} when $j=n$.
\end{proof}

The following lemma enables us to change $x_n$-derivatives to $x'$-derivatives.
\begin{lem}
Let $i, j , m=1,\ldots,n$. For $i<n$,
\EQ{\label{eq_pdxnDi}
\pd_{x_n}D_{ijm}(x,y,t)=\pd_{x_i}D_{njm}(x,y,t) + \int_{\R^n}\pd_i\pd_j\pd_mB(x^*-y-w,t)\pd_nE(w)\,dw,
}
and for $i=n$,
\EQ{\label{eq_pdxnDn}
\pd_{x_n}D_{njm}(x,y,t)=
-\sum_{\be=1}^{n-1}\pd_{x_\be}D_{\be jm}(x,y,t)-\frac12\,\pd_n\Ga_{jm}(x^*-y,t).
}
\end{lem}
\begin{proof}
After changing variables, $D_{ijm}$ becomes
\[D_{ijm}(x,y,t)=\int_{-x_n-y_n}^{-y_n}\int_\Si\pd_j\Ga_{mn}(z,t)\, \pd_iE(x-y^*-z^*)\,dz'dz_n.\]
For $i<n$ we have
\[D_{ijm}(x,y,t)=\pd_{x_i}\int_{-x_n-y_n}^{-y_n}\int_\Si\pd_j\Ga_{mn}(z,t)\, E(x-y^*-z^*)\,dz'dz_n.\]
Hence
\[\begin{aligned}\pd_{x_n}D_{ijm}(x,y,t)=&~\pd_{x_i}\int_\Si\pd_j\Ga_{mn}(z',-x_n-y_n,t) E(x'-y'-z',0)\,dz'\\
&+\pd_{x_i}\int_{-x_n-y_n}^{-y_n}\int_\Si\pd_j\Ga_{mn}(z,t) \pd_nE(x-y^*-z^*)\,dz'dz_n\\
=&~I + \pd_{x_i}D_{njm}(x,y,t),
\end{aligned}
\]
where
\[
I=\pd_{x_i}\int_\Si\bke{\int_{\R^n}\pd_j\pd_m\Ga(z'-w',-x_n-y_n-w_n,t)\pd_nE(w)\,dw} E(x'-y'-z',0)\,dz'.
\]
After changing variables $\xi'=x'-y'-z'$ and applying Fubini theorem,
\EQN{
I &=\int_{\R^n} \bke{ \pd_{x_i}\int_\Si \pd_j\pd_m\Ga(x'-y'-\xi'-w',-x_n-y_n-w_n,t) E(\xi',0)\,d\xi'} \pd_nE(w)\,dw\\
&=\int_{\R^n}\pd_{x_i}\pd_{y_j}\pd_{y_m}B(x^*-y-w,t)\pd_nE(w)\,dw\\
&=\int_{\R^n}\pd_i\pd_j\pd_mB(x^*-y-w,t)\pd_nE(w)\,dw
}
using $i<n$ again.
This proves \eqref{eq_pdxnDi}.

For \eqref{eq_pdxnDn}, we first move normal derivatives in the definition \eqref{D_def} of $D_{njm}$ to tangential derivatives. Observe that, using $\pd_j \Ga_{mn} = \pd_n \Ga_{jm}$,
\[\begin{aligned}
D_{njm}(x,y,t)
=&~\lim_{\varepsilon\to0_+}\left[\int_{\varepsilon}^{x_n}\int_\Si\pd_{z_n}\Ga_{jm}(x^*-y-z^*,t)\pd_nE(z)\,dz'dz_n\right]\\
=&~\lim_{\varepsilon\to0_+}
\left[\int_\Si\Ga_{jm}(x'-y'-z',-y_n,t)\pd_nE(z',x_n)\,dz'\right.\\
&~~~~~~~~-\int_\Si\Ga_{jm}(x'-y'-z',-x_n-y_n+\varepsilon,t)\pd_nE(z',\varepsilon)\,dz'\\
&~~~~~~~~\left.-\int_{\varepsilon}^{x_n}\int_\Si\Ga_{jm}(x^*-y-z^*,t)\pd_n^2E(z)\,dz'dz_n\right],
\end{aligned}\]
by integration by parts in the $z_n$-variable. Using the fact that $-\De E=\de$, we obtain
\[\begin{aligned}
D_{njm}(x,y,t)=&~\pd_{y_j}\pd_{y_m}\int_{\R^n}e^{\frac{-(y_n+w_n)^2}{4t}}\,\pd_nA(x'-y'-w',x_n,t) E(w)\,dw\\
&-\pd_{y_j}\pd_{y_m}\int_{\R^n}e^{-\frac{(x_n+y_n+w_n)^2}{4t}}\,\pd_nA(x'-y'-w',0_+,t) E(w)\,dw
+J,
\end{aligned}\]
where
\EQN{
J&=\sum_{\beta=1}^{n-1}\lim_{\varepsilon\to0_+}\int_{\varepsilon}^{x_n}\int_\Si\Ga_{mj}(x^*-y-z^*,t)\pd_\beta^2E(z)\,dz'dz_n
\\
&=  \sum_{\beta=1}^{n-1}\int_0^{x_n}\int_\Si\pd_\be\Ga_{mj}(x^*-y-z^*,t)\pd_\be E(z)\,dz'dz_n,
}
by integration by parts in the $z'$-variable.
Note that
\[
\pd_n A(x',0_+,t) = \lim_{\e \to 0_+} \int_\Si \Ga(x'-z',0,t)\pd_n E(z',\e)dz' = - \frac 12\, \Ga(x',0,t)
\]
since $-2\pd_nE(x)$ is the Poisson kernel for the Laplace equation in $\R^n_+$.
Using $e^{-\frac{(x_n+y_n+w_n)^2}{4t}}\Ga(x'-y'-w',0,t)=\Ga(x^*-y-w,t)$, we get
\begin{align}
D_{njm}(x,y,t)=&~\pd_{y_j}\pd_{y_m}\int_{\R^n}e^{\frac{-(y_n+w_n)^2}{4t}}\,\pd_nA(x'-y'-w',x_n,t) E(w)\,dw+\frac12\,\Ga_{mj}(x^*-y,t) \nonumber \\
&+\sum_{\beta=1}^{n-1}\int_0^{x_n}\int_\Si\pd_\be\Ga_{mj}(x^*-y-z^*,t)\pd_\be E(z)\,dz'dz_n.
\end{align}
In this form we have moved normal derivatives in the definition \eqref{D_def} of $D_{njm}$ to tangential derivatives.
Consequently,
\[\begin{aligned}
\pd_{x_n}&D_{njm}(x,y,t)\\
=&~\pd_{y_j}\pd_{y_m}\int_{\R^n}e^{\frac{-(y_n+w_n)^2}{4t}}\,\pd_n^2A(x'-y'-w',x_n,t) E(w)\,dw - \frac12\,\pd_n\Ga_{mj}(x^*-y,t) \\
&+\sum_{\beta=1}^{n-1}\int_\Si\pd_\be\Ga_{mj}(x'-y'-z',-y_n,t)\pd_\be E(z',x_n)\,dz'\\
&-\sum_{\beta=1}^{n-1}\int_0^{x_n}\int_\Si\pd_n\pd_\be\Ga_{mj}(x^*-y-z^*,t)\pd_\be E(z)\,dz'dz_n\\
=&~\pd_{y_j}\pd_{y_m}\int_{\R^n}e^{\frac{-(y_n+w_n)^2}{4t}}\,\pd_n^2A(x'-y'-w',x_n,t) E(w)\,dw - \frac12\,\pd_n\Ga_{mj}(x^*-y,t)\\
&+\sum_{\beta=1}^{n-1}\pd_{y_j}\pd_{y_m}\int_{\R^n}e^{-\frac{(y_n+w_n)^2}{4t}}\,\pd_\be^2A(x'-y'-w',x_n,t) E(w)\,dw
-\sum_{\beta=1}^{n-1}\pd_{x_\be} D_{\be jm}(x,y,t).
\end{aligned}\]
The first term cancels the third term since  $\De_x A(x,t)=0$ for $x_n>0$.
This proves \eqref{eq_pdxnDn}.
\end{proof}

\begin{remark} \label{rk-pdxnDn}
Note that Lemma \ref{cancel_C_H} and \eqref{eq_pdxnDn} imply
\[\textstyle \sum_{i=1}^n \pd_{x_i}\widehat H_{ij}(x,y,t) = \frac12\, \ep_j\pd_n\Ga_{nj}(x-y^*,t) - \frac12\,\de_{nj} \pd_n\Ga(x-y^*,t),\]
which is equivalent to $\sum_{i=1}^n \pd_{x_i}G_{ij}(x,y,t)=0$ using  Lemma \ref{Green-formula}. Since we will use \eqref{eq_pdxnDn} to prove \eqref{Green_est}, the property $\sum_{i=1}^n \pd_{x_i}G_{ij}(x,y,t)=0$ cannot be used to improve \eqref{Green_est}. However, we will use it to prove
\eqref{eq_Green_estimate}.
\end{remark}

The following lemma will be used in the $x_n$-derivative estimate of \thref{D_estimate}.

\begin{lem}\thlabel{lem_Bint_est}
For $B(x,t)$ defined by \eqref{eq_def_B}, for $l,k \in \NN_0$,
\EQS{\label{Bint_est}
\abs{\int_{\R^n}\pd_{x'}^{l+1}\pd_{x_n}^{k}B(x-w,1)\pd_n E(w)\,dw}
\lec \frac {1+\de_{n2} \log \bka{\de_{k0}|x'|+ |x_n|}}{\bka{x}^{l+n-1} \bka{x_n}^k}.
}
Note in \eqref{Bint_est} $\de_{k0}|x'|=0$ for $k>0$.
\end{lem}

Recall that $\pd_{x'}^{l}\pd_{x_n}^{k}B$ satisfies \eqref{eq_estB}-\eqref{eq_estB2} if $l+n\ge 3$, which is invalid if $l=0$ and $n=2$.
\begin{proof}
We will prove by induction in $k$. First consider $k=0$ and full $\pd E$ instead of just $\pd_n E$. Change variables and denote
$
J=\int_{\R^n}\pd_{w'}^{l+1} B(w,1)\pd E(x-w)\,dw$.
By \eqref{eq_estB},
\[
|J|
\lec
\int_{\R^n}
\frac{dw}{\bka{w}^{l+n-1}\bka{w_n} |x-w|^{n-1}},
\]
which is bounded for all $x$. We now assume $|x|>10$ to show its decay.
Decompose $\R^n$ to 4 regions: $\textup{I}=\{w:|w'|>2|x|\}$,
$\textup{II}=\{w:|w'|<2|x| ,\, |w_n|>|x|/2\}$,
$\textup{III}=\{w:|x|/2<|w'|<2|x|,\,  |w_n|<|x|/2\}$,
 and $\textup{IV}=\{w:|w'|<|x|/2,\,  |w_n|<|x|/2\}$. Decompose
\[
J= \bke{\int_{\textup{I}} + \int_{\textup{II}} +  \int_{\textup{III}} + \int_{\textup{IV}} }
(\pd_{w'}^{l+1}B)(w,1)\pd E(x-w) \,dw = J_1 + J_2+J_3+J_4.
\]

Using \eqref{eq_estB2},
\[
|J_1| \lec \int_{\textup{I}}\frac{e^{-w_n^2/10}} {|w|^{l+n-1}\, |x-w|^{n-1}}\,dw
\lec  \int_{|w'|>2|x|}\frac{e^{-w_n^2/10}} {|w'|^{l+n-1}\, |w'|^{n-1}}\,dw = \frac C {|x|^{l+n-1}}.
\]

Also by \eqref{eq_estB2}, and with $z'=x'-w'$,
\EQN{
|J_2| &\lec \int_{\textup{II}}\frac{e^{-w_n^2/10}} {|x|^{l+n-1}\, |x-w|^{n-1}}\,dw
\lec \frac 1{|x|^{l+n-1}} \int _{|w_n|\ge |x|/2} \int _{|z'|<3|x|} \frac{e^{-w_n^2/10}} {(|x_n-w_n|+|z'|)^{n-1}}\,dz'dw_n
\\
&=  \frac 1{|x|^{l+n-1}} \int _{|w_n|\ge |x|/2} \int _0^{3|x|} \frac{r^{n-2}} {(|x_n-w_n|+r)^{n-1}}\,dr\, e^{-w_n^2/10}\,dw_n.
}
By \thref{lem6-1}, the inner integral is bounded by $1+\log_+\frac{3|x|}{|x_n-w_n|}$.
\[
|J_2| \lec  \frac 1{|x|^{l+n-1}} \int _{|w_n|\ge |x|/2} \bke{1+\log_+\frac{3|x|}{|x_n-w_n|}} \, e^{-w_n^2/10}\,dw_n
\lec \frac 1 {|x|^{l+n-1}}.
\]

For $J_3$, if we have
$\pd_nE(x-w)\sim \frac{x_n-w_n}{|x-w|^n}$ in the integrand,
using \eqref{eq_estB2} and \thref{lem6-1},
\EQN{
|J_{3}|\lesssim&~\int_{\textup{III} }\frac{e^{-w_n^2/10}}{|x|^{l+n-1}}\frac{|x_n-w_n|}{(|x'-w'|+|x_n-w_n|)^n}\,dw\\
\lesssim&~\frac1{|x|^{l+n-1}}\int_{\R}|x_n-w_n|e^{-w_n^2/10}\int_0^{3|x|}\frac{r^{n-2}}{(|x_n-w_n|+r)^n}\,drdw_n\\
\lesssim&~\frac1{|x|^{l+n-1}}\int_{\R}|x_n-w_n|e^{-w_n^2/10}\,\frac{\min(|x_n-w_n|,3|x|)^{n-1}}{|x_n-w_n|^n}\,dw_n
\lesssim\frac1{|x|^{l+n-1}}.
}
If we have $\pd_\be E(x-w)$ with $\be<n$ in $J_3$, and if $n\ge 3$, we integrate $J_3$ by parts in $w_\be$,
\EQN{
J_{3}&=\int_{\textup{III}}\pd_{w'}^{l+2}B(w,1)E(x-w)\,dw + \int_\Ga
  \pd_{w'}^{l+1}B(w,1)E(x-w)\,dS_w,
}
where $\Ga=\bket{(w',w_n)\mid {|w'|=|x|/2 \text{ or } |w'|=2|x|,\, |w_n|<|x|/2}}$ is the lateral boundary of III.
Now using \eqref{eq_estB2} and that
$|x-w|> c|x|$ on $\Ga$,
\EQN{
|J_3| &\le
\int_{\textup{III}}\frac{e^{-w_n^2/10}}{|x|^{l+n}}\frac1{|x-w|^{n-2}}\,dw
+ \int_{\Ga}\frac{e^{-w_n^2/10}}{|x|^{l+n-1}}\frac1{|x-w|^{n-2}}dS_w
\\
&\lec \int_{|w_n|<|x|/2}  \frac{e^{-w_n^2/10}}{|x|^{l+n}}  \bke{\int_{|z'|<3|x|} \frac {dz'}{|z'|^{n-2}}} dw_n
+ \int_{\Ga}\frac{e^{-w_n^2/10}}{|x|^{l+n-1}}\frac1{|x|^{n-2}}\,dS_w
\lesssim\frac1{|x|^{l+n-1}}.
}
If $\be<n=2$, integration by parts does not help. Direct estimating using \thref{lem6-1}
gives
\EQN{
|J_{3}|
& \lec \frac{1}{|x|^{l+n-1}}\int_{|w_n|<|x|/2} e^{-w_n^2/10} \int_0^{3|x|}\frac{1}{(|x_n-w_n|+r)}\,drdw_n
\\
&\lec \frac{1}{|x|^{l+n-1}}\int_{|w_n|<|x|/2} e^{-w_n^2/10} \bke{1+ \log \frac{3|x|}{|x_n-w_n|}} \,dw_n.
}

If $|x_n| \ge \frac 34 |x|$ so that $|x_n-w_n|\ge \frac 14|x|$, the integral is of order one. If $|x_n|<\frac 34|x|$ so that $|x'| \ge c |x|$,
the integral is bounded by $\log \bka{x'}$.
Thus
\[
|J_3| \lec  \frac{1}{|x|^{l+n-1}} \bke{1 + \de_{n2} \log \bka{x'}}.
\]

Finally we consider $J_4$ in region $\textup{IV}$. Denote $\Ga=\{(w',w_n): |w'|=|x|/2 \ge |w_n|\}$ the lateral boundary of $\textup{IV}$.
Integrating by parts repeatedly,
\[
J_4= \int_{\textup{IV}}B(w,1)\pd_{w'}^{l+1}\pd E(x-w)\,dw + \sum_{p=0}^{l} \int_{\Ga}\pd_{w'}^{l-p}B(w,1)\,\pd_{w'}^p\pd E(x-w)\cdot\chi_p(w)dS_w
\]
where $\chi_p$ are uniformly bounded functions on $\Ga$ depending on multi-index $p$. By
 \eqref{eq_estB2}, that $|x-w|> c|x|$ on IV and $\Ga$, and $|w|> c|x|$ on $\Ga$, and \thref{lem6-1},
\EQN{
|J_4| &\le
\int_{\textup{IV}}\frac{e^{-w_n^2/10}}{\bka{w}^{n-2}}\frac1{|x|^{l+n}}\,dw
+ \sum_{p=0}^l \int_{\Ga}\frac{e^{-w_n^2/10}}{|x|^{l-p+n-2}}\frac1{|x|^{p+n-1}}dS_w
\\
&\lec \int_{|w_n|<|x|/2}  \frac{e^{-w_n^2/10}}{|x|^{l+n}}  \bke{\int_{|z'|<3|x|} \frac {dz'}{|z'|^{n-2}}} dw_n
+ \int_{\Ga}\frac{e^{-w_n^2/10}}{|x|^{l+2n-3}}\,dS_w
\lesssim\frac1{|x|^{l+n-1}}.
}
If $n=2$, we do one less step in integration by parts,
\[
J_4= \int_{\textup{IV}}\pd_{w'}B(w,1)\pd_{w'}^{l}\pd E(x-w)\,dw + \sum_{p=0}^{l-1} \int_{\Ga}\pd_{w'}^{l-p}B(w,1)\,\pd_{w'}^p\pd E(x-w)\cdot\chi_p(w)dS_w
\]
Thus for $n=2$, by
 \eqref{eq_estB2} and \thref{lem6-1},
\EQN{
|J_4| &\le
\int_{\textup{IV}}\frac{e^{-w_n^2/10}}{|w|}\frac1{|x|^{l+n-1}}\,dw
+ \sum_{p=0}^{l-1} \int_{\Ga}\frac{e^{-w_n^2/10}}{|x|^{l-p+n-2}}\frac1{|x|^{p+n-1}}dS_w
\\
&\lec \int_{|w_n|<|x|/2}  \frac{e^{-w_n^2/10}}{|x|^{l+n-1}}  \bke{\int_0^{|x|/2} \frac {dr}{|w_n|+r} } dw_n
+ \frac1{|x|^{l+n-1}}
\\
&\lec  \frac1{|x|^{l+n-1}}  \bke{1+ \int_{|w_n|<|x|/2} e^{-w_n^2/10}  \bke{ 1+ \log \frac{|x|}{|w_n|} } dw_n}
\lesssim\frac{ \log \bka{x}}{|x|^{l+n-1}}.
}
Unlike  $\log \bka{x'}$ for $J_3$, we need $ \log \bka{x}$ for $J_4$.

Summing the estimates, we conclude for $k=0$, for all $x\in \R^n$ and $n \ge 2$,
\begin{equation}\label{eq_est_intB}
\begin{aligned}\abs{\int_{\R^n}\pd_{x'}^{l+1}B(x-w,1)\pd E(w)\,dw} \lesssim\frac {1 + \de_{n2} \log \bka{x}}{\bka{x}^{l+n-1}}.
\end{aligned}\end{equation}

Suppose now $k \ge 1$ and \eqref{Bint_est} has been proved for all $k' \le k-1$.
Thanks to $-\De E=\de$, we can reduce the order of the $x_n$-derivative in the integral as
\EQN{
J&=\int_{\R^n}(\pd_{x'}^{l+1} \pd_{x_n}^{k}B)(x-w,1)\pd_nE(w)\,dw\\
& =(\pd_{x'}^{l+1} \pd_{x_n}^{k-1}B)(x,1)-\sum_{\beta_1=1}^{n-1}\int_{\R^n}(\pd_{x'}^{l+1} \pd_{x_n}^{k-1}\pd_{w_{\beta_1}}B)(x-w,1)\pd_{\beta_1}E(w)\,dw.
}
If $k=1$, \eqref{Bint_est} follows from \eqref{eq_estB} and \eqref{eq_est_intB},
\EQN{
|J|&\lec |\pd_{x'}^{l+1} B(x,1)| + \frac{1+ \de_{n2} \log \bka{x}}{\bka{x}^{l+n}}
\\
&\lec  \frac{e^{-x_n^2/10}}{\bka{x}^{l+n-1}}+\frac{1+ \de_{n2} \log (|x|+e)}{(|x|+e)^{l+n}}
\lec \frac{1+ \de_{n2} \log (|x_n|+e)}{\bka{x}^{l+n-1}\,(|x_n|+e)}.
}
In the last inequality we have used that for $m \ge 1$
\EQS{\label{logt.t-decay}
f(t) =t^{-m} \log t \quad\text{ is decreasing in }t>e.
}

If $k\ge 2$, by
integrating by parts, the second term becomes
\[\begin{aligned}\int_{\R^n}(&\pd_{x'}^{l+1} \pd_{x_n}^{k-1}\pd_{w_{\beta_1}}B)(x-w,1)\pd_{\beta_1}E(w)\,dw
=\int_{\R^n}(\pd_{x'}^{l+1} \pd_{x_n}^{k-2}\pd_{w_{\beta_1}}^2B)(x-w,1)\pd_nE(w)\,dw.\end{aligned}\]
By \eqref{Bint_est} for $k'=k-2$, and \eqref{logt.t-decay} with $m=2$,
\EQN{
|J|&\lec |\pd_{x'}^{l+1} \pd_{x_n}^{k-1} B(x,1)| +  \frac{1+ \de_{n2} \log \bka{x}}{\bka{x}^{l+n+1}\bka{x_n}^{k-2}}
\\
&\lec  \frac{e^{-x_n^2/10}}{\bka{x}^{l+n-1}}  +  \frac{1+ \de_{n2} \log (|x|+e)}{\bka{x}^{l+n+1}(|x_n|+e)^{k-2}}
\lec \frac{1+ \de_{n2} \log \bka{x_n}}{\bka{x}^{l+n-1}\bka{x_n}^k}.\qedhere
}
\end{proof}

\begin{lem}\thlabel{lem_Bint_est<n}
For $B(x,t)$ defined by \eqref{eq_def_B}, for $l,k \in \NN_0$, for $\be<n$,
\EQS{\label{Bint_est<n}
\abs{\int_{\R^n}\pd_{x'}^{l+1}\pd_{x_n}^{k}B(x-w,1)\pd_\be E(w)\,dw}
\lec \frac {1+\de_{n2} \log \bka{\de_{k\le 1}|x'|+|x_n|}}{\bka{x}^{l+n-\de_{k0}} \bka{x_n}^{(k-1)_+}}.
}
Note in \eqref{Bint_est<n} $\de_{k\le 1}|x'|=0$ for $k>1$.
\end{lem}
\begin{proof}
The case $k=0$ is proved in the proof for \thref{lem_Bint_est}. When $k \ge 1$, we integrate by parts
\EQN{
J=\int_{\R^n}\pd_{x'}^{l+1}\pd_{x_n}^{k}B(x-w,1)\pd_\be E(w)\,dw
= \int_{\R^n}\pd_{x'}^{l+1}\pd_{x_n}^{k-1}\pd_\be B(x-w,1)\pd_n E(w)\,dw.
}
By \thref{lem_Bint_est},
\[
|J| \lec \frac {1+\de_{n2} \log \bka{\de_{k\le 1}|x'|+|x_n|}}{\bka{x}^{l+n} \bka{x_n}^{k-1}}.\qedhere
\]
\end{proof}

The following is our estimates of derivatives of $D_{ijm}$.

\begin{prop}\thlabel{D_estimate}
For $x,y\in\mathbb{R}^n_+$, $l,k,q \in \NN_0$, $i,m=1,\ldots,n$, and $j<n$, we have
\EQ{\label{D_est}
|\pd_{x',y'}^l\pd_{x_n}^k\pd_{y_n}^qD_{ijm}(x,y,1)|\lec
\frac{ 1+ \mu\de_{n2}  \log \bka{\nu|x'-y'|+x_n+y_n}}
{\bka{x-y^*}^{l+k+n-\si} \bka{x_n+y_n}^\si \bka{y_n}^q},
}
where $\si = (k+ \de_{m n}- \de_{in} -1)_+$, $\mu=1-\de_{k0}-\de_{k1}\de_{in}$, and $\nu =  \de_{q0} \de_{m<n} \de_{k(1+\de_{in})}$.
\end{prop}

\begin{remark}\label{D_estimate-rmk}
By a similar proof (see \cite[Appendix B]{KLLTbs}), we can show
\EQ{\label{C-estimate}
|\pd_{x',y'}^l\pd_{x_n}^k\pd_{y_n}^q  C_i(x,y,1)|\lesssim\frac{e^{-\frac1{30}{ y_n^2}}} {\bka{x-y^*}^{l+n-1}\bka{x_n+y_n}^k\bka{y_n}^{q+1}} ,
}
whose decay in $x'$ is not as good as \eqref{D_est} since $\pd_n \Ga$ in the definition of $C_i$ has an additional $\pd_n$ derivative  than $\pd_j \Ga_{mn}$  in the definition of $D_{ijm}$.
This is why formula \eqref{new_W} for $W_{ij}$ is preferred than \eqref{eq_Wij_formula}. It is worth to note that the main term of $G^*_{ij}$ in \eqref{E1.6} is closely related to $\pd_{y_j}C_i$ (compare \eqref{0902a}). Henceforth, their estimates \eqref{Solonnikov.est} and \eqref{C-estimate} are similar.
\end{remark}

\begin{proof}
{\bf $\bullet\,\boldsymbol{\pd_{x',y'}, \pd_{y_n}}$-estimate:}
Recall the definition \eqref{D_def} of $D_{ijm}$.
Changing the variables $w=x-y^*-z$ after taking derivatives, and using $j<n$,
\[
\pd_{x',y'}^l\pd_{y_n}^q  D_{ijm} (x,y,1) = \int _\Pi\pd_{w'}^{l+1}\pd_n^{q}\Ga_{mn}(w,1)\,\pd_iE(x-y^*-w)\,dw
\]
up to a sign,
where $\Pi=\{w\in\R^n:y_n\le w_n\le x_n+y_n\}$. It is bounded for finite $|x-y^*|$, and to prove the estimate, we may assume $R=|x-y^*|>100$. Decompose $\Pi=\Pi_1+\Pi_2$ where
\[\Pi_1=\Pi\cap\left\{|w|<\tfrac34R\right\},\ \ \
\Pi_2=\Pi\cap\left\{|w|>\tfrac34R\right\}.
\]
Integrating by parts in $\Pi_1$ with respect to $w'$ iteratively, it equals
\EQN{
=&\int_{\Pi_1}\bke{\pd_{w_n}^{q}\Ga_{mn}(w,1)}\,\pd_{w'}^{l+1}\pd_iE(x-y^*-w)\,dw'dw_n\\
&+\sum_{p=0}^{l}\int_{\Pi\cap\{|w|=\frac34R\}}\bke{\pd_{w'}^{l-p}\pd_{w_n}^{q}\Ga_{mn}(w,1)}\,\pd_{w'}^p\pd_iE\cdot \chi_p(x-y^*-w)\,dS_w\\
&+\int_{\Pi_2}\bke{\pd_{w'}^{l+1}\pd_{w_n}^q\Ga_{mn}(w,1)}\,\pd_iE(x-y^*-w)\,dw=I_1+I_2+I_3,
}
where $\chi_p$ are bounded functions on the boundary.
Estimate \eqref{eq_estSij} and \thref{lem6-1} imply
\EQN{
|I_1|\lesssim&~\int_{y_n}^{x_n+y_n}\int_{\R^{n-1}}\frac1{(|w'|+w_n+1)^{q+n}R^{l+n}}\,dw'dw_n \\
\lesssim&~\frac1{R^{l+n}}\int_{y_n}^{x_n+y_n}\frac1{(w_n+1)^{q+1}}\,dw_n
\lesssim\frac{x_n}{R^{l+n}(y_n+1)^{q}(x_n+y_n+1)}.
}
For $I_2$, estimate \eqref{eq_estSij} gives%
\EQN{
|I_2|\lesssim&~\sum_{p=0}^{l}\int_{|w|=\frac34R}\frac1{\bka{w}^{l+q-p+n}}\,\frac1{|x-y^*-w|^{n+p-1}}\,dS_w \\
\lesssim&~\sum_{p=0}^{l}\frac1{R^{l+q-p+n}R^{n+p-1}}\,R^{n-1}
\sim\frac1{R^{l+q+n}}
}
Using the estimate \eqref{eq_estSij} and \thref{lemma2.2}, %
\EQN{
|I_3|\lesssim&~\int_{\Pi_2}\frac1{\bka{w}^{l+q+n+1}}\,\frac1{|x-y^*-w|^{n-1}}\,dw \\
\lesssim&~\frac1{R^{l+q+n+1/2}}\int_{y_n}^{x_n+y_n}\int_{\R^{n-1}}\frac1{(|w'|+w_n+1)^{1/2}(|x'-y'-w'|+(x_n+y_n-w_n))^{n-1}}\,dw'dw_n\\
\lesssim&~\frac1{R^{l+q+n+1/2}}\int_{y_n}^{x_n+y_n}\bke{R^{-1/2}+R^{-1/2}\log\frac{R}{(x_n+y_n-w_n)}}dw_n\\
\sim&~\frac{x_n}{R^{l+q+n+1}}\bke{1+\log\frac{R}{x_n}}\lec  \frac1{R^{l+q+n}} ,
}
noting $|x'-y'|+w_n+1+x_n+y_n-w_n\sim R$. Therefore, we conclude that for $i,m=1,\ldots,n$ and $j<n$,%
\EQS{\label{eq_estx'y'D}
|\pd_{x',y'}^l\pd_{y_n}^qD_{ijm}(x,y,1)|\lesssim&\frac1{\bka{x-y^*}^{l+n}\bka{y_n}^{q}} .
}
\smallskip

\noindent{\bf $\bullet\,\boldsymbol{\pd_{x_n}}$-estimate:}
Note $j<n$ always. Also note that $j$ and $m$ in $D_{ijm}$ are not changed in \eqref{eq_pdxnDi} and \eqref{eq_pdxnDn}. For $k \ge 1$ and $i<n$, by \eqref{eq_pdxnDi} and \thref{lem_Bint_est},
\EQ{\label{0811a}
\pd_{x',y'}^l \pd_{x_n}^k\pd_{y_n}^q D_{ijm}(x,y,1)\lec \pd_{x',y'}^{l+1} \pd_{x_n}^{k-1}\pd_{y_n}^q D_{njm}(x,y,1) + \frac {\LN'}{ \bka{x-y^*}^{l+n+1-\de_{m n}} \bka{x_n+y_n}^{k+q-1+\de_{m n}}},
}
where %
\[
\LN' =  1+ \de_{n2} \log \bka{\nu|x'-y'|+x_n+y_n},\quad \nu=\de_{0(k+q-1+\de_{m n})} = \de_{k1} \de_{q0} \de_{m<n}.
\]

For $k \ge 1$ and $i=n$, by \eqref{eq_pdxnDn} and \eqref{eq_estSij},
\EQ{\label{0811b}
\pd_{x',y'}^l \pd_{x_n}^k\pd_{y_n}^q D_{njm}(x,y,1)\lec \pd_{x',y'}^{l+1} \pd_{x_n}^{k-1}\pd_{y_n}^q D_{\be jm}(x,y,1) + \frac {1}{ \bka{x-y^*}^{l+n+k+q}},
}
where $\be<n$.The proof of \eqref{D_est} is then completed by induction in $k$ using \eqref{0811a}, \eqref{0811b} and the base case \eqref{eq_estx'y'D}.%
\end{proof}

\medskip

\begin{prop}\thlabel{Hhat_estimate}
For $x,y\in\mathbb{R}^n_+$, $t>0$, $l,k,q,m \in \NN_0$, $i,j=1,\ldots,n$, we have
\EQS{\label{Hhat_est}
|\pd_{x',y'}^l\pd_{x_n}^k\pd_{y_n}^q\pd_t^m \widehat H_{ij}(x,y,t)|\lec
\frac{1+\mu\,\de_{n2}\bkt{\log(\nu|x'-y'|+x_n+y_n+\sqrt{t})-\log(\sqrt{t})}}{t^{m}(|x^*-y|^2+t)^{\frac{l+k+n-\si}2}((x_n+y_n)^2+t)^{\frac{\si}2}(y_n^2+t)^{\frac{q}2}},
}
where $\si = (k- \de_{in}-\de_{jn} )_+$, $\mu=1-(\de_{k0}+\de_{k1}\de_{in})\de_{m0}$, and $\nu =  \de_{q0} \de_{jn} \de_{k(1+\de_{in})} \de_{m0}+\de_{m>0}$.
\end{prop}
\begin{proof}
From \eqref{Hhat_D} and \eqref{D_est},
\EQS{\label{Hhat_est1}
|\pd_{x',y'}^l\pd_{x_n}^k\pd_{y_n}^q \widehat H_{ij}(x,y,1)|\lec
\frac{1+\mu\,\de_{n2}\log\bka{\nu|x'-y'|+x_n+y_n}}{\bka{x^*-y}^{l+k+n-\si}\bka{x_n+y_n}^{\si}\bka{y_n}^{q}},
}
with corresponding $\si$, $\mu$ and $\nu$.
Note that $\widehat H_{ij}$ satisfies the scaling property
\begin{equation}\label{Hhat_invariant}
\widehat H_{ij}(x,y,t)=\frac1{t^{\frac{n}2}}\,\widehat H_{ij}\left(\frac{x}{\sqrt{t}},\frac{y}{\sqrt{t}},1\right).
\end{equation}
Therefore, \eqref{Hhat_est} can be obtained by differentiating \eqref{Hhat_invariant} in $t$ and using \eqref{Hhat_est1}.
Indeed,
\begin{multline*}
\pd_{x',y'}^l \pd_{x_n}^k\pd_{y_n}^q\pd_t^m \widehat H_{ij}(x,y,t)
=\bke{\frac{\pd}{\pd t}}^m \bke{t^{-\frac {l+k+q+n}2}
\pd_{X',Y'}^l \pd_{X_n}^k\pd_{Y_n}^q \widehat H_{ij}
\bke{ \frac {x}{\sqrt t} , \frac {y}{\sqrt t} ,1}}
\\
\qquad \sim t^{-\frac {l+k+q+n}2-m} \bke{ 1+ \textstyle \sum _{p=1}^n\frac {x_p}{\sqrt t}  \pd_{X_p} + \frac {y_p}{\sqrt t} \pd_{Y_p} }^m
 \pd_{X',Y'}^l \pd_{X_n}^k\pd_{Y_n}^q \widehat H_{ij}
 \bke{ \frac {x}{\sqrt t} , \frac {y}{\sqrt t} ,1}.
\end{multline*}
Here we use $\frac{\pd}{\pd t}$ to indicate a total derivative, and $\pd_{X_p}$ for a partial derivative in that position, e.g., $\frac{\pd}{\pd x} (f(ax,by)) = a \pd_{X} f(ax,by) $.
Note that $\frac {x_p}{\sqrt t}  \pd_{X_p} $ and $\frac {y_p}{\sqrt t} \pd_{Y_p}$ do not change the decay estimate no matter $p<n$ or $p=n$, except that we take $\mu=\nu=1$ when $m>0$ for simplicity.
This completes the proof of \thref{Hhat_estimate}.
\end{proof}

\subsection{Estimates of $V_{ij}$}

\begin{lem}\label{lem57}
Let $V_{ij}(x,y,t)$ be defined by \eqref{eq-Vij-def}, $x,y\in\mathbb{R}^n_+$, $t>0$.
For $i<n$,
\EQ{\label{eq4-13}
V_{ij}(x,y,t)
= 2\ep_j\int_0^{x_n}
\int_{\R^n_+} \pd_{x_n} G^{ht}((x_n-z_n)e_n, w,t)\,\pd_j\pd_i E(w+x'-y^* +z_ne_n )\, dw\, dz_n.
}
For $i=n$,
\EQ{\label{eq4-14}
V_{nj}(x,y,t)
= -2\ep_j\sum_{\be<n} \int_0^{x_n}
\int_{\R^n_+}  G^{ht}((x_n-z_n)e_n, w,t)\,\pd_j\pd_\be^2 E(w+x'-y^* +z_ne_n )\, dw\, dz_n.
}
\end{lem}

\begin{proof}
First of all, by changing variables $\td w = (y+w)^*$ in definition \eqref{eq-La-def},
\EQS{\label{eq43-1}
\La_j(x,y,t) &=\pd_{y_n}\pd_{y_j} \int_{\R^n_+} G^{ht}(x,\td w^*,t) E(\td w^*-y)\, d\td w\\
&=-\pd_{y_n}\pd_{y_j} \int_{\R^n_+} G^{ht}(x,\td w,t) E(\td w-y^*)\, d\td w\\
&=-\pd_{y_j} \int_{\R^n_+} G^{ht}(x,\td w,t) \pd_{n}E(\td w-y^*)\, d\td w.
}
Decompose $V_{ij}(x,y,t) = V_{ij,1}(x,y,t) + V_{ij,2}(x,y,t)$, where $V_{ij,1}(x,y,t) = -2\de_{in} \La_j(x,y,t)$
and 
\EQ{\label{eq-Vij2-def}
V_{ij,2}(x,y,t) = -4\int_0^{x_n} \int_{\Si} \pd_{x_n}\La_j(x-z,y,t) \pd_iE(z)\, dz'dz_n.
}

If $i<n$, integrating by parts,
\EQN{
V_{ij,2}(x,y,t) &= 4\int_0^{x_n}\!\!\! \int_\Si \pd_{z_i} \pd_{x_n}\La_j(x-z,y,t) E(z)\, dz'dz_n\\
&= -4 \pd_{x_i}\int_0^{x_n}\!\!\! \int_\Si \pd_{x_n}\La_j(x-z,y,t) E(z)\, dz'dz_n.
}
From the third line of \eqref{eq43-1},
changing variable $w=\td w-x'$ and using $G^{ht}(x,w+p',t)=G^{ht}(x-p',w,t)$ for any $p'\in \Si$,
\EQ{\label{eq43-2}
\La_j(x,y,t) =-\pd_{y_j} \int_{\R^n_+} G^{ht}(x_ne_n, w,t) \pd_{n}E(w+x'-y^*)\, dw.
}
Using \eqref{eq43-2},
\EQN{
V_{ij,2}&(x,y,t)\\
&= 4\pd_{x_i}\int_0^{x_n}\!\!\! \int_\Si \pd_{x_n}\bke{ \pd_{y_j}
\int_{\R^n_+} G^{ht}((x_n-z_n)e_n, w,t) \pd_{n}E(w+x'-z'-y^*)\, dw}
E(z)\, dz'dz_n
\\
&= 2\pd_{x_i}\int_0^{x_n}  \pd_{x_n}\pd_{y_j}
\int_{\R^n_+} G^{ht}((x_n-z_n)e_n, w,t)\bke{2\int_\Si  \pd_{n}E(w+x'-z'-y^*)
E(z)\, dz'}\, dw\, dz_n.
}
Using the stationary Poisson formula \eqref{KMT2.32}, $w+x'-z'-y^*\in \R^n_+$ and $E(z) = E(z' - (z_ne_n))$,
\EQS{\label{eq4-6}
V_{ij,2}(x,y,t)
&=- 2\pd_{x_i}\int_0^{x_n}  \pd_{x_n}\pd_{y_j}
\int_{\R^n_+} G^{ht}((x_n-z_n)e_n, w,t)\,E(w+x'-y^* +z_ne_n )\, dw\, dz_n
\\
&=- 2\int_0^{x_n}  \pd_{y_j}
\int_{\R^n_+} \pd_{x_n} G^{ht}((x_n-z_n)e_n, w,t)\,\pd_i E(w+x'-y^* +z_ne_n )\, dw\, dz_n.
}
Since $V_{ij,1}=0$ when $i<n$, we get \eqref{eq4-13}.

If $i=n$,
\[
V_{ij,2}(x,y,t) = -4\int_0^{x_n}\!\!\! \int_\Si \pd_{x_n}\La_j(x-z,y,t) \pd_nE(z)\, dz'dz_n.
\]
From the second line of \eqref{eq43-1},
changing variable $w=\td w-x'$ and using $G^{ht}(x,w+p',t)=G^{ht}(x-p',w,t)$ for any $p'\in \Si$,
\EQS{\label{eq43-4}
\La_j(x,y,t) &=-\pd_{y_n}\pd_{y_j} \int_{\R^n_+} G^{ht}(x,\td w,t) E(\td w-y^*)\, d\td w
\\
&=-\pd_{y_n}\pd_{y_j} \int_{\R^n_+} G^{ht}(x_ne_n, w,t) E(w+x'-y^*)\, dw.
}
Using \eqref{eq43-4},
\EQN{
V_{ij,2}&(x,y,t)\\
&= 4\int_0^{x_n}\!\!\! \int_\Si \pd_{x_n}\bke{ \pd_{y_n} \pd_{y_j}
\int_{\R^n_+} G^{ht}((x_n-z_n)e_n, w,t) E(w+x'-z'-y^*)\, dw}\pd_n
E(z)\, dz'dz_n
\\
&= 2\int_0^{x_n}  \pd_{x_n}\pd_{y_n} \pd_{y_j}
\int_{\R^n_+} G^{ht}((x_n-z_n)e_n, w,t)\bke{2\int_\Si  E(w+x'-z'-y^*)\pd_{n}
E(z)\, dz'}\, dw\, dz_n.
}
By \eqref{KMT2.32}, one has
\EQN{
2\int_\Si  E(w+x'-z'-y^*)\pd_{n}
E(z)\, dz'
&=2\int_\Si  E(z'-(w+x'-y^*))\pd_{n}
E(z_ne_n-z')\, dz' \\ &= -E(w+x'-y^*+z_ne_n)
}
and
\EQS{\label{eq4.25}
V_{ij,2}(x,y,t)
&=-2\int_0^{x_n}  \pd_{x_n}\pd_{y_n} \pd_{y_j}
\int_{\R^n_+} G^{ht}((x_n-z_n)e_n, w,t)E(w+x'-y^*+z_ne_n)\, dw\, dz_n
\\
&=-2\int_0^{x_n}  \pd_{y_j}
\int_{\R^n_+} \pd_{x_n}G^{ht}((x_n-z_n)e_n, w,t)\pd_{n} E(w+x'-y^*+z_ne_n)\, dw\, dz_n.
}

Therefore, for all $1\le i ,j\le n$, including $i=n$ or $j=n$, we have \eqref{eq4-6}.
Integrating \eqref{eq4.25} by parts in $z_n$,
\EQN{
&V_{ij,2}(x,y,t)
= - 2\ep_j\int_0^{x_n}
\int_{\R^n_+} \pd_{z_n} G^{ht}((x_n-z_n)e_n, w,t)\,\pd_j\pd_n E(w+x'-y^* +z_ne_n )\, dw\, dz_n\\
&= - 2\ep_j\int_{\R^n_+} G^{ht}(0, w,t)\,\pd_j\pd_n E(w+x-y^*)\, dw
 + 2\ep_j\int_{\R^n_+} G^{ht}(x_ne_n, w,t)\,\pd_j\pd_n E(w+x'-y^*)\, dw\\
&\quad + 2\ep_j\int_0^{x_n}
\int_{\R^n_+}  G^{ht}((x_n-z_n)e_n, w,t)\,\pd_j\pd_n^2 E(w+x'-y^* +z_ne_n )\, dw\, dz_n\\
&= 0 - V_{ij,1}(x,y,t) - 2\ep_j\int_0^{x_n}
\int_{\R^n_+}  G^{ht}((x_n-z_n)e_n, w,t)\,\sum_{\be <n}\pd_j\pd_\be^2 E(w+x'-y^* +z_ne_n )\, dw\, dz_n.
}
Then \eqref{eq4-14} follows from the above equation, completing the proof of the lemma.
\end{proof}

\begin{prop}\thlabel{V_estimate}
For $x,y\in\mathbb{R}^n_+$, $t>0$, $l,k,q,m \in \NN_0$, $i,j=1,\ldots,n$, we have
\EQ{\label{eq-Error-est}
| \pd_{x',y'}^l \pd_{x_n}^k \pd_{y_n}^q \pd_t^m V_{ij}(x,y,t) |
\lec \frac1{t^m (|x-y^*|^2 + t)^{\frac{l+k-k_i+q+n}2} (x_n^2 + t)^{\frac{k_i}2}}, \quad k_i=(k-\de_{in})_+.
}
\end{prop}

\begin{proof} We first consider $t=1$. 

\medskip

\noindent{\bf $\bullet$\, All spatial derivatives are bounded}

By Lemma \ref{lem57}, all spatial derivatives of $V_{ij}$ at $t=1$ are equal to sums of integrals of the form
\[
\int_0^{x_n}\int_{\R^n_+} \pd_{x_n}^k G^{ht}((x_n-z_n)e_n, w,1)\,\pd_{y_n}^q\pd_{x',y'}^l E(w+x'-y^* +z_ne_n )\, dw\, dz_n
\]
and
\[
\int_{\R^n_+} \pd_{x_n}^k G^{ht}(0, w,1)\,\pd_{x_n,y_n}^q\pd_{x',y'}^l E(w+x-y^* )\, dw.
\]
We first use $\pd_n^2 E=-\sum_{\be<n}\pd_\be^2 E$ to reduce $q$ to $q=0,1$. We then integrate by parts in tangential variables to move $\pd_{x',y'}^l$ in front of $G^{ht}$. When $|w+x'-y^* +z_ne_n|>1$,  
$\pd_{y_n}^q E(w+x'-y^* +z_ne_n )$ is bounded and the heat kernel derivative is integrable. 
When $|w+x'-y^* +z_ne_n|<1$, the heat kernel derivative is bounded and $\pd_{y_n}^q E(w+x'-y^* +z_ne_n )$ is integrable. 
Hence all such integrals are bounded. These bounds are uniform in $x,y\in \R^n_+$ with $x_n < C$ for fixed $C>0$. 
\medskip

\noindent{\bf $\bullet\,\boldsymbol{\pd_{x',y'}, \pd_{y_n}}$-estimate:}
We first estimate $V_{ij}(x,y,1)$.

For $i=n$, changing variable in \eqref{eq4-14}, it follows that
\EQ{
V_{nj}(x,y,1)
= -2\sum_{\be<n} \int_0^{x_n}
\int_{w_n<x_n-z_n}  G^{ht}((x_n-z_n)e_n, (x_n-z_n)e_n-w,1)\,\pd_j\pd_\be^2 E(w-x+y^*  )\, dw\, dz_n.
}
We split the set $A:=\bket{w\in \R^n: w_n<x_n-z_n}$ into two disjoint sets denoted by
\[
A_L=\bket{w: |w-x+y^*|>\frac{|x-y^*|}{2}}\cap A, \quad
A_S=\bket{w: |w-x+y^*|\le \frac{|x-y^*|}{2}}\cap A.
\]
For the region on $A_L$, it is direct that
\EQS{\label{eq4-15}
&\abs{\int_0^{x_n}
\int_{A_L}  G^{ht}((x_n-z_n)e_n, (x_n-z_n)e_n-w,1)\,\pd_j\pd_\be^2 E(w-x+y^*  )\, dw\, dz_n}\\
&\qquad\le \frac{c}{|x-y^*|^{n+1}}\int_0^{x_n}
\int_{A_L}  \abs{G^{ht}((x_n-z_n)e_n, (x_n-z_n)e_n-w,1)}\, dw\, dz_n\\
&\qquad\lec \frac1{|x-y^*|^{n+1}}\int_0^{x_n} dz_n \lec \frac1{|x-y^*|^{n+1}}\, Cx_n \lec |x-y^*|^{-n}.
}

On the other hand, on $A_S$, noting that $|w|>\frac{|x-y^*|}{2}$, 
that $x_n-z_n>w_n$ implies $2(x_n-z_n)-w_n>|w_n|$ and $|2(x_n-z_n)-w|\ge |w|$,
and using integration by parts, the integral over $A_S$ is bounded by
\EQN{
&\abs{\int_0^{x_n}
\int_{A_S} \pd_{w_\be}^2G^{ht}((x_n-z_n)e_n, (x_n-z_n)e_n-w,1)\, \pd_{j} E(w-x+y^*  )\, dw\, dz_n} \\
&+ \abs{\sum_{m=0}^1 \int_0^{x_n} \int_{|w-x+y^*|=\frac{|x-y^*|}2}  \pd^mG^{ht}((x_n-z_n)e_n, (x_n-z_n)e_n-w,1)\, \pd^{1-m}\pd_{j} E(w-x+y^*  ) \,dS_wdz_n}\\
&\qquad\le ce^{-c|x-y^*|^2} ( |x-y^*|^{-1}+1) x_n \lec |x-y^*|^{-n}.
}

For $i<n$, noting that $G^{ht}((x_n-z_n)e_n, w,1)=0$ if $z_n =x_n$, it follows via integration by parts in \eqref{eq4-13} that
\EQS{\label{eq4-13-byparts}
V_{ij}(x,y,1)&=2\ep_j\int_{\R^n_+}  G^{ht}(x_ne_n, w,1)\,\pd_j\pd_i E(w+x'-y^* )\, dw\\
&\quad +2\ep_j\int_0^{x_n}
\int_{\R^n_+}  G^{ht}((x_n-z_n)e_n, w,1)\,\pd_n\pd_j\pd_i E(w+x'-y^* +z_ne_n )\, dw\, dz_n.
}
For the second integral, we use $\pd^2_{n}\pd_i E=-\sum_{\be=1}^{n-1}\pd^2_{\be}\pd_i E$ when $j=n$
to reduce the order of normal derivative on $E$ at most 1.
 Then the second integral can be treated in exactly the same way as the case $i=n$, for both $j<n$ and $j=n$. 
 
For the first integral, as before, by change of variables, we rewrite
\EQN{
\int_{\R^n_+}& G^{ht}(x_ne_n, w,1)\,\pd_{j}\pd_i E(w+x'-y^* )\, dw\\
&=\int_{w_n<x_n}  G^{ht}(x_ne_n, x_ne_n -w,1)\,\pd_{j}\pd_i E(w-x+y^*  )\, dw\\
&=\int_{A_L} \cdots\, dw+\int_{A_S} \cdots\, dw.
}
Here we split the integral into two regions $A_L$ and $A_S$ with replacement of $A:=\bket{w\in \R^n: w_n<x_n}$.
The first term is rather direct that
\[
\abs{\int_{A_L}  G^{ht}(x_ne_n, x_ne_n -w,1)\,\pd_{j}\pd_i E(w-x+y^*  )\, dw}\le \frac{c}{|x-y^*|^{n}}.
\]
For the second term, since $i<n$, by integration by parts, we have
\EQN{
\abs{\int_{A_S} \cdots\, dw}&=\abs{\int_{A_S} \pd_{w_i} G^{ht}(x_ne_n, x_ne_n -w,1)\,\pd_{j} E(w-x+y^*  )\, dw}
\\
&\le  ce^{-c|x-y^*|^2} |x-y^*| \le {|x-y^*|^{-n}}.
}

Hence, for $i,j=1,\ldots,n$, since $V_{ij}$ is bounded for $|x-y^*| \le 1$, we have that
\EQS{
|V_{ij}(x,y,1)|\lesssim\frac1{\bka{x-y^*}^n} .
}

Any higher tangential derivative can be treated similarly as above. Furthermore, any order of normal derivative in $y_n$ works out as well, with the aid of $\Delta E=0$.
Therefore, we conclude that for $i,j=1,\ldots,n$,
\EQS{\label{eq-Vij-tangential}
|\pd_{x',y'}^l\pd_{y_n}^qV_{ij}(x,y,1)|\lesssim\frac1{\bka{x-y^*}^{l+q+n}} .
}

\medskip
\noindent{\bf $\bullet\,\boldsymbol{\pd_{x_n}}$-estimate:}

For $i=n$, using $G^{ht}((x_n-z_n)e_n, w,1)=0$ if $z_n =x_n$ in \eqref{eq4-14}, we deduce by integration by parts in $w_n$ and induction in $k$
that
\begin{align}
\nonumber
\pd_{x_n}^k V_{nj}(x,y,1) &=  \sum_{\be<n} \sum_{m=0}^{k-1} c\int_{\R^n_+} \pd_{x_n}^m G^{ht}(x_ne_n,w,t)\, \pd_n^{k-1-m}\pd_j\pd_\be^2 E(w+x'-y^*)\, dw\\
&\quad +c \sum_{\be<n} \int_0^{x_n} \int_{\R^n_+} G^{ht}((x_n-z_n)e_n, w,t)\, \pd_n^k\pd_j\pd_\be^2E(w+x'-y^* +z_ne_n )\, dw\, dz_n \nonumber\\
&=: I_1 + I_2.\label{eq-Vnj-xn-derivative}
\end{align}
Using estimates similar to \eqref{eq-Vij-tangential} via the spatial decomposition $A_S\cup A_L$, one has $|I_2| \lec |x-y^*|^{-n-k}$.
For $I_1$, recall $G^{ht}(x_ne_n,w,t) = \Ga(x_ne_n-w,t) - \Ga(x_ne_n-w^*,t)$. 
The contribution from $\Ga(x_ne_n-w^*,t)$ to $I_2$ is bounded by the contribution from $\Ga(x_ne_n-w,t)$ since $w_n\ge0$.
It suffices to estimate
\EQN{
\sum_{\be<n} \int_{\R^n_+} &\pd_n^m \Ga(x_ne_n-w,t)\, \pd_n^{k-1-m}\pd_j\pd_\be^2 E(w+x'-y^*)\, dw\\
&= - \int_{\R^n_+} \pd_n^m \Ga(x_ne_n-w,t)\, \pd_n^{k+1-m}\pd_j E(w+x'-y^*)\, dw.
}
Moving the normal derivatives form $\Ga$ to $E$ via integration by parts, it becomes
\EQN{
&=  \sum_{p=0}^{m-1} c\int_\Si \pd_n^p\Ga(x_ne_n-w',t)\, \pd_n^{k-p}\pd_j E(w'+x'-y^*)\, dw'\\
&\quad+c \int_{\R^n_+} \Ga(x_ne_n-w,t) \pd_n^{k+1} \pd_j E(w+x'-y^*)\, dw,
}
where the sum is bounded by $e^{-\frac{x_n^2}8} |x-y^*|^{-n-k+m}$ and the second term is bounded by $|x-y^*|^{-n-k}$ using estimates similar to $I_2$.
Thus,
\EQ{\label{eq-Vnj-xn-decay}
| \pd_{x_n}^k V_{nj}(x,y,1) | 
\lec e^{-\frac{x_n^2}8} \sum_{m=0}^{k-1} |x-y^*|^{-n-k+m} + |x-y^*|^{-n-k}
\lec |x-y^*|^{-n-1} x_n^{1-k}.
}

For $i<n$, we use \eqref{eq4-13-byparts} to derive a formula similar to \eqref{eq-Vnj-xn-derivative}, replacing
$\sum_{\be<n}\pd_\be^2$ by $\pd_n \pd_i$, and $\sum_{m=0}^{k-1}$ by $\sum_{m=0}^{k}$.
Thus, $x_n$-derivatives of $V_{ij}$ satisfies the decay estimate
\EQ{
| \pd_{x_n}^k V_{nj}(x,y,1) | 
\lec e^{-\frac{x_n^2}8} \sum_{m=0}^{k} |x-y^*|^{-n-k+m} + |x-y^*|^{-n-k}
\lec |x-y^*|^{-n} x_n^{-k}.
}

Therefore, by the same argument of the first part of the proof that $V_{ij}$ is bounded for $x_n\le1$, we obtain for $i,j=1,\ldots,n$ that
\EQS{
| \pd_{x_n}^k V_{ij}(x,y,1) | 
\lec \frac1{\bka{x-y^*}^{n+k-k_i} \bka{x_n}^{k_i}}, \quad k_i=(k-\de_{in})_+.
}
Moreover,
tangential and $y_n$ derivatives will only hit the second factor $E$ in the integrand, and the same estimate leads to
\EQS{\label{eq-Vij-normal}
| \pd_{x',y'}^l \pd_{x_n}^k \pd_{y_n}^q   V_{ij}(x,y,1) | 
\lec \frac1{\bka{x-y^*}^{n+l +q+ k-k_i} \bka{x_n}^{k_i}}. 
}

Finally, \thref{V_estimate} follows from \eqref{eq-Vij-tangential}, \eqref{eq-Vij-normal}, and the scaling property
\[
V_{ij}(x,y,t)=\frac1{t^{\frac{n}2}}\,V_{ij}\left(\frac{x}{\sqrt{t}},\frac{y}{\sqrt{t}},1\right).\qedhere
\]
\end{proof}

\subsection{Proof of \thref{prop2}}

We now prove \thref{prop2}.

\begin{proof}[Proof of \thref{prop2}]
We first estimate the Green tensor $G_{ij}$, which satisfies the formula in Lemma \ref{Green-formula}.
By \eqref{eq_estSij} and \thref{Hhat_estimate}, the estimates of $G_{ij}$ is
bounded by the sum of $\left(|x-y|^2+t\right)^{-\frac{l+k+q+n}2-m}$,
those in \thref{Hhat_estimate} for $\widehat H_{ij}$
and those in \thref{V_estimate} for $V_{ij}$.
This shows \eqref{Green_est}.

We now estimate the pressure tensor $g_j$.
Recall the decomposition formula \eqref{eq_formula_gj} that $g_j = -F_j^y(x)\de(t) +\widehat w_j$ in
\thref{gj-decomp}. For $t>0$, it suffices to estimate
\EQN{
\pd_{x',y'}^l\pd_{x_n}^k\pd_{y_n}^q \widehat w_j(x,y,t)
\sim&- \sum_{i<n} 8\int_0^t\int_{\Si}\pd_i \pd_n^{k+1}A(\xi',x_n,\tau)\pd_{x'}^l\pd_n^{q+1}S_{ij}(x'-y'-\xi',-y_n,t-\tau)\,d\xi'd\tau\\
&+ \sum_{i<n} 4\int_{\Si}\pd_{x'}^l\pd_{x_n}^k\pd_i E(x-\xi') \pd_n^{q+1}S_{ij}(\xi'-y,t)\,d\xi' \\
&+ 8\int_{\Si} \pd_n^{k+1}A(\xi',x_n,t) \pd_{x'}^l \pd_n^{q+1}\pd_jE(x'-y'-\xi',-y_n)\, d\xi'
=:\textup{I}+\textup{II}+\textup{III}.
}

We first estimate $\textup{I}$. Using \eqref{eq_estA} and \eqref{eq_estSij}, we get
\EQN{
\textup{I}\lesssim&\int_0^t\int_{\Si}\frac1{\tau^{\frac12}(|\xi'|+x_n+\sqrt{\tau})^{k+n}}\,\frac1{(|\xi'-(x'-y')|+y_n+\sqrt{t-\tau})^{l+q+n+1}}\,d\xi'd\tau\\
=&~\bke{ \int_0^{t/2}\! \int_{\Si} +  \int_{t/2}^t \int_{\Si} } \bket{\cdots}\,d\xi'd\tau
=:~\textup{I}_1+\textup{I}_2.
}
We have
\[
|\textup{I}_1|\lesssim\int_{\Si}
\bke{\int_0^{t/2}\frac1{\tau^{\frac12}(|\xi'|+x_n+\sqrt{\tau})^{k+n}}\,d\tau }
\frac1{(|\xi'-(x'-y')|+y_n+\sqrt{t})^{l+q+n+1}}\,d\xi'.
\]
Let
\[
R=|x-y^*|+\sqrt t.
\]
By \thref{lem6-1} ($k>d$ case),
\EQN{
|\textup{I}_1| &\lec
\int_{\Si}\frac{\sqrt t}{(|\xi'|+x_n)^{k+n-1}(|\xi'|+x_n+\sqrt t)}\,\frac1{(|\xi'-(x'-y')|+y_n+\sqrt{t})^{l+q+n+1}}\,d\xi'
\\
&%
\lec
\int_{\Si}\frac{1}{(|\xi'|+x_n)^{k+n-1}}\,\frac1{(|\xi'-(x'-y')|+y_n+\sqrt{t})^{l+q+n+1}}\,d\xi'.
}
By \thref{lemma2.2},%
\[
|\textup{I}_1| \lec R^{-l-q-k-n-1} + \de_{k0} R^{-l-q-n-1} \log\bke{\frac R{x_n}} + \mathbbm 1_{k>0} R^{-l-q-n-1} x_n^{-k} + R^{-k-n+1} (y_n + \sqrt{t})^{-l-q-2}.
\]

For $\textup{I}_2$ and all $n\ge2$, by \thref{lem6-1},%
\EQN{
|\textup{I}_2|\lesssim&\int_{\Si}\frac1{t^{\frac12}(|\xi'|+x_n+\sqrt{t})^{k+n}}
\bke{\int_{\frac{t}2}^t\frac1{(|\xi'-(x'-y')|+y_n+\sqrt{t-\tau})^{l+q+n+1}}\,d\tau } d\xi'\\
\lesssim&\int_{\Si}\frac1{t^{\frac12}(|\xi'|+x_n+\sqrt{t})^{k+n}}\,\frac t {(|\xi'-(x'-y')|+y_n)^{l+q+n-1}(|\xi'-(x'-y')|^2+y_n^2+ t)   }\,d\xi' \\
\lesssim&%
\int_{\Si}\frac1{(|\xi'|+x_n+\sqrt{t})^{k+n}}\,\frac 1 {(|\xi'-(x'-y')|+y_n)^{l+q+n}   }\,d\xi' .
}
By \thref{lemma2.2},%
\[
|\textup{I}_2| \lec R^{-l-q-k-n-1} + R^{-l-q-n} (x_n + \sqrt{t})^{-k-1} + R^{-k-n} y_n^{-l-q-1}.
\]

Now, we estimate $\textup{II}$. Using the definition of $E$, \eqref{eq_estSij} and \thref{lemma2.2}, after integrating by parts, we get
\EQN{
|\textup{II}|\lesssim&\int_{\Si}\frac1{(|\xi'|+x_n)^{k+n-1}}\,\frac1{(|\xi'-(x'-y')|+y_n+\sqrt{t})^{l+q+n+1}}\,d\xi',
}
which is similar to $I_1$.
Hence
\EQN{
|\textup{I}+\textup{II}| &\lec \de_{k0} R^{-l-q-n-1} \log\bke{\frac R{x_n}} + \mathbbm 1_{k>0} R^{-l-q-n-1} x_n^{-k}+ R^{-l-q-n} (x_n + \sqrt{t})^{-k-1}  \\
&\quad+ R^{-k-n+1} (y_n + \sqrt{t})^{-l-q-2} + R^{-k-n} y_n^{-l-q-1}.
}
Using \eqref{eq_estA} and the definition of $E$, we have
\EQN{
|\textup{III}|\lesssim&\int_{\Si}\frac1{t^{\frac12}(|\xi'|+x_n+\sqrt{t})^{k+n-1}}\,\frac1{(|\xi'-(x'-y')|+y_n)^{l+q+n}}\,d\xi'.
}
By \thref{lemma2.2},
\[
|\textup{III}| \lec t^{-\frac12}\bke{%
\de_{k0} R^{-l-q-n}\log\frac{R}{x_n + \sqrt{t}} + \mathbbm{1}_{k>0} R^{-l-q-n}(x_n+\sqrt t)^{-k} + R^{-k-n+1} y_n^{-l-q-1}}.
\]

We conclude
\[
|\pd_{x',y'}^l\pd_{x_n}^k\pd_{y_n}^q\widehat w_j(x,y,t)|
\lec  t^{-\frac12} \bke{\de_{k0} \frac1{R^{l+q+n}} \log{\frac R{x_n}}+  \frac1{R^{l+q+n} x_n^{k}}
+ \frac 1{R^{k+n-1}y_n^{l+q+1}}}.
\]
This proves estimate \eqref{pressure_est} and completes the proof of \thref{prop2}.
\end{proof}

\begin{remark}
The pressure tensor estimate \eqref{pressure_est} is sufficient for our proof of \thref{prop1}, and can be improved by several ways: One can get alternative estimates by integrating $\xi'$ by parts in all three terms I, II and III to move decay exponents from $y_n$ to $x_n$. Furthermore, we can rewrite the last term III using integration by parts and $\De E = 0$ as
\[
\textup{III} = %
\left \{
\begin{split}
 8\int_{\Si} \pd_j \pd_nA(\xi',x_n,t) \pd_n E(x'-y'-\xi',-y_n)\, d\xi'\qquad \text{if } j<n,
 \\
\textstyle  \sum_{i<n} 8\int_{\Si} \pd_i \pd_nA(\xi',x_n,t) \pd_i E(x'-y'-\xi',-y_n)\, d\xi'\qquad \text{if } j=n.
\end{split}
\right.
\]
\end{remark}

\section{Restricted Green tensors and convergence to initial data}

In this section we first study the restricted Green tensors acting on solenoidal vector fields, showing Theorem \ref{th6.1}.
In addition to the restricted Green tensor $\breve G_{ij}$ of Solonnikov given in \eqref{E1.6}, we also identify another restricted Green tensor $\widehat G_{ij}$ in \eqref{0827a}.
 We then use them to show the convergence to initial data in pointwise and $L^q$ sense for solenoidal and general $u_0$ in Lemma \ref{th6.2} and Lemma \ref{th:Gij-initial-Lq}, respectively. These show Theorem \ref{Convergence-to-initial-data}.

\begin{proof}[Proof of Theorem \ref{th6.1}]
Suppose $\div u_0 = 0$ and $u_{0,n}|_{\Si} = 0$.
Let
\EQS{\label{breve-u-def}
u_i^L(x,t) &= \sum_{j=1}^n \int_{\R^n_+} G_{ij}(x,y,t) u_{0,j}(y)\, dy,\\
\breve u_i^L(x,t) = \sum_{j=1}^n \int_{\R^n_+} \breve G_{ij}(x,y,t) & u_{0,j}(y)\, dy,\quad
 \widehat u_i^L(x,t) =
\sum_{j=1}^n \int_{\R^n_+} \widehat G_{ij}(x,y,t)u_{0,j}(y)\,dy.
}
By Lemma \ref{Green-formula},
\EQS{\label{lem7-1-pf1}
u_i^L (x,t)  = &~
 \int_{\R^n_+} ( \Ga(x-y,t) - \Ga(x-y^*,t)) (u_0)_i(y) \, dy \\
& + \sum_{j=1}^n \int_{\R^n_+} \bke{ \Ga_{ij}(x-y,t) - \ep_i \ep_j \Ga_{ij}(x-y^*,t) } (u_0)_j(y) \, dy\\
& -4 \sum_{j=1}^n \int_{\R^n_+} \widehat H_{ij}(x,y,t) (u_0)_j(y) \, dy 
+ \sum_{j=1}^n \int_{\R^n_+} V_{ij}(x,y,t) (u_0)_j(y) \, dy\\
 =&: I_1+I_2+I_3 + I_4.
}

Note that $I_1$ corresponds to the tensor $\de_{ij}\bkt{ \Ga(x-y,t) - \Ga(x-y^*,t) }$ in both \eqref{E1.6} and \eqref{0827a}.
We claim $I_2+I_4=0$. Indeed, 
since $\Ga_{ij}(x-y,t) - \ep_i\ep_j\Ga_{ij}(x-y^*,t) + V_{ij}(x,y,t) = \pd_{y_j} T_i(x,y,t)$ with
\EQN{
T_i(x,y,t) &= \int_{\R^n}\pd_{y_i}\bkt{ \Ga(x-y-w,t)-\Ga(x-y^*-w,t)} E(w)\,dw\\
&\quad+ \left[ -2\de_{in} \pd_{y_n} \int_{w_n<-y_n} G^{ht}(x,y+w,t) E(w)\, dw \right.\\
&\qquad\qquad \left.-4 \int_0^{x_n} \int_{\Si} \pd_{x_n} \bke{\pd_{y_n} \int_{w_n<-y_n} G^{ht}(x-z,y+w,t) E(w)\, dw } dz'dz_n \right],
}
by \eqref{eq_def_Sij}, \eqref{eq-Vij-def} and \eqref{eq-La-def},
\[
I_2 + I_4 
= \sum_{j=1}^n \int_{\R^n_+} \pd_{y_j} T_i(x,y,t)  (u_0)_j(y) \, dy
=- \sum_{j=1}^n \int_{\R^n_+} T_i(x,y,t) \pd_{y_j} (u_0)_j(y) \, dy=0.
\]

For $I_3$, by separating the sum over $j<n$ and $j=n$, and using \thref{cancel_C_H},
\[
I_3 =-4 \sum_{j<n} \int_{\R^n_+}  (-D_{ijn})(x,y,t) (u_0)_j(y) \, dy - 4 \int_{\R^n_+} \sum_{\be<n} D_{i\be\be}(x,y,t) (u_0)_n(y) \, dy.
\]
Note that
\begin{align*}
&\sum_{j<n} \int_{\R^n_+} (-D_{ijn})(x,y,t) (u_0)_j(y) \, dy\\
&=\sum_{j<n} \int_{\R^n_+} \pd_{y_j}\bke{ \int_0^{x_n}\int_{\Si} \Ga_{nn}(x^*-y-z^*,t)\pd_iE(z)\,dz'dz_n} (u_0)_j(y) \, dy
\\
&=-\sum_{j<n} \int_{\R^n_+} \bke{ \int_0^{x_n}\int_{\Si} \Ga_{nn}(x^*-y-z^*,t)\pd_iE(z)\,dz'dz_n} \pd_{y_j} (u_0)_j(y) \, dy
\\
&= \int_{\R^n_+} \bke{ \int_0^{x_n}\int_{\Si} \Ga_{nn}(x^*-y-z^*,t)\pd_iE(z)\,dz'dz_n}\pd_{y_n} (u_0)_n(y) \, dy
\\
&= \int_{\R^n_+}\int_0^{x_n}\int_{\Si} \pd_n\Ga_{nn}(x^*-y-z^*,t)\pd_iE(z)\,dz'dz_n\, (u_0)_n(y) \, dy\\
&= \int_{\R^n_+} D_{inn}(x,y,t) (u_0)_n(y) \, dy.
\end{align*}
Hence
\[
I_3 =  - 4 \int_{\R^n_+}  \sum_{\be=1}^n D_{i\be\be}(x,y,t)\,(u_0)_n(y) \, dy.
\]

Since
\EQN{
\sum_{\be=1}^n D_{i\be\be}(x,y,t) =&~ \int_0^{x_n} \int_{\Si} \sum_{\be=1}^n \pd_{\be} \Ga_{\be n}(x^*-y-z^*,t)\pd_iE(z)\,dz'dz_n  \\
=&~ \int_0^{x_n} \int_{\Si} \sum_{\be=1}^n \pd_{y_\be}^2 \int_{\R^n} \pd_n\Ga(x^*-y-z^*-w,t)E(w) \, dw\,\pd_iE(z)\,dz'dz_n\\
=&~ -\int_0^{x_n} \int_{\Si} \pd_n \Ga(x^*-y-z^*,t) \pd_iE(z)\,dz'dz_n
=+ C_i(x,y,t),
}
where we used $-\De E=\de$, \eqref{lem7-1-pf1} becomes
\EQS{\label{lem7-1-pf3}
&\sum_{j=1}^n \int_{\R^n_+} G_{ij}(x,y,t) (u_0)_j(y) \, dy\\
& = \int_{\R^n_+} ( \Ga(x-y,t) - \Ga(x-y^*,t)) (u_0)_i(y) \, dy
 - 4 \int_{\R^n_+} C_i(x,y,t)\, (u_0)_n(y)\, dy.
}
This gives \eqref{0827a}.
On the other hand,
\EQN{
 I_3 &= 4\int_{\R^n_+} \int_0^{x_n} \int_{\Si} \pd_n \Ga(x^*-y-z^*,t) \pd_iE(z)\,dz'dz_n\, (u_0)_n(y)\, dy\\
&= 4\int_{\R^n_+} \int_0^{x_n} \int_{\Si} \Ga(x^*-y-z^*,t) \pd_iE(z)\,dz'dz_n\, \pd_n(u_0)_n(y)\, dy\\
&= -4\sum_{\be<n} \int_{\R^n_+} \int_0^{x_n} \int_{\Si} \Ga(x^*-y-z^*,t) \pd_iE(z)\,dz'dz_n\, \pd_\be (u_0)_\be(y)\, dy\\
&= -4\sum_{\be<n} \int_{\R^n_+} J_{i\be}(x,y,t) \cdot  (u_0)_\be(y)\, dy,
}
where for $\be<n$
\EQN{
J_{i\be} &= \pd_{x_\be}\int_0^{x_n} \int_{\Si}  \Ga(x^*-y-z^*,t) \pd_iE(z)\,dz'dz_n
\\
&=\pd_{x_\be} \int_0^{x_n} \int_{\Si}  \Ga(z-y^*,t) \pd_iE(x-z)\,dz'dz_n.
}
we conclude that
\EQS{\label{0825a}
&\sum_{j=1}^n \int_{\R^n_+} G_{ij}(x,y,t) (u_0)_j(y) \, dy \\
&= \int_{\R^n_+} ( \Ga(x-y,t) - \Ga(x-y^*,t)) (u_0)_i(y) \, dy
 - 4 \sum_{\be<n} \int_{\R^n_+} J_{i\be}(x,y,t) \cdot  (u_0)_\be(y)\, dy,
}
which gives \eqref{E1.6}. This completes the proof of Theorem \ref{th6.1}.
\end{proof}

\begin{remark}\label{presure-equav}
Similar to Theorem \ref{th6.1}, we have \emph{restricted pressure tensors}.
Let $f \in C^1_c(\overline {\R^n_+}\times\R;\R^n)$ be a vector field in $\R^n_+\times \R$ and $f = \bP f$, i.e., $\div f=0$ and $f_{n}|_\Si=0$. Then
\EQN{%
  \sum_{j=1}^n \int_{-\infty}^\infty \int_{\R^n_+} g_j(x,y,t-s) f_j(y,s)\,dy\,ds
  &= \sum_{j=1}^n \int_{-\infty}^t \int_{\R^n_+} \breve g_j(x,y,t-s) f_j(y,s)\,dy\,ds\\
  &= \sum_{j=1}^n \int_{-\infty}^t \int_{\R^n_+} \widehat g_j(x,y,t-s) f_j(y,s)\,dy\,ds,
  }
  where
\[
\breve g_j(x,y,t) = (\delta_{jn}-1)\pd_{y_j} Q(x,y,t)
 ,\quad
\widehat g_j(x,y,t)= \delta_{jn}\pd_{y_j} Q(x,y,t) ,
\]
and
\[
Q(x,y,t) = 4 \int_{\Si} \bigg[ E(x-\xi')\partial_n\Gamma(\xi'-y,t) + \Ga(x'-y'-\xi',y_{n},t)\pd_nE(\xi',x_n)\bigg] \,d\xi'.
\]
An equivalent formula of $\breve g_j$ appeared in Solonnikov \cite[(2.4)]{MR1992567}, but no $\widehat g_j$.%
Both $\breve g_j$ and $\widehat g_j$ are functions and do not contain delta function in time. Note that $\widehat g_j(x,y,t)= \breve g_j(x,y,t) -
\pd_{y_j} Q(x,y,t)$.
We can get infinitely many restricted pressure tensors by adding to them any gradient field $\pd_{y_j} P(x,y,t)$.
\hfill \qed
\end{remark}

\begin{lem}\label{th6.2}
Let $u_0 \in C^1_c(\overline {\R^n_+};\R^n)$ be a vector field in $\R^n_+$ and $u_0 = \bP u_0$, i.e., $\div u_0=0$ and $u_{0,n}|_\Si=0$. Then for all $i=1,\ldots,n$, and $1< q\le \infty$,
\EQS{\label{lem-7-2_2}
\lim_{t\to 0_+}\norm{u_{0,i}(x) - \tsum_{j=1}^n \int_{\R^n_+} G_{ij}(x,y,t)u_{0,j}(y)\,dy}_{L^q_x(\R^n_+)}=0.
}
\end{lem}
Note that the exponent $q$ in \eqref{lem-7-2_2} includes $\I$ but not $1$.

\begin{proof}
Choose $R>0$ so that $K= \bket{(x',x_n)\in \R^n: |x'|\le R, 0\le x_n\le R}$ contains the support of $u_0$.
Since $u_0$ is uniformly continuous with compact support inside $K$,
\EQ{
\left\|\int_{\R^n_+}\Ga(x-y,t) u_{0,i}(y)\, dy - u_{0,i}(x)\right\|_{L^q_x(\R^n_+)}
+ \left\|\int_{\R^n_+}\Ga(x^*-y,t) u_{0,i}(y)\, dy \right\|_{L^q_x(\R^n_+)} \to 0
}
as $ t\to 0_+$ for all $i$.
In view of \eqref{lem7-1-pf3}, to show \eqref{lem-7-2_2}, it suffices to show
\EQ{\label{0924a}
\lim_{t\to 0_+}\sup_{x\in \R^n_+} \norm{v_i(\cdot,t)}_{L^q(\R^n_+)}=0,
}
where
\[
v_i(x,t)=\int_{\R^n_+} C_i(x,y,t) u_{0,n}(y)\, dy.
\]
Note that $|u_{0,n}(y)| \le Cy_n$, for some $C>0$, since $u_{0,n}|_\Si=0$ and $u_0 \in C^1_c(\overline {\R^n_+})$. Using estimate \eqref{C-estimate} for $C_i$, we have
\EQS{\label{0926a}
|v_i(x,t)| &\le \int_K \frac{e^{-\frac{ y_n^2}{30t}}}{\bke{y_n+\sqrt{t}}\bke{|x-y^*|+\sqrt{t}}^{n-1}}\, Cy_n \,dy
\\
&\lec \int_{K}  \frac 1{(|x-y|+\sqrt t)^{n-1}} e^{-\frac{ y_n^2}{30t}}  \, dy = \int_{\R^n} f(x-y{, t})g(y{, t})\,dy,
}
where
\[
f(x{, t})= \frac 1{(|x|+\sqrt t)^{n-1}},\quad g(x{, t})=e^{-\frac{ x_n^2}{30t}} \mathbbm{1}_K(x).
\]
By Young's convolution inequality,
\[ \norm{v_i(\cdot,t)}_{L^q(\R^n_+)}\lec \|{(f*g)(\cdot,t)}\|_{L^q(\R^n)} \le \|f{(\cdot,t)}\|_{L^p(\R^n)}\|g{(\cdot,t)}\|_{L^r(\R^n)}\]
where
\[
\frac1p+\frac1r = \frac1q + 1, \qquad 1\le p,q,r\le \I.
\]
We first compute $L^p$-norm of $f$.
If $p>\frac{n}{n-1}$,
\[\|f{(\cdot,t)}\|_{L^p(\R^n)}=\bke{\int_{\R^n} \frac1{(|z|+\sqrt{t})^{(n-1)p}}\,dz}^{1/p} =C \sqrt{t}^{\frac{n}p-(n-1)}.\]
Next, we compute $L^r$-norm of $g$. We need $0\le\frac1p- \frac1q< 1$ so that $1\le r<\infty$.
\[
\int_{\R^n} |g|^r \le \int_0^{R} \int_{B_R'} e^{-\frac{z_n^2}{30t}}\,dz'dz_n = CR^{n-1}\sqrt{t} \int_0^{\frac{R}{\sqrt{t}}} e^{-\frac{u^2}{30}}\,du  \lesssim \sqrt{t}.\]
Hence
$\|g(\cdot,t)\|_{L^r}\lesssim \sqrt{t}^{\frac1r}$, and
\[\|{(f*g)(\cdot,t)}\|_{L^q}\lesssim \sqrt{t}^{\frac{n}p-(n-1)+\frac1r} = \sqrt{t}^{\frac1q+1+(n-1)\bke{\frac1p-1}}.\]
To have vanishing limit when $t\to0_+$, we require $\frac1q+1+(n-1)\bke{\frac1p-1}>0$.

When $q \in (\frac{n}{n-1},\infty]$,
we can choose $p \in (\frac n{n-1},\min(q,\frac{n-1}{n-2}))$ so that all conditions on $p$,
\[
p>\frac{n}{n-1}, \quad 0\le\frac1p- \frac1q< 1, \quad \frac1q+1+(n-1)\bke{\frac1p-1}>0
\]
are satisfied. This shows \eqref{0924a} for all $q \in (\frac{n}{n-1},\infty]$.

For the small $q$ case, let
\[
u^*_i(x,t) = \int_{\R^n_+} G^*_{ij}(x,y,t) u_{0, j}(y)\,dy,
\]
where $G^*_{ij}$ is given in the \eqref{E1.6}, and is the sum of the last terms of \eqref{E1.6}. It suffices to show
\[
\lim_{t \rightarrow 0} \norm{u^*_i(x,t) }_{L^q_x (\R^n_+)}=0.
\]
By estimate \eqref{Solonnikov.est},  $|G^*_{ij}(x,y,t)|\lesssim e^{-\frac{Cy_n^2}{t}}\bke{|x^*-y|^2+t}^{-\frac{n}{2}}$.
For $1<q<\infty$, using the Minkowski's inequality,
\EQN{
\norm{u^*_i(x,t)}_{L^q_x(\R^n_+)}&\lesssim \int_{\R^n_+}\bke{\int_{\R^n_+}\abs{ G^*_{ij}(x,y,t)}^q dx}^{\frac{1}{q}} \abs{u_{0}(y)}dy
\\
&\lesssim  \int_{\R^n_+}\bke{\int_{\R^n_+}\frac{dx}{\bke{|x^*-y|^2+t}^{\frac{nq}{2}} }}^{\frac{1}{q}} e^{-\frac{Cy_n^2}{t}}\abs{u_{0}(y)}dy
\\
&\lesssim \int_0^R \int_{|y'|<R}\frac{1}{(y_n +\sqrt{t})^{\frac{n(q-1)}{q}}}\,e^{-\frac{Cy_n^2}{t}}\,dy'\,dy_n
\\
&\lesssim t^{\frac{1}{2}\bke{1-\frac{n(q-1)}{q}}}\int_0^{R/\sqrt t}\frac{1}{(z_n +1)^{\frac{n(q-1)}{q}}}\,e^{-Cz_n^2} \,dz_n,
}
where $y_n=\sqrt t z_n$.
Therefore, if $1<q< \frac{n}{n-1}$, then the right hand side goes to zero as $ t\to 0_+$.

The case $q=\frac{n}{n-1}$ can be obtained using the previous cases and the H\"older inequality.

This finishes the proof
 of Lemma \ref{th6.2}.
\end{proof}

\begin{remark}
In the proof of Lemma \ref{th6.2}, we have used $\widehat G_{ij}$ for large $q$ and
$\breve G_{ij}$ for small $q$. We do not use $\widehat G_{ij}$ for small $q$ because the estimate \eqref{0926a} for $v_i$ does not have enough decay in $x$.
 We can not use $\breve G_{ij}$ for $q=\infty$ because, although the pointwise estimate of $u_i^*(x,t)$ using \eqref{Solonnikov.est}
 converges to 0 as $t\to0$ for each $x\in \R^n_+$, it is not uniform in $x$. In contrast, it is uniform for $v_i$ thanks to $|u_{0,n}(y)| \le C y_n$.

\end{remark}

\begin{lem}\thlabel{th:Gij-initial-Lq}
Let $u_0$ be a vector field in $\R^n_+$, $u_0\in L^q(\R^n_+)$, $1<q<\infty$,  and let $u_i(x,t)=\tsum_{j=1}^n \int_{\R^n_+} G_{ij}(x,y,t)u_{0,j}(y)\,dy$.
Then $u(x,t)\to (\bP u_0)(x)$ in $L^q(\R^n_+)$.
\end{lem}

This lemma does not assume $u = \bP u$, and implies \eqref{Gij-initial0}.

\begin{proof}
Since the Helmholtz projection $\bP$ is bounded in $L^q(\R^n_+)$, we also have $\bP u_0 \in L^q(\R^n_+)$. For any $\e>0$, choose $a =\bP a \in C^\infty_c(\overline{\R^n_+};\R^n)$ with $\norm{a-\bP u_0}_{L^q} \le \e$. Such $a$ may be obtained by first localizing $\bP u_0$ using a Bogovskii map, and then
mollifying the extension defined in \eqref{div0-extension} of the localized vector field.
Let $v_i(x,t)=\sum_{j=1}^n\int_{\R^n_+}G_{ij}(x,y,t)a_j(y)\,dy$. By Lemma \ref{th6.2}, there is $t_\e>0$ such that
\[
\norm{v(\cdot,t) -a}_{L^q(\R^n)} \le  \e, \quad \forall t \in (0,t_\e).
\]
By $L^q$ estimate \eqref{E10.5} in Lemma \ref{th9.1}, $\norm{u(t) - v(t)}_{L^q} \le C \norm{\bP u_0-a}_{L^q}\le C\e$.
Hence
\[
\norm{u(t)-\bP u_0}_{L^q} \le \norm{u(t)-v(t)}_{L^q} + \norm{v(t)-a}_{L^q} + \norm{a-\bP u_0}_{L^q} \le C\e
\]
for $t \in (0,t_\e)$.
This shows $L^q$-convergence of $u(t)$ to $\bP u_0$.
\end{proof}

\begin{proof}[Proof of Theorem \ref{Convergence-to-initial-data}]
Part (a) is by Lemma \ref{th:Gij-initial}.
Part (b) is by Lemma \ref{th:Gij-initial-Lq}.
Part (c) is by Lemma \ref{th6.2}.
\end{proof}

\section{The symmetry of the Green tensor}
\label{sec7}

In this section we prove \thref{prop1}, i.e., the symmetry of the Green tensor of the Stokes
system in the half-space,
\EQS{\label{Green.symmetry}
G_{ij}(x,y,t)=G_{ji}(y,x,t),\quad
\forall
x,y\in\R^n_+,\
\forall t\in\R\setminus\{0\}.
}
In the Green tensor formula in Lemma \ref{Green-formula}, this symmetry property is valid for the first three terms but unclear for the last two terms $ - \ep_i \ep_j \Ga_{ij}(x-y^*,t) - 4\widehat H_{ij}(x,y,t)$.
To prove it rigorously, we will use its regularity away from the singularity, bounds on spatial decay, and estimates near the singularity from the previous sections. For example, without the pointwise bound in \thref{prop2}, the bound
\eqref{sym_lim_y} is unclear, and it will take extra effort to show their zero limits as $\ep\to 0$.

Denote $G^y_{ij}(z,\tau)=G_{ij}(z, y, \tau)$ and $g^y_j(z,\tau)=g_j(z,y,\tau)=\widehat w_j^y(z,\tau) - F_j^y(z)\de(\tau)$ by \thref{gj-decomp}. Equation \eqref{Green-def-distribution} reads: For fixed $j=1,2,\cdots, n$ and $y\in \R^n_+$,
\EQS{\label{eq7.2}
\partial_{\tau} G^y_{ij}-\Delta_z
G^y_{ij}+\partial_{z_i}g^y_j=\delta_{ij}\delta_y(z)\delta(\tau),
\quad \sum_{i=1}^n \partial_{z_i}G^y_{ij}=0,\quad
(z,\tau)\in\R^n_+\times\R,
}
and $G^y_{ij}(z',0,\tau)=0$.
Denote $U:=\R^{n}_+\times \R$ and

\[
Q^{y,t}_{\epsilon}=B^y_\epsilon\times (t-\epsilon, \, t+ \epsilon).
\]
The inward normal $\nu_z$ on $\pd Q^{y,t}_{\epsilon}$ is defined on its lateral boundary as
\[
\nu_i (z,\tau)= -\frac{z_i-y_i}{|z-y|}.
\]

\begin{lem}\label{th7-1}
For $j=1,\ldots,n$, $y\in\R^n_+$, $t>0$, and all $f\in C^\infty (\R^n_+\times[0,t];\R^n)$, we have
\begin{align}
\label{fj_weak_form}
\notag
&f_j(y,0)=\lim_{\epsilon\to0_+}\sum_{k=1}^n\left[\int_{|z-y|=\ep}F^y_j(z)f_k(z,0)\nu_k\,dS_z\right.-\int_0^\ep\!\int_{|z-y|=\ep}G^y_{kj}(z,\tau)\na_zf_k(z,\tau)\cdot\nu_z\,dS_zd\tau\\
&+\int_0^\ep\!\int_{|z-y|=\ep}(\na_zG^y_{kj}(z,\tau)\cdot\nu_z)f_k(z,\tau)\,dS_zd\tau
\left.- \int_0^\ep\!\int_{|z-y|=\ep} \widehat w^y_j (z,\tau)f_k(z,\tau)\nu_k\,dS_zd\tau\right].
\end{align}
\end{lem}

\begin{proof}
We first assume $f\in C^\infty_c(\R^n_+\times \R;\R^n)$.
By the defining property \eqref{eq7.2} of Green tensor, we have
\EQN{
f_j(y,0)=&\sum_{k=1}^n\int_{U}\left[G^y_{kj}(z,\tau)(-\pd_{\tau} f_k(z,\tau)-\De_zf_k(z,\tau))-\widehat w^y_j  (z,\tau)\pd_{z_k}f_k(z,\tau)\right]dz\,d\tau\\
&+\int_{\R^n_+} {F_j^y} (z)\div f(z,0)\,dz.
}
Separating the domain of the first integral, we have
\EQN{
f_j(y,0)=&\lim_{\epsilon\to0_+}\sum_{k=1}^n\int_{U\setminus Q^{y,0}_\ep}\left[G^y_{kj}(z,\tau)(-\pd_\tau f_k(z,\tau)-\De_zf_k(z,\tau)) - \widehat w_j^y(z,\tau)\pd_{z_k} f_k(z,\tau)\right]dz\,d\tau\\
&+\int_{\R^n_+}{F_j^y}\div f(z,0)\,dz\\
=&\lim_{\epsilon\to0_+}\sum_{k=1}^n\left(\int_0^\infty\int_{|z-y|>\ep}+\int_{\ep}^\infty\int_{|z-y|<\ep}\right)[\cdots]\,dz\,d\tau+\int_{\R^n_+}{F_j^y} (z)\div f(z,0)\,dz.
}
Here we have used the fact that $G^y_{kj}(z,\tau)=w^y_j(z,\tau)=0$ for $\tau<0$.

Integrating by parts and using $f\in C^\infty_c(\R^n_+\times \R)$, we get
\begin{align*}
f_j(y,0)=&\lim_{\epsilon\to0_+}\sum_{k=1}^n\left[\int_{|z-y|>\ep}G^y_{kj}(z, 0_+)f_k(z,0)\,dz+\int_0^\infty\int_{|z-y|>\ep}\pd_\tau G^y_{kj}(z,\tau)f_k(z,\tau)\,dzd\tau\right.\\
&+\int_0^\infty\int_{|z-y|=\ep}\bket{-G^y_{kj}(z,\tau)\na_zf_k(z,\tau) + f_k(z,\tau) \na_z G^y_{kj}(z,\tau)} \cdot\nu_z\,dS_zd\tau\\
&-\int_0^\infty\int_{|z-y|>\ep}\De_zG^y_{kj}(z,\tau)f_k(z,\tau)\,dzd\tau-\int_0^\infty\int_{|z-y|=\ep} \widehat w^y_j  (z,\tau)f_k(z,\tau)\nu_k\,dS_zd\tau\\
&+\int_0^\infty\int_{|z-y|>\ep}\pd_{z_k} \widehat w^y_j  (z,\tau)f_k(z,\tau)\,dzd\tau\\
&+\int_{|z-y|<\ep} G^y_{kj}(z,\ep)f_k(z,\ep)\,dz +\int_\ep^\infty\int_{|z-y|<\ep}\pd_\tau G^y_{kj}(z,\tau)f_k(z,\tau)\,dzd\tau\\
&+\int_\ep^\infty\int_{|z-y|=\ep} \bket{-G^y_{kj}(z,\tau)\na_zf_k(z,\tau) + f_k(z,\tau) \na_z G^y_{kj}(z,\tau)} \cdot(-\nu_z)\,dS_zd\tau\\
&-\int_\ep^\infty\int_{|z-y|<\ep}\De_zG^y_{kj}(z,\tau)f_k(z,\tau)\,dzd\tau
-\int_\ep^\infty\int_{|z-y|=\ep} \widehat w^y_j  (z,\tau)f_k(z,\tau)(-\nu_k)\,dzd\tau\\
&+\int_\ep^\infty\int_{|z-y|<\ep}\pd_{z_k} \widehat w^y_j  (z,\tau)f_k(z,\tau)\,dzd\tau
+ \bke{\int_{|z-y|<\ep}+ \int_{|z-y|>\ep}} {F_j^y}(z)\pd_{z_k}  f_k(z,0)\,dz\bigg].
\end{align*}
Note that $\pd_\tau G^y_{kj}-\De_zG^y_{kj}+\pd_{z_k}\widehat w^y_j =\pd_\tau G^y_{kj}-\De_zG^y_{kj}+\pd_{z_k}g^y_j=0$ for $\tau>0$ and that $G^y_{kj}(z,0_+)=\pd_kF^y_j(z)$ if $y\not=z$. Therefore, after combining and integrating by parts the sum of the first term and the last term,
\EQN{
f_j(y,0)=&\lim_{\epsilon\to0_+}\sum_{k=1}^n\left[\int_{|z-y|=\ep}F^y_j(z)f_k(z,0)\nu_k\,dS_z\right.-\int_0^\ep\int_{|z-y|=\ep}G^y_{kj}(z,\tau)\na_zf_k(z,\tau)\cdot\nu_z\,dS_zd\tau\\
&+\int_0^\ep\int_{|z-y|=\ep}(\na_zG^y_{kj}(z,\tau)\cdot\nu_z)f_k(z,\tau)\,dS_zd\tau
-\int_0^\ep\int_{|z-y|=\ep} \widehat w^y_j  (z,\tau)f_k(z,\tau)\nu_k\,dS_zd\tau\\
&+\int_{|z-y|<\ep} G^y_{kj}(z,\ep)f_k(z,\ep)\,dz
+\left.\int_{|z-y|<\ep} {F_j^y}(z)\pd_{z_k}f_k(z,0)\,dz
\right].
}
The last two terms vanish as $\ep\to0_+$ since
\EQS{\label{sym_lim_y}
\left|\int_{|z-y|<\ep}G^y_{kj}(z,\ep)f_k(z,\ep)\,dz\right|\lesssim\int_0^\ep\frac{\|f\|_\infty}{(r+\sqrt{\ep})^{n}}\,r^{n-1}\,dr\lesssim\ep^{\frac{n+1}2}\to0\ \text{ as }\ep\to0_+
}
by \eqref{Green_est} and \thref{lem6-1},
and
\[\left|\int_{|z-y|<\ep} {F_j^y}(z)\pd_{z_k}f_k(z,0)\,dz\right|\lesssim\int_0^\ep\frac{ \|\nabla f\|_\infty}{r^{n-1}}\,r^{n-1}\,dr\lesssim\ep\to0\ \text{ as }\ep\to0_+.\]
Hence \eqref{fj_weak_form} is valid for all $f\in C^\infty_c(\R^n_+\times\R)$.

If $f\in C^\infty_c(\R^n_+\times[0,t])$, we can extend it to
$\tilde f\in C^\infty_c(\R^n_+\times\R)$. Hence \eqref{fj_weak_form} is valid for all such $f$. %
Finally, if $f\in C^\infty(\R^n_+\times[0,t])$ for some $t>0$,
let $\tilde f = f \zeta$ where $\zeta(z,\tau)$ is a smooth cut-off function which equals 1 in $Q^{y,0}_{2\ep}$.
Then \eqref{fj_weak_form} is valid for $\tilde f\in C^\infty_c(\R^n_+\times[0,\infty))$ and hence also for $f$.
This completes the proof of the lemma.
\end{proof}

We now prove the symmetry.
\begin{proof}[Proof of  \thref{prop1}]
 Fix $\Phi \in C^\infty_c(\R)$,  $\Phi(s)=1$ for $s \le 1$, and $\Phi(s)=0$ for $s\ge 2$.
For fixed $x\not= y\in\R^n_+$, $t>0$, and $i,j=1,\ldots,n$, by choosing $f_k(z,\tau)=G^x_{ki}(z,t-\tau)\eta^{x,t}(z,\tau)$ in \eqref{fj_weak_form} of Lemma \ref{th7-1},
where $\eta^{x,t}$ is a smooth cut-off function defined by
\[\eta^{x,t}(z,\tau)=1-\Phi\left(\frac{|x-z|}\ep\right)\Phi\left(\frac{|t-\tau|}\ep\right),\]
and using that $\eta^{x,t}(z,\tau)=1$ on $\{(z,\tau):0\le\tau\le\ep,\,|z-y|\le\ep\}$ for $\ep<|x-y|/3$,
we obtain
\EQS{\label{Gx_weak_form}
G^x_{ji}(y,t)=&\lim_{\epsilon\to0_+}\sum_{k=1}^n\left[\int_{|z-y|=\ep}F^y_j(z)G^x_{ki}(z,t)\nu_k\,dS_z\right.\\
&-\int_0^\ep\int_{|z-y|=\ep}G^y_{kj}(z,\tau)\na_zG^x_{ki}(z,t-\tau)\cdot\nu_z\,dS_zd\tau\\
&+\int_0^\ep\int_{|z-y|=\ep}(\na_zG^y_{kj}(z,\tau)\cdot\nu_z)G^x_{ki}(z,t-\tau)\,dS_zd\tau\\
&-\left.\int_0^\ep\int_{|z-y|=\ep}\widehat w^y_j (z,\tau)G^x_{ki}(z,t-\tau)\nu_k\,dS_zd\tau\right].
}
Switching $y$ and $j$ in the above identity with $x$ and $i$, respectively,
and changing the variables in $\tau$, we get
\EQS{\label{Gy_weak_form}
G^y_{ij}(x,t)=&\lim_{\epsilon\to0_+}\sum_{k=1}^n\left[\int_{|z-x|=\ep}F^x_i(z)G^y_{kj}(z,t)\nu_k\,dS_z\right.\\
&-\int_{t-\ep}^t\int_{|z-x|=\ep}G^x_{ki}(z,t-\tau)\na_zG^y_{kj}(z,\tau)\cdot\nu_z\,dS_zd\tau\\
&+\int_{t-\ep}^t\int_{|z-x|=\ep}(\na_zG^x_{ki}(z,t-\tau)\cdot\nu_z)G^y_{kj}(z,\tau)\,dS_zd\tau\\
&\left.-\int_{t-\ep}^t\int_{|z-x|=\ep} {\widehat w^x_i}(z,t-\tau)G^y_{kj}(z,\tau)\nu_k\,dS_zd\tau\right].
}

Denote
\[
U_\ep^{L,\de}:=\bket{(\R^n_+\cap \bket{|z|<L, L z_n>1 }) \times [\de,t-\de]} \setminus
(Q^{x,t}_{\epsilon}\cup Q^{y,0}_{\epsilon})
\]
for $0<\de<\ep<\min(t,|x-y|)/2$ and $L > 2(|x|+|y|+1)$.
Since $G^x_{ki}(z, t-\tau)$ and $G^y_{kj}(z, \tau)$ are smooth in $U_\ep^{L,\de}$,
$\bkt{(\partial_{\tau}-\Delta_z)G^y_{kj}+\partial_{z_k}g^y_j}(z,\tau)$ and $\bkt{(-\partial_{\tau}-\Delta_z)G^x_{ki}+\partial_{z_k}g^x_i}(z,t-\tau)$ vanish in $U_\ep^{L,\de}$,
and $g_j(x,y,t)= \widehat w_j(x,y,t)$ for $t>0$,
\EQN{
0
=&~\sum_{k=1}^n \int_{U_\ep^{L,\de}}G^x_{ki}(z,
t-\tau)\bkt{(\partial_{\tau}-\Delta_z)G^y_{kj}+\partial_{z_k} \widehat w^y_j  }(z,\tau)\,dz\,d\tau\\
&~-\sum_{k=1}^n \int_{U_\ep^{L,\de}}G^y_{kj}(z,\tau)\bkt{(-\partial_{\tau}-\Delta_z)G^x_{ki}+\partial_{z_k} {\widehat w^x_i} }(z,t-\tau)\,dz\,d\tau\\
=&~\sum_{k=1}^n\left(\int_\de^\ep\int_{|z-y|>\ep}+\int_\ep^{t-\ep}\int_{|z|<L}+\int_{t-\ep}^{t-\de}\int_{|z-x|>\ep}\right)[\cdots]\,dz\,d\tau.
}

By integration by parts,  $G^y_{kj}(z',0,t)=0$, $G^y_{kj}(z,0_+)=\pd_kF_j^y(z)$ if $y\not=z$,
$\sum_{k=1}^n \partial_{z_k}G^x_{kj}=0$, and taking limits $L\to \infty$ and $\de\to0_+$, ($\e>0$ fixed), we get
\EQS{\label{sym_1}
\sum_{k=1}^n& \left[-\int_{|z-y|<\ep}G^x_{ki}(z,t-\ep)G^y_{kj}(z,\ep)\,dz-\int_{|z-y|=\ep}G^x_{ki}(z,t)F^y_j(z)\nu_k\,dS_z\right.\\
&+\int_{|z-x|<\ep}G^x_{ki}(z,\ep)G_{kj}^y(z,t-\ep)\,dz+\int_{|z-x|=\ep}G_{kj}^y(z,t)F^x_i(z)\nu_k\,dS_z\\
&-\int_0^\ep\int_{|z-y|=\ep}\bkt{G^x_{ki}(z,t-\tau)\na_zG^y_{kj}(z,\tau)-G^y_{kj}(z,\tau)\na_zG^x_{ki}(z,t-\tau)}\cdot\nu_z\,dS_zd\tau\\
&-\int_{t-\ep}^t\int_{|z-x|=\ep}\bkt{G^x_{ki}(z,t-\tau)\na_zG^y_{kj}(z,\tau)-G^y_{kj}(z,\tau)\na_zG^x_{ki}(z,t-\tau)}\cdot\nu_z\,dS_zd\tau\\
&+\int_0^\ep\int_{|z-y|=\ep}\bkt{G^x_{ki}(z,t-\tau)\widehat w^y_j(z,\tau)-G^y_{kj}(z,\tau)\widehat w^x_i(z,t-\tau)}\nu_k\,dS_zd\tau\\
&\left.+\int_{t-\ep}^t\int_{|z-x|=\ep}\bkt{G^x_{ki}(z,t-\tau)\widehat w^y_j(z,\tau)-G^y_{kj}(z,\tau)\widehat w^x_i(z,t-\tau)}\nu_k\,dS_zd\tau\right]=0.
}
Note that the above integrals are over finite regions. We can take limits $\de\to0_+$ because in these regions we do not evaluate $G_{kj}^y(z,\tau)$ and $\widehat w_j^y(z,\tau)$ at their singularity $(y,0)$, nor $G_{ki}^x(z,t-\tau)$ and $\widehat w_i^x(z,t-\tau)$ at their singularity $(x,t)$.
To justify the limits $L \to \infty$, we first need to show that the far-field integrals
\EQN{
J_1&= \int_{\R^n_+\cap\{|z|=L\}}\bkt{G^x_{ki}(z,t)F^y_j(z)-G_{kj}^y(z,t)F^x_i(z)}\nu_k\,dS_z\\
J_2&=\int_0^t\int_{\R^n_+\cap\{|z|=L\}}\bkt{G^x_{ki}(z,t-\tau)\na_zG^y_{kj}(z,\tau)-G^y_{kj}(z,\tau)\na_zG^x_{ki}(z,t-\tau)}\cdot\nu_z\,dS_zd\tau\\
J_3&=\int_0^t\int_{\R^n_+\cap\{|z|=L\}} \bkt{G^x_{ki}(z,t-\tau)\widehat w^y_j(z,\tau)-G^y_{kj}(z,\tau)\widehat w^x_i(z,t-\tau)}
\nu_k\,dS_zd\tau
}
vanish as $L\to\infty$.
By \eqref{Green_est},
\[
|J_1| \lec \int _{\R^n_+\cap\{|z|=L\}} L^{-n}\,L^{1-n} \,dS_z = CL^{-n}\to 0.
\]
For $J_2$ with $L>2(|x|+|y|+\sqrt t)$, the worst estimate of $\na _zG_{kj}^y(z,\tau)$ by \eqref{Green_est} is $L^{-n} (z_n+\sqrt \tau)^{-1} \log \frac L{\sqrt \tau}$. Thus
\[
|J_2| \lec\int_0^t \int _{\R^n_+\cap\{|z|=L\}}L^{-n}\,L^{-n} \tau^{-1/2} (\log L + |\log \tau| \mathbbm 1_{\tau<1} )\,dS_z d\tau\lesssim \sqrt{t}\,L^{-(n+1)}\log L \to 0.
\]

For the integral $J_3$, by \eqref{pressure_est} with $r=\min(x_n,y_n)>0$,
\[
|J_3| \lec
	\int_0^t \int _{\R^n_+\cap\{|z|=L\}} L^{-n}  \tau^{-1/2} \bkt{\frac1{L^{n}} \log\frac{L}{z_n}  + \frac1{L^{n-1}r}}
\,dS_z d\tau.
\]
Using
\[
  \int _{|z|=L,\, z_n<1} |\log z_n| \,dS_z
  =
  \int_0^1\!\int_{|z'|={\sqrt{L^2-z_n^2}}} |\log z_n| \, dS_{z'} dz_n
  \lesssim
  \int_0^1L^{n-2} \, |\log z_n|\,dz_n
  \lec L^{n-2},
\]
we get
\[
|J_3| \lec \sqrt t \bke{L^{-n-1}\log L + L^{-n-2} + L^{-n} r^{-1} }\to 0,\quad\text{as } L \to \infty.
\]

We also need to show the boundary integrals similar to $J_1$, $J_2$  and $J_3$ at $z_n=1/L$ (instead of $|z|=L$) vanish as $L \to \infty$. This is clear for $J_1$ and $ J_2 $ as $G_{ki}^x(z',0,t)=0$ and the factors $F_j^y$ and $\nb_z G^y_{kj}$ are bounded near $z_n=0$.
For $J_3$, estimate \eqref{pressure_est} of the factor $\widehat w_j^y$ has a log singularity $\log z_n$, and we use the boundary vanishing estimate \eqref{eq_thm3_al_yn} of $G_{ki}^x$, %
\EQN{
|J_3|\lec &\int_0^{t} \int_{z_n = 1/L} \frac{z_n \log\bke{e+\frac{|x^*-z|}{\sqrt{t-s}} }}{\sqrt{t-\tau}(|z'-x'|+|z_n-x_n|+\sqrt{t-\tau})^n}\\
&\cdot\tau^{-\frac12}\,\frac1{(|z'-y'|+z_n+y_n+\sqrt{\tau})^n}\,\log\bke{1+\frac{|z'-y'|+y_n+\sqrt{\tau}}{z_n}} dS_zd\tau,
}
which vanishes as $L\to \infty$. Note that the proof of the base case (no derivatives) of \eqref{eq_thm3_al_yn}, to be given in \S\ref{sec8}, does not rely on the symmetry.%

The above show \eqref{sym_1}.

Now take $\ep\to0$. Using \eqref{Gx_weak_form} and \eqref{Gy_weak_form}, the identity \eqref{sym_1} becomes
\EQS{\label{sym_2}
\lim_{\ep\to0_+}\sum_{k=1}^n&\left[-\int_{|z-y|<\ep}G^x_{ki}(z,t-\ep)G^y_{kj}(z,\ep)\,dz+\int_{|z-x|<\ep}G^x_{ki}(z,\ep)G^y_{kj}(z,t-\ep)\,dz\right.\\
&-\int_0^\ep\int_{|z-y|=\ep}G^y_{kj}(z,\tau){\widehat w^x_i} (z,t-\tau)\nu_k\,dS_zd\tau\\
&\left.+\int_{t-\ep}^t\int_{|z-x|=\ep}G^x_{ki}(z,t-\tau)\widehat w^y_j (z,\tau)\nu_k\,dS_zd\tau\right]
-G^x_{ji}(y,t) +G^y_{ij}(x,t) =0.
}
The first two terms tend to zero as $\ep\to0_+$ by the same reason as for \eqref{sym_lim_y}. Moreover, since $w^x_i(z,t-\tau)$ is uniformly bounded (independent of $\ep$) for $(z,\tau)\in\{(z,\tau):|z-y|=\ep,\,0<\tau<\ep\}$ by \eqref{pressure_est},
we obtain from \eqref{Green_est} that
\EQS{\label{sym_lim_2}
&\left|\int_0^\ep\!\int_{|z-y|=\ep}G^y_{kj}(z,\tau){\widehat w^x_i}(z,t-\tau)\nu_k\,dS_zd\tau\right|\lesssim \int_0^\ep\!\int_{|z-y|=\ep}\frac1{(|z-y|+\sqrt{\tau})^{n}}\,dS_zd\tau\\
&\lesssim\int_0^\ep\frac1{(\ep+\sqrt{\tau})^{n}}\,\ep^{n-1}\,d\tau
\lesssim\ep + \de_{n2} \ep\log \frac 1\ep \to0\ \ \ \text{ as }\ep\to0_+.
}
Similarly, $\int_{t-\ep}^t\int_{|z-x|=\ep}G^x_{ki}(z,t-\tau)w^y_j(z,\tau)\nu_k\,dS_zd\tau$ goes to zero as $\e\to0_+$.
By \eqref{sym_lim_y} and \eqref{sym_lim_2}, the equation \eqref{sym_2} turns into
\EQN{
-G^x_{ji}(y,t)+G^y_{ij}(x,t)=0.
}
This completes the proof of \thref{prop1}, i.e., the symmetry \eqref{Green.symmetry} of the Green tensor.
\end{proof}

\begin{remark}\label{rem7.1}
We can actually show an alternative estimate of $\widehat w_j^y(z,\tau)$ which has no singularity as $z_n \to 0_+$ by estimating \eqref{0715a} instead of \eqref{w-formula}, cf.~Remark \ref{rem3.5}(i). Using it, we don't need the vanishing estimate \eqref{eq_thm3_al_yn}.
We do not present it in this way since its proof is more involved, in particular in the case $n=2$.
\end{remark}

\section{The main estimates of the Green tensor}
\label{sec8}
In this section we prove the main estimates in \thref{thm3} and \thref{thm4}.

\begin{proof}[Proof of \thref{thm3}]
From \eqref{Green_est} we have that
\EQS{\label{main_eq1}
&|\pd_{x',y'}^l\pd_{x_n}^k\pd_{y_n}^q\pd_t^mG_{ij}(x,y,t)|\lesssim\frac1{(|x-y|^2+t)^{\frac{l+k+q+n}2+m}}\\
&~~~~+\frac{\LN_{ijkq}^{mn}}{t^{m}(|x^*-y|^2+t)^{\frac{ l+k-k_i+n }2} (x_n^2+t)^{\frac{k_i}2} (y_n^2+t)^{\frac{q}2}},
}
where $k_i = (k - \de_{in})_+$,
and
\[
\LN_{ijkq}^{mn} := 1+ \de_{n2}\mu_{ik}^m\bkt{\log(\nu_{ijkq}^m|x'-y'|+x_n+y_n+\sqrt{t}) - \log(\sqrt{t})},
\]
\[
\mu_{ik}^m=1-(\de_{k0}+\de_{k1}\de_{in})\de_{m0},\quad
\nu_{ijkq}^m =  \de_{q0} \de_{jn} \de_{k(1+\de_{in})} \de_{m0}+\de_{m>0}.
\]
On the other hand, by the symmetry of the Green tensor (\thref{prop1}) and \eqref{Green_est},
\EQS{\label{main_eq2}
&|\pd_{x',y'}^l\pd_{x_n}^k\pd_{y_n}^q\pd_t^mG_{ij}(x,y,t)|=|\bke{\pd_{x',y'}^l\pd_{Y_n}^k\pd_{X_n}^q\pd_t^mG_{ji}}(y,x,t)|\\
&\lesssim\frac1{(|x-y|^2+t)^{\frac{l+k+q+n}2+m}}
+\frac{\LN_{jiqk}^{mn}}{t^{m}(|x^*-y|^2+t)^{\frac{l+q-q_j+n}2 } (y_n^2+t)^{\frac{q_j}2} (x_n^2+t)^{\frac{k}2}},
}
where $q_j = (q-\de_{jn})_+$, $\pd_{X_n}$ denotes the partial derivative in the $n$-th variable, and $\pd_{Y_n}$ denotes the partial derivative in the $2n$-th variable.
The combination of \eqref{main_eq1} and \eqref{main_eq2} gives
\EQN{
|\pd_{x',y'}^l \pd_{x_n}^k \pd_{y_n}^q \pd_t^m G_{ij}(x,y,t)|&\lesssim\frac1{(|x-y|^2+t)^{\frac{l+k+q+n}2+m}}
\\&+\frac{\LN_{ijkq}^{mn}+\LN_{jiqk}^{mn}}
{t^{m}(|x^*-y|^2+t)^{\frac{l+k-k_i+q-q_j+n}2}(x_n^2+t)^{\frac{k_i}2}(y_n^2+t)^{\frac{q_j}2} }.
}
This shows \eqref{eq_Green_estimate} and completes the proof of \thref{thm3}.
\end{proof}

We next show the boundary vanishing of derivatives of $G_{ij}$ at $x_n=0$ or $y_n=0$.
\begin{proof}[Proof of \thref{thm4}]
Denote
\[
\LN= {\textstyle\sum_{k=0}^1}( \LN_{ijkq}^{mn} + \LN_{jiqk}^{mn})(x,y,t) .
\]
By $\pd_{x',y'}^l\pd_{y_n}^q\pd_t^mG_{ij}|_{x_n=0}=0$ and \eqref{eq_Green_estimate} with $k=1$, %
we have
\EQS{\label{eq_thm3_pf}
& \left|\pd_{x',y'}^l\pd_{y_n}^q \pd_t^mG_{ij}(x,y,t)\right|\le \int_0^{x_n}\left|\pd_{x',y'}^l\pd_{x_n}\pd_{y_n}^q\pd_t^mG_{ij}(x',z_n,y,t)\right|\,dz_n  \\
& \lesssim \int_0^{x_n}\left[\frac1{(|x'-y'|^2+|z_n-y_n|^2+t)^{\frac{l+q+n+1}2+m}}\right.\\
&\quad \quad \left.+\frac{\LN}{t^{m}(|x'-y'|^2+(z_n+y_n)^2+t)^{\frac{l+q-q_j+n}2 } (z_n^2+t)^{\frac12 } (y_n^2 + t)^{\frac{q_j}2} }\right]\,dz_n
=: I_1 + I_2.
}
Above we have used that $\LN_{ijkq}^{mn}(x',z_n,y,t)$ is nondecreasing in $z_n$.

\medskip

We first estimate $I_1$.

Case 1.
If $3x_n<y_n$, then $|z_n-y_n|>\frac12 (x_n + y_n)$ and $z_n+y_n>\frac14(x_n+y_n)$ for $0<z_n<x_n$. Thus,  \eqref{eq_thm3_pf} gives
\EQN{
I_1
&\lesssim \frac{x_n}{(|x-y^*|^2+t)^{\frac{l+q+n+1}2+m}}.
}

Case 2.
If $y_n<3x_n < \frac12 \bke{|x'-y'|+y_n+\sqrt{t}}$, then $x_n+y_n<\frac43(|x'-y'|+\sqrt t)$, which implies $|x-y^*|+\sqrt t\lec |x'-y'|+\sqrt t$.
We drop $|z_n-y_n|$ in the integrand of \eqref{eq_thm3_pf} to get
\EQN{
I_1
\lesssim\frac{x_n}{(|x'-y'|^2+t)^{\frac{l+q+n+1}2+m}}
\lesssim\frac{x_n}{(|x-y^*|^2+t)^{\frac{l+q+n+1}2+m}}.
}

Case 3.
If $3 x_n > y_n > \frac12 \bke{|x'-y'|+y_n+\sqrt{t}}$ or $3 x_n > \frac12 \bke{|x'-y'|+y_n+\sqrt{t}} > y_n$, then $x_n\approx |x-y^*|+\sqrt{t}$. By \eqref{eq_Green_estimate} with $k=0$,
\EQN{
I_1
\lesssim\frac1{(|x-y|^2+t)^{\frac{l+q+n}2+m}}
\lesssim\frac{x_n}{(|x-y|^2+t)^{\frac{l+q+n}2+m}(|x-y^*|^2+t)^{\frac12}} .
}
Thus, we have
\[
I_1
\lesssim\frac1{(|x-y|^2+t)^{\frac{l+q+n}2+m}}
\lesssim\frac{x_n}{(|x-y|^2+t)^{\frac{l+q+n}2+m}(|x-y^*|^2+t)^{\frac12}}.
\]

\medskip

Next, we estimate $I_2$.

If $x_n< y_n+ \frac12 \sqrt{t}$, then $|x-y^*|^2+t\approx|x'-y'|^2+y_n^2+t$.
We drop $z_n$ in the integrand of \eqref{eq_thm3_pf} to get
\EQN{
I_2
&\lesssim \frac{x_n\,\LN }{t^{m+\frac12}(|x'-y'|^2+y_n^2+t)^{ \frac{l+q-q_j+n}2 } (y_n^2+t)^{ \frac{q_j}2}}
\lesssim \frac{x_n\,\LN }{t^{m+\frac12}(|x-y^*|^2+t)^{ \frac{l+q-q_j+n}2 } (y_n^2+t)^{\frac{q_j}2}}.
}

If $x_n>y_n+ \frac12\sqrt{t}$, then $x_n\gtrsim\sqrt{t}$. By \eqref{eq_Green_estimate} with $k=0$,
\EQN{
I_2
&\lesssim \frac{\LN}
{t^{m}(|x-y^*|^2+t)^{ \frac{l+q-q_j+n}2 } (y_n^2+t)^{\frac{q_j}2}}
\lesssim \frac{x_n\,\LN}
{t^{m+\frac12}(|x-y^*|^2+t)^{ \frac{l+q-q_j+n}2 }(y_n^2+t)^{\frac{q_j}2}}.
}

Combining the above cases, we derive
\EQN{
\left|\pd_{x',y'}^l\pd_{y_n}^q\pd_t^mG_{ij}(x,y,t)\right|
&\lesssim\frac{x_n}{(|x-y|^2+t)^{\frac{l+q+n}2+m}(|x-y^*|^2+t)^{\frac12}}
\\
&\quad +\frac{x_n\,\LN}
{t^{m+\frac12}(|x-y^*|^2+t)^{ \frac{l+q-q_j+n}2 }(y_n^2+t)^{\frac{q_j}2}},
}
which is \eqref{eq_thm3_al_yn} for $\al=1$. Since \eqref{eq_thm3_al_yn} also holds for $\al=0$ by \eqref{eq_Green_estimate},  it holds for all $0\le \al\le 1$. Finally, \eqref{eq_thm3_al_xn} follows from the symmetry. This completes the proof of \thref{thm4}.
\end{proof}

\section{Mild solutions of Navier-Stokes equations}
In this section we apply our linear estimates to the construction of mild solutions of Navier-Stokes equations \eqref{NS}.

\subsection{Mild solutions in $L^q$}


In this subsection we prove Lemma \ref{th9.1}. It is standard to prove Theorem \ref{thm5} using
 estimates in  Lemma \ref{th9.1}
 and a  fixed point argument. We skip the proof of Theorem \ref{thm5}.

\begin{lem}\thlabel{th9.1}
Let $n \ge 2$, $1\le p\leq q\le \infty$ and $1<q$.

(a) If $u_0 \in L^p_\si(\R^n_+)$ and $\breve u_i(x,t)=\sum_{j=1}^n\int _{\R^n_+} \breve G_{ij}(x,y,t) u_{0,j}(y) dy$, then
\EQ{\label{E10.5}
\norm{\breve u(\cdot,t) }_{L^q(\R^n_+)} \le C  t^{-\frac{n}{2}(\frac{1}{p}-\frac{1}{q})}\norm{u_0}_{L^p(\R^n_+)}, \quad \text{if } u_0 = \bP u_0,
}
\EQ{\label{E10.5a}
L^q\text{-}
\lim_{t \to 0_+} t^{\frac{n}{2}(\frac{1}{p}-\frac{1}{q})} \breve u(\cdot,t) =
\left\{
\begin{aligned}
0 , \quad \text{if } 1\le p<q\le \I,
\\
u_0, \quad \text{if } 1<p=q<\I.
\end{aligned}\right.
}
\eqref{E10.5a}$_2$ is also valid for $p=q=\I$ if $u_0 $ in the $L^\infty$-closure of $C^1_{c,\si}(\overline{\R^n_+})$.

(b) Let $F \in L^p(\R^n_+)$, $a,b\in \NN_0$, and $1 \le a+b$. Assume $b\ge1$ and $n\ge3$ if $p=q=\infty$. Then
\EQ{\label{E10.6}
\norm{\int _{\R^n_+} \pd_x^a \pd_{y}^b G_{ij}(x,y,t) F(y) dy}_{L^q(\R^n_+)} \le C t^{-\frac{a+b}{2}-\frac{n}{2}(\frac{1}{p}-\frac{1}{q})} \norm{F}_{L^p(\R^n_+)}. %
}
\end{lem}

\begin{proof}
We consider \eqref{E10.5} and decompose $\breve u_i(x,t)$ defined in \eqref{breve-u-def} as
\EQN{
\breve u_i(x,t) &=\int _{\R^n_+} \Ga(x-y,t) u_{0,i}(y) dy + \int _{\R^n_+} G_{ij}^*(x,y,t) u_{0,j}(y) dy = :
u^{heat}_i(x,t) + u_i^*(x,t).
}

The basic property of heat kernel yields
$$
\|u^{heat}\|_{L^{q}(\R_{+}^{n})}\leq C  t^{\frac{n}{2}(\frac{1}{q}-\frac{1}{p})} \norm{u_0}_{L^p(\R_{+}^{n})}.
$$

By \eqref{Solonnikov.est}, $u^*(x,t)$ is bounded by
\EQN{
J_t(x)&=  \int _{\R^n_+}\frac{e^{-\frac{cy_n^2}t}} {(|x'-y'| + x_n+y_n+\sqrt t)^n} |u_0(y)| \,dy
\\
&= \int_0^\I
\frac{1} {(|x'| + x_n+y_n+\sqrt t)^n} *_{\Si} |u_0(x',y_n)| e^{-\frac{cy_n^2}t}\,dy_n,
}
where $*_\Si$ indicates convolution over $\Si$. By Minkowski and Young inequalities,
\EQN{
\norm{J_t(\cdot,x_n)}_{L^q(\Si)} &\lec \int_0^\I \norm{
\frac{1} {(|x'| + x_n+y_n+\sqrt t)^n} *_{\Si} |u_0(x',y_n)|}_{L^q(\Si)}  e^{-\frac{cy_n^2}t}\,dy_n
\\
&\lec \int_0^\I \norm{
\frac{1} {(|x'| + x_n+y_n+\sqrt t)^n}}_{L^r(\Si)} \cdot  \norm{ u_0(\cdot,y_n)}_{L^p(\Si)}  e^{-\frac{cy_n^2}t}\,dy_n,\quad \frac1q + 1 = \frac1r + \frac1p,
\\
&\lec \int_0^\I
\frac{1} { (x_n+y_n+\sqrt t)^{1-(n-1)\bke{\frac1q - \frac1p}}} \cdot  \norm{ u_0(\cdot,y_n)}_{L^p(\Si)}  e^{-\frac{cy_n^2}t}\,dy_n.
}
By Minkowski inequality again, (here we need $q>1$)
\begin{align}
\norm{J_t}_{L^q(\R^n_+)} &=\norm{ \norm{J_t(\cdot,x_n)}_{L^q(\Si)}}_{L^q(0,\I)} \notag
\\
&\lec \int_0^\I \norm{
\frac{1} { (x_n+y_n+\sqrt t)^{1-(n-1)\bke{\frac1q - \frac1p}}} }_{ L^q(x_n\in(0,\I))} \cdot  \norm{ u_0(\cdot,y_n)}_{L^p(\Si)}  e^{-\frac{cy_n^2}t}\,dy_n \notag
\\
&\lec \int_0^\I
\frac{1} {(y_n+\sqrt t)^{1-\frac1q - (n-1)\bke{\frac1q - \frac1p}} } \cdot  \norm{ u_0(\cdot,y_n)}_{L^p(\Si)}  e^{-\frac{cy_n^2}t}\,dy_n.\label{0912a}
\end{align}
By H\"older inequality,
\EQN{
\norm{J_t}_{L^q(\R^n_+)} &\lec \norm{ u_0}_{L^p(\R^n_+)} \bke{ \int_0^\I
\bke{\frac{1} {(y_n+\sqrt t)^{1-\frac1q - (n-1)\bke{\frac1q - \frac1p}} }\, e^{-\frac{cy_n^2}t}} ^{\frac p{p-1}}\,dy_n}^{\frac {p-1}p}\lec
t^{\frac{n}2\bke{\frac1q - \frac1p}} \norm{ u_0}_{L^p(\R^n_+)}
}
by the change of variables $y_n = \sqrt t z$. This proves \eqref{E10.5}.

\medskip

For \eqref{E10.5a}, denote $\si = \frac n2(\frac 1p-\frac 1q)$.
 If $1\le p < q \le \infty$, then $\si>0$.
For any $\e>0$, we can choose $b \in L^p_\si \cap L^q_\si$ with $\norm{u_0 - b}_{L^p} \le \e$. Let $v_i(x,t)=\sum_{j=1}^n\int _{\R^n_+} \breve G_{ij}(x,y,t) b_{j}(y) dy$.
Then by \eqref{E10.5},
\EQN{
t^\si \norm{\breve u(\cdot,t)}_{L^q}&\le t^\si \norm{v(\cdot,t)}_{L^q} + t^\si \norm{\breve u(\cdot,t)-v(\cdot,t)}_{L^q}
\\
&\lec t^\si \norm{b}_{L^q} +  \norm{u_0-b}_{L^p}
}
which is less than $C\e$ for $t$ sufficiently small. This shows \eqref{E10.5a}$_1$.

If $1<p=q<\infty$, For any $\e>0$, there is $M>0$ such that $\big\|\norm{ u_0(\cdot,y_n)}_{L^q(\Si)} \one_M\big\|_{L^q(0,\infty)} \le \e$, where
$\one_{M}(y_n) = 1$ if $\norm{ u_0(\cdot,y_n)}_{L^q(\Si)}\ge M$, and $\one_{M}(y_n) = 0$ otherwise. Then
\[
\norm{ u_0(\cdot,y_n)}_{L^q(\Si)} \le M + \norm{ u_0(\cdot,y_n)}_{L^q(\Si)}\one_M.
\]
Applying H\"older inequality to \eqref{0912a},
\EQN{
\norm{J_t}_{L^q(\R^n_+)} &\lec \big\|\norm{ u_0(\cdot,y_n)}_{L^q(\Si)} \one_M\big\|_{L^q(0,\infty)}
+  \int_0^\I
\frac{1} {(y_n+\sqrt t)^{1-1/q} } \cdot  M  e^{-\frac{cy_n^2}t}\,dy_n\\
&\lec \e + M t^{1/2q}
}
which is bounded by $C\e$ for $t$ sufficiently small. Since $u^{heat}(\cdot,t) \to u_0$ in $L^q$ as $ t\to 0_+$, this shows $\breve u^{L}(\cdot,t) \to u_0$ in $L^q$ as $ t\to 0_+$. This shows \eqref{E10.5a}$_2$.

If $p=q=\I$ and $u_0 $ in the $L^\infty$-closure of $C^1_{c,\si}(\overline{\R^n_+})$,
for any $\e>0$, we can choose $b \in C^1_{c,\si}(\overline{\R^n_+})$ with $\norm{u_0 - b}_{L^\infty} \le \e$. Let $v_i(x,t)=\sum_{j=1}^n\int _{\R^n_+} \breve G_{ij}(x,y,t) b_{j}(y) dy$. By Lemma \ref{th6.2}, $\lim_{t \to 0_+}\norm{v_i(\cdot,t) - b}_{L^\infty(\R^n_+)}=0$.
Then by \eqref{E10.5},
\[
 \norm{\breve u(\cdot,t)-u_0}_{L^\infty}\le  \norm{\breve u(\cdot,t)-v(\cdot,t)}_{L^\infty} + \norm{v(\cdot,t) -b}_{L^\infty} + \norm{b-u_0}_{L^\infty}
\lec  \e + o(1),
\]
which is less than $C\e$ for $t$ sufficiently small. This shows the remark after \eqref{E10.5a}$_2$.

\medskip

For \eqref{E10.6}, denote
 $$
 w(x,t) = \int _{\R^n_+} \pd_{x}^a\pd_{y}^b  G_{ij}(x,y,t) F(y) dy ,\quad m=a+b.
 $$
By Theorem \ref{thm3},
\[
|\pd_{x}^a\pd_{y}^b G_{ij}(x,y,t)|\lesssim \frac1{(|x-y|^2+t)^{\frac{n+m}2}}
+\frac{1+ \de_{n2} \log (1+ \frac {|x^*-y|}{\sqrt t})}
{(|x^*-y|^2+t)^{\frac{n}2}  (x_n^2+t)^{\frac{a}2} (y_n^2+t)^{\frac{b}2} }.
\]
Using 
\[
\frac {\log(e+r)}{(e+r)^n} \le \frac {\log(e+s)}{(e+s)^n} , \quad \forall 0\le s\le r,
\]
we have
\[
|\pd_{x}^a\pd_{y}^b G_{ij}(x,y,t)|
\lesssim \frac1{(|x-y|^2+t)^{\frac{n+m}2}}
+ \frac{\de_{n2}\log (e+ \frac {|x-y|}{\sqrt t})} {(|x-y|^2+t)^{\frac{n}2}  (x_n^2+t)^{\frac{a}2} (y_n^2+t)^{\frac{b}2} }.
\]
Extend $F(y)$ to $y \in \R^n$ by zero for $y_n<0$.
We have
\EQS{\label{1015a}
|w(x,t)|&\lec \int_{\R^n} H_t^0(x-y) |F(y)|dy + \int_{\R^n} H_t(x-y^*) |F(y)|\, \frac1{(y_n^2+t)^{\frac{b}2}}\, dy\, \frac1{(x_n^2+t)^{\frac{a}2}}\\
&:= w_1(x,t) + w_2(x,t),
}
where
\EQ{\label{1015b-0}
H_{t}^0(x) = t^{-\frac {n+m}2} H_1^0\bke{\frac x {\sqrt t}},\quad
H_1^0(x) = \frac1 {(|x|^2+1)^{\frac{n+m}2}} \in L^1\cap L^\infty(\R^n),
}
\EQ{\label{1015b}
H_{t}(x) = t^{-\frac {m}2} H_1\bke{\frac x {\sqrt t}},\quad
H_1(x) = \frac{\de_{n2}\log (e+ |x|)} {(|x|^2+1)^{\frac{n}2} } .
}
By Young's convolution inequality with $\frac1q=\frac1r+\frac1p-1$,
\EQN{
\|w_1(\cdot,t)\|_{L^q} \lec  \norm{H_t^0}_{L^r(\R^n)}  \norm{F}_{L^p}
=
t^{-\frac{m}{2}+\frac{n}{2}(\frac{1}{q}-\frac{1}{p})} \norm{H_1^0}_{L^r(\R^n)}  \norm{F}_{L^p}.
}

It remains to estimate $w_2(\cdot,t)$.

If $p<q$, we drop the factors $(y_n^2+t)^{-\frac{b}2}$ and $(x_n^2+t)^{-\frac{a}2}$ in \eqref{1015a}, $H_t(x-y^*)$ by $H_t(x-y)$, and applying Young's convolution inequality with $\frac1q = \frac1r + \frac1p - 1$ to get
\[
\norm{ t^{-\frac{m}2} \int_{\R^n} H_t(x-y) |F(y)|\, dy}_{L^q} \lec t^{-\frac{m}2} \norm{H_t}_{L^r(\R^n)} \norm{F}_{L^p}
\lec t^{-\frac{m}2 + \frac{n}2\bke{\frac1q - \frac1p}} \norm{H_1}_{L^r(\R^n)} \norm{F}_{L^p}.
\]
Note that $H_1\in L^r$ since $r>1$ when $p<q$.
Thus, we get for $p<q$ that
\[
\norm{w_2(\cdot,t)}_{L^q} \lec t^{-\frac{m}2 + \frac{n}2\bke{\frac1q - \frac1p}} \norm{F}_{L^p}.
\]

If $p=q=\infty$, by the hypotheses $b\ge1$ and $n\ge3$ so there is no log term in \eqref{1015b}.
In this case,
\EQN{
w_2(x,t)  
&\lec \int_{\R^n_+} H_t(x-y^*) |F(y)|\, \frac1{(y_n^2 + t)^{\frac{b}2}}\, dy\, \frac1{(x_n^2 + t)^{\frac{a}2}}\\
&\lec \norm{F}_{L^\infty}\, \frac1{(x_n + t)^{\frac{a}2}} \int_{\R^n_+} \frac1{(|x-y^*|^2 + t)^{\frac{n}2} (y_n^2 + t)^{\frac{b}2}}\, dy\\
&\lec \norm{F}_{L^\infty}\, \frac1{(x_n + t)^{\frac{a}2}} \int_0^\infty \frac1{(x_n^2 + y_n^2 + t)^{\frac12} (y_n^2 + t)^{\frac{b}2}}\, dy_n
\lec t^{-\frac{m}2} \norm{F}_{L^\infty}.
}
This proves \eqref{E10.6}.
\end{proof}

\begin{remark}
Let $1\le p<q\le\I$ and $u_0 \in L^p_\si(\R^n_+)$. We claim that
\EQ{\label{eq1020b}
 u^L_i (x,t)= \int_{\R^n_+} G_{ij}(x,y,t) u_{0,j}(y)\, dy,\quad
 \widehat  u^L_i (x,t)=  \int_{\R^n_+} \widehat G_{ij}(x,y,t) u_{0,j}(y)\, dy,
}
are also defined in $L^q(\R^n_+)$ for fixed $t>0$ and \eqref{E10.5} holds for $u^L$ and $\widehat u^L$:
\EQ{\label{eq1020a}
\norm{u^L(\cdot,t) }_{L^q(\R^n_+)} + \norm{\widehat u^L(\cdot,t) }_{L^q(\R^n_+)}
 \le C  t^{-\frac{n}{2}(\frac{1}{p}-\frac{1}{q})}\norm{u_0}_{L^p(\R^n_+)}.
}
Our claim does not include the case $p=q$ as in \eqref{E10.5}.  For $u^L$, this is because $|G_{ij}(x,y,t)|\lec (|x-y|+\sqrt{t})^{-n}$ and, by Young's convolution inequality with $1+\frac1{q} = \frac1r + \frac1p$,
\EQN{
\norm{u^L(\cdot,t) }_{L^{q}} \lec \norm{ (|x|+\sqrt{t})^{-n} * u_0}_{L^q}
\lec \Big(\int_{\R^n}(|x|+\sqrt{t})^{-nr}dx\Big)^{\frac{1}{r}}\norm{u_0 }_{L^p}\lec t^{-\frac{n}{2}(\frac{1}{p}-\frac{1}{q})}\norm{u_0 }_{L^p}
}
where we used $q>p$ so that $r>1$. For $\widehat u^L$, by \eqref{0827a}, we can decompose
\EQN{
\widehat u_i^L(x,t) &=\int _{\R^n_+} \Ga(x-y,t) u_{0,i}(y) dy + \int _{\R^n_+} \widehat G_{ij}^*(x,y,t) u_{0,j}(y) dy = :
u^{heat}_i(x,t) + \widehat u_i^*(x,t),
}
where $\widehat G_{ij}^*(x,y,t) = -\de_{ij} \Ga(x-y^*,t) - 4 \de_{jn}C_i(x,y,t)$.
The first term $u^{heat}$ satisfies \eqref{eq1020a} by the basic property of heat kernel. The second term
 $\widehat u^*(x,t)$ is bounded by
\[
|\widehat u^*(x,t)|\lec \int _{\R^n_+}\frac{e^{-\frac{cy_n^2}t}} {(|x'-y'| + x_n+y_n+\sqrt t)^{n-1} (y_n+\sqrt t)} |u_0(y)| \,dy
\]
using \eqref{C-estimate}. Similar to the proof of \eqref{E10.5}, we can first apply
Minkowski and Young inequalities in $x'$ (using $q>p$ so that $r>1$), and then Minkowski and H\"older inequalities in $x_n$ to bound
$\norm{\widehat u^*(\cdot,t)}_{L^q(\R^n_+)}$ by $t^{\frac{n}2\bke{\frac1q - \frac1p}} \norm{ u_0}_{L^p(\R^n_+)} $.  The above shows \eqref{eq1020a} for $1\le p<q\le\I$ and $u_0 \in L^p_\si(\R^n_+)$.
\end{remark}

\begin{remark}
\label{rem9.2}
The following extends Theorem \ref{th6.1}.
Assume $u_0\in L^p_{\si}(\R^n_+)$, $1\le p<\infty$, and
$u^L$, $\breve u^L$ and $\widehat u^L$ are defined as in \eqref{eq1020b}.
There exist $u^k_0\in C^1_{c,\si}(\overline{\R^n_+})$ such that $u^k_0\to u_0$ in $L^p(\R^n_+)$ as $k\to\I$.
Let
\[
 u^k_i (x,t)= \int_{\R^n_+} G_{ij}(x,y,t) u_{0,j}^k(y)\, dy,\quad
 \breve u^k_i (x,t)=  \int_{\R^n_+} \breve G_{ij}(x,y,t) u_{0,j}^k(y)\, dy,
\]
\[
\widehat u^k_i (x,t)= \int_{\R^n_+} \widehat G_{ij}(x,y,t) u_{0,j}^k(y)\, dy.
\]
They are equal by Theorem \ref{th6.1} since $u_{0}^k\in C^1_{c,\si}(\overline{\R^n_+})$. On the other hand, by
\eqref{E10.5} and \eqref{eq1020a},
\[
\norm{u^L(t) -  u^k(t)}_{ L^q}
+\norm{\breve u^L(t) -  \breve u^k(t)}_{ L^q}
+\norm{\widehat  u^L(t) - \widehat  u^k(t)}_{ L^q}
\le C   t^{-\frac{n}{2}(\frac{1}{p}-\frac{1}{q})} \norm{u_{0}-u_{0}^k}_{L^p}
\]
which vanishes as $k\to\infty$. This shows $u^L(t) = \breve u^L(t) =\widehat  u^L(t) $ in $L^q$ for any $q\in(p, \infty]$ and fixed $t$.

For $u_0\in L^\infty_{\si}(\R^n_+)$ and $u_0$ in the $L^\infty$-closure of $C^1_{c,\si}(\overline{\R^n_+})$, we can also show $u^L(t) = \breve u^L(t)$ (but we do not know about  $\widehat  u^L(t) $).
We use the boundary vanishing \eqref{eq_thm3_al_yn} to get
\EQN{
|u^L(x,t) - u^k(x,t)|
&=\abs{\int_{\R^n_+} G_{ij}(x,y,t) \bke{u_{0,j}(y) - u_{0,j}^k(y)} dy} \le C_1 \norm{u_0 - u_0^k}_{L^\infty}
}
where
\begin{align}\notag
C_1&= \int_{\R^n_+} \frac{x_n}{(|x-y|+\sqrt t)^n (x_n+y_n+\sqrt t)}\, dy\\ \label{C1.def}
& \lec\int_0^\infty \frac{x_n}{(|x_n-y_n|+\sqrt t) (x_n+y_n+\sqrt t)}\, dy_n\\ \notag
&  \lec\int_0^{2(x_n+\sqrt t)}\frac{x_n}{(|x_n-y_n|+\sqrt t) (x_n+\sqrt t)}\, dy_n  + \int _{2(x_n+\sqrt t)}^ \infty \frac{x_n}{y_n^2}\, dy_n \lec  \ln (e+\frac {x_n}{\sqrt t}).
\end{align}
It converges to 0
as $k\to\infty$, and the convergence is uniform in $x_n\le M \sqrt t$ for any fixed $t,M>0$. As $u^k (t)\to \breve u^L(t)$ in $L^\infty$ by \eqref{E10.5},
this shows $u^L(x,t)=\breve u^L(x,t)$.
\hfill $\square$
\end{remark}

\subsection{Mild solutions with pointwise decay}

In this subsection we prove Theorems \ref{thm6} and \ref{thm1.8}. We first consider Theorem \ref{thm6}. Recall Theorem \ref{thm6} is a direct consequence of \cite[Theorem 1]{CJ-NA2017} using the estimates in \cite[Theorem 1]{CJ-JDE2017} for $0<a<n$. For $a=n$, the hypothesis of \cite[Theorem 1]{CJ-NA2017} is not satisfied: $(1+|x|+\sqrt t)^n e^{-tA} u_0 \sim \log(2+t) \not \in L^\infty(\R^n_+\times (0,\infty))$ (see \cite[Theorem 1]{CJ-JDE2017} and \eqref{1001b}). Nonetheless, the proof of local existence still works if $\norm{(1+|x|+\sqrt t)^n e^{-tA} u_0}_{L^\infty(\R^n_+\times(0,T))}\le C(T)$, which is true for $u_0\in Y_n$. Theorem \ref{thm6} can be proved using the estimates in Lemma \ref{1001a} below and the same iteration argument in \cite{CJ-NA2017}. We omit its proof and focus on Lemma \ref{1001a}.

\begin{lem}\label{1001a}
Let $n\ge 2$ and $0 \le a \le n$. %
For $u_0\in Y_a$ with $\div u_0 =0$ and $u_{0,n}|_\Si=0$,
\EQ{\label{1001b}
\norm{\sum_{j=1}^n \int _{\R^n_+} \breve G_{ij}(x,y,t) u_{0,j}(y) dy }_{Y_a} \le C (1 + \delta_{an} \log_+t)  \norm{u_0}_{Y_a}.
}
For $F\in Y_{2a}$,
\EQ{\label{1001c}
\norm{\int _{\R^n_+} \pd_{y_p} G_{ij}(x,y,t) F_{pj}(y)\,dy }_{Y_a} \le C t^{-1/2} \norm{F}_{Y_{2a}}.
}
\end{lem}

The estimate \eqref{1001b} is proved in \cite[Theorem 1]{CJ-JDE2017} with space-time decay (see also \cite[Theorem 4.2]{CM2004}), whereas \eqref{1001c} is not known in \cite{CM2004} and \cite{CJ-JDE2017} since the pointwise estimates of the Green tensor $G_{ij}$ was not available. Instead, they used \eqref{eq:1015} for the bilinear form in the Duhamel's formula when constructing mild solutions.

Note that $Y_0=L^\infty$ and $a\le n$ in \eqref{1001b} since the decay cannot be faster than the Green tensor. The case $a=0$ is a special case of \eqref{E10.5}. It is similar to \cite[Theorem 1.1]{MR1981302} which further assumes continuity.
We do not assume any boundary condition on $F_{pj}$. Also note
\[
\norm{|u|^2}_{Y_{2a}} = \sup_{x \in \R^n_+} |u(x)|^2\bka{x}^{2a} = \norm{u}_{Y_{a}}^2.
\]

\begin{proof}
If $a=0$, the lemma follows from \eqref{E10.5} and \eqref{E10.6} with $p=q=\infty$.
 Thus, we consider $a>0$.
For \eqref{1001b},
write
\EQN{
\sum_{j=1}^n\int _{\R^n_+} \breve G_{ij}(x,y,t) u_{0,j}(y) dy &=\int _{\R^n_+} \Ga(x-y,t) u_{0,i}(y) dy + \int _{\R^n_+} G_{ij}^*(x,y,t) u_{0,j}(y) dy\\
& = :
u^{heat}_i(x,t) + u_i^*(x,t).
}
It is known that for $0\le a \le n$
\EQ{\label{Knightly}
\norm{u^{heat}}_{Y_a} \lec (1+\de_{an}\log_+ t) \norm{u_0}_{Y_a}.
}
See e.g.~\cite[Lemma 1]{MR191213} for $n=3$ case. Its statement corresponds to $1\le a \le n$ but its proof also works for $0\le a <1$.

For $u^*$ with $|u_0(y)| \lec \bka{y}^{-a}$, by  \eqref{Solonnikov.est},  (for both $n \ge 3$ and $n=2$)
\[
|u^*(x,t)| \lec
J(x)=  \int _{\R^n_+}\frac{e^{-\frac{cy_n^2}t}} {(|x^*-y|^2+t)^{\frac{n}2} \bka{y}^a} \,dy.
\]

Suppose $0< a<n-1$. By Lemma \ref{lemma2.2},
\EQN{
J \lec& \int_0^\infty e^{-\frac{cy_n^2}t} \int_{\Si} \frac1{(|x'-y'|+x_n+y_n+\sqrt{t})^n(|y'|+y_n+1)^a}\, dy'dy_n\\
\lec& \int_0^\infty \bkt{\frac1{(|x|+y_n+\sqrt{t}+1)^{a+1}} + \frac{e^{-\frac{cy_n^2}t}}{(|x|+y_n+\sqrt{t}+1)^a(x_n+y_n+\sqrt{t})}} dy_n\\
\lec& \frac1{(|x|+\sqrt{t}+1)^{a}} + \frac1{(|x|+\sqrt{t}+1)^{a}} \int_0^\infty \frac{e^{-u^2}}{\bke{\frac{x_n}{\sqrt{t}}}+1}\, du\\
\lec& \frac1{(|x|+\sqrt{t}+1)^{a}}.
}
This proves
\[\norm{u^*}_{Y_a} \lec \norm{u_0}_{Y_a},\ \ \ 0< a<n-1.\]

If $a=n-1$, we have an additional term from Lemma \ref{lemma2.2},
\EQN{
&\int_0^\infty e^{-\frac{cy_n^2}t}\, \frac1{(|x|+y_n+\sqrt{t}+1)^{n}}\, \log\bke{1 + \frac{|x|+\sqrt{t}}{y_n+1}} dy_n\\
&\lec \frac1{(|x|+\sqrt{t}+1)^{n}} \int_0^\infty e^{-\frac{cy_n^2}t} \bke{\frac{|x|+\sqrt{t}}{y_n+1}}^\e\, dy_n,
}
where
\EQN{
\int_0^\infty& e^{-\frac{cy_n^2}t} \bke{\frac{|x|+\sqrt{t}}{y_n+1}}^\e\, dy_n\\
& \le
\int_0^{|x|+\sqrt{t}} \bke{\frac{|x|+\sqrt{t}}{y_n+1}}^\e dy_n + \int_{|x|+\sqrt{t}}^\infty e^{-\frac{cy_n^2}t} \bke{\frac{|x|+\sqrt{t}}{y_n+1}}^\e dy_n\\
& \le (|x|+\sqrt{t})^\e (|x|+\sqrt{t}+1)^{1-\e} + \frac{(|x|+\sqrt{t})^\e}{(|x|+\sqrt{t}+1)^{\e}} \int_0^\infty e^{-u^2} \sqrt{t}\, du\\
& \le |x|+\sqrt{t}+1.
}
So the additional term is bounded by $(|x|+\sqrt{t}+1)^{1-n} = (|x|+\sqrt{t}+1)^{-a}$.

If $n-1<a<n$, we have an additional term from Lemma \ref{lemma2.2},
\EQN{
\int_0^\infty& \frac{e^{-\frac{cy_n^2}t}}{(|x|+y_n+\sqrt{t}+1)^{n}(y_n+1)^{a-n+1}}\, dy_n \\
& \lec \frac1{(|x|+\sqrt{t}+1)^{n}} \bkt{\int_0^{|x|+\sqrt{t}+1} \frac1{(y_n+1)^{a-n+1}}\, dy_n + \int_{|x|+\sqrt{t}+1}^\infty \frac{e^{-\frac{cy_n^2}t}}{(|x|+\sqrt{t}+1)^{a-n+1}}\, dy_n}\\
& \lec \frac1{(|x|+\sqrt{t}+1)^{n}} \bkt{\frac1{(|x|+\sqrt{t}+1)^{a-n}} + \frac1{(|x|+\sqrt{t}+1)^{a-n+1}} \int_0^\infty e^{-u^2}\sqrt{t}\,du}\\
& \lec \frac1{(|x|+\sqrt{t}+1)^{a}}.
}
If $a=n$, we have the same additional term from Lemma \ref{lemma2.2},
\EQN{
\int_0^\infty &\frac{e^{-\frac{cy_n^2}t}}{(|x|+y_n+\sqrt{t}+1)^{n}(y_n+1)}\, dy_n
 \lec \frac{1}{(|x|+\sqrt{t}+1)^{n}}
 \int_0^\infty \frac{e^{-\frac{cy_n^2}t}}{y_n+1}\, dy_n
\\
&\lec \frac{1}{(|x|+\sqrt{t}+1)^{n}}\bke{ \int_0^{\sqrt t} \frac1{y_n+1}\, dy_n + \int_{\sqrt t}^\infty \frac{e^{-\frac{cy_n^2}t}}{y_n}\, dy_n}
\\
&\lec \frac{1}{(|x|+\sqrt{t}+1)^{n}}\bke{\log(1+\sqrt t) + 1}.
}

We have proved
\[\norm{u^*}_{Y_a} \lec \norm{u_0}_{Y_a},\ \ \ 0< a<n;\quad
\norm{u^*(t)}_{Y_n} \lec \log(2+t) \norm{u_0}_{Y_n},
\]
and hence \eqref{1001b}.

We next consider \eqref{1001c}. For $k=0$ and $l+q=1$, by Proposition \ref{prop2} with $k=k_i=0$ and $q=1$ we have
\EQ{
|\pd_{y'}^l \pd_{y_n}^q G_{ij}(x,y,t)|\lec\frac1{(|x-y|^2+t)^{\frac{n+1}2}}
+\frac{1}
{(|x^*-y|^2+t)^{\frac{n}2}(y_n^2+t)^{\frac{1}2} }.
}
It suffices to show
\[
I_1+I_2\lec t^{-1/2}  \frac 1{\bka{x}^{a}}
\]
where
\EQN{
I_1&= \int _{\R^n_+} \frac1{(|x-y|+\sqrt t)^{n+1} } \frac 1{\bka{y}^{2a}}\,dy,
\\
I_2 &= \int _{\R^n_+} \frac1{(|x^*-y|+\sqrt t)^n (y_n+\sqrt t)} \frac 1{\bka{y}^{2a}}\,dy .
}

For $I_1$,
by Lemma \ref{lemma2.2},
we have
\EQN{
I_1 &\le \int _{\R^n} \frac1{(|x-y|+\sqrt t)^{n+1} } \frac 1{(|y|+1)^{2a}}\,dy
\\
&\lec \frac 1{(|x|+\sqrt t + 1)^{2a} \sqrt t } + \frac 1{(|x|+\sqrt t + 1)^{n+1}} \bke{ \mathbbm{1}_{2a=n}\log(|x|+\sqrt t + 1)  +  \one_{2a>n} }
}
Thus, if $0<a \le n$,
\EQ{\label{1006a}
I_1 \lec \frac 1{(|x|+\sqrt t + 1)^{a} \sqrt t } .
}

For $I_2$,
let $A=x_n+y_n+\sqrt t$.
We have
\[
I_2 \lec \int_0^\infty \bke{ \int_\Si  \frac1{(|x'-y'|+ A)^n(|y'|+ y_n+1)^{2a}}\,dy'} \,\frac{dy_n}{y_n+\sqrt t}.
\]
Let $R=|x'|+A+(y_n+1) \sim |x|+y_n + 1+\sqrt t$.
By Lemma \ref{lemma2.2},
\EQN{
I_2 &\lec \int_0^\infty \bke{ R^{-2a} A^{-1} + R^{-n} \bke{ \mathbbm 1_{2a=n-1}  \log \frac R{y_n+1}
+ \frac{\mathbbm 1_{2a>n-1}} {  (y_n+1)^{2a+1-n}} }
} \frac{dy_n}{y_n+\sqrt t}
\\
&=I_3+I_4+I_5.
}
We have
\[
I_3 \lec \int_0^\infty \frac {dy_n}{( |x|+1+\sqrt t)^{2a}(y_n+\sqrt t)^2}
\lec \frac {1}{(|x|+1+\sqrt t)^{2a}\sqrt t}.
\]

If $2a=n-1$, for any $0<\ep <a$, we have $n-1-\ep>a$ and
\EQN{
I_4 &\lec \int_0^\infty \frac {\log (y_n + |x|+1+\sqrt t)  }{(y_n + |x|+1+\sqrt t)^{n}\sqrt t} dy_n
\\
&\lec \frac {1 }{( |x|+1+\sqrt t)^{n-1-\ep}\sqrt t}
\lec \frac {1 }{( |x|+1+\sqrt t)^{a}\sqrt t} .
}

If $\frac {n-1}2 < a \le n$,
\EQN{
I_5 &\lec  \frac1{\sqrt t} \bke{\int_0^{|x| + 1 + \sqrt t}
+  \int_{|x| + 1 + \sqrt t}^\infty} \frac {dy_n}{(y_n + |x|+1+\sqrt t)^{n} (y_n+1)^{2a+1-n}}
\\
&\lec \frac1{\sqrt t}\int_0^{|x| +1+ \sqrt t} \frac {dy_n}{( |x|+1+\sqrt t)^{n}  (y_n+1)^{2a+1-n}}
+  \frac1{\sqrt t}\int_{|x| + 1+\sqrt t}^\infty  \frac {dy_n}{y_n^{2a+1}}
\\
& \lec \frac1{\sqrt t} \bke{ \frac {1 }
{( |x|+1+\sqrt t)^{2a}}
+ \frac {\mathbbm{1}_{2a=n} \log ( |x|+1+\sqrt t)}
{( |x|+1+\sqrt t)^{2a}}
+ \frac {\mathbbm{1}_{2a>n} }
{( |x|+1+\sqrt t)^{n}}}.
}

Thus, if $0<a \le n$,
\[
I_2 \lec I_3+I_4+I_5\lec \frac 1{(|x|+\sqrt t + 1)^{a} \sqrt t } .
\]
This and the $I_1$ estimate \eqref{1006a} show \eqref{1001c}.
\end{proof}

\emph{Remark.} In the proof of  \eqref{1001c}, we use Proposition \ref{prop2} instead of Theorem \ref{thm3} to avoid  $\LN$ since $\mu_{jq}$ may be $1$ when $q=1$.

\medskip

We next consider Theorem \ref{thm1.8}. It can be proved using the same iteration argument in \cite{CJ-NA2017} and the estimates in the following.

\begin{lem}\label{lem9.3}
Let $n\ge 2$ and $0 \le a \le 1$. %
For $u_0\in Z_a$ with $\div u_0 =0$ and $u_{0,n}|_\Si=0$,
\begin{equation}\label{weight-xn-10}
\norm{\sum_{j=1}^n \int _{\R^n_+} \breve G_{ij}(x,y,t) u_{0,j}(y) dy }_{Z_a} \le C (1 + \delta_{a1} \log_+t)  \norm{u_0}_{Z_a}.
\end{equation}
For $F\in Z_{2a}$,
\begin{equation}\label{weight-xn-20}
\norm{\int _{\R^n_+} \pd_{y_p} G_{ij}(x,y,t) F_{pj}(y)\,dy }_{Z_a} \le C t^{-1/2} \norm{F}_{Z_{2a}}.
\end{equation}
\end{lem}

Our estimates for both inequalities fail for $a >1$. See Remark \ref{rem9.4} after the proof.

\begin{proof}
If $a=0$, the lemma follows from \eqref{E10.5} and \eqref{E10.6} with $p=q=\infty$.
Thus, we only consider $0<a\le 1$. We may suppose that $\norm{u_0}_{Z_a}=1$ without loss of generality.
For \eqref{weight-xn-10},
write
\EQN{
\sum_{j=1}^n\int _{\R^n_+} \breve G_{ij}(x,y,t) u_{0,j}(y) dy &=\int _{\R^n_+} \Ga(x-y,t) u_{0,i}(y) dy + \int _{\R^n_+} G_{ij}^*(x,y,t) u_{0,j}(y) dy\\
& = :
u^{heat}_i(x,t) + u_i^*(x,t).
}

Denote by $\Gamma_k$ the $k$-dimensional heat kernel.
When $|u_0(y)| \le \bka{y_n}^{-a}$, we have
\EQN{
|u^{heat}_i(x,t) | \lec &\int_0^\infty \frac{\Gamma_{1}(x_n-y_n,t)}{(y_n+1)^a} \int_{\Si} \Gamma_{n-1}(x'-y',t)\, dy'dy_n\\
\lec & \int_0^\infty \frac{\Gamma_{1}(x_n-y_n,t)}{(y_n+1)^a}  dy_n\\
\lec &  (1+\de_{a=1}\log_+ t)(x_n+1)^{-a}.
}
We have used the one dimensional version of \eqref{Knightly} for the last inequality and $0<a\le 1$.

For $u^*$ with $|u_0(y)| \le \bka{y_n}^{-a}$, by  \eqref{Solonnikov.est},  (for both $n \ge 3$ and $n=2$) we get
\[
|u^*(x,t)| \lec
J(x,t)=  \int _{\R^n_+}\frac{e^{-\frac{cy_n^2}t}} {(|x^*-y|^2+t)^{\frac{n}2} \bka{y_n}^a} \,dy.
\]
%
%
For $0 < a<\infty$,
 we have
\EQN{
J &\lec \int_0^\infty \frac{e^{-\frac{cy_n^2}t}}{(y_n+1)^a} \int_{\Si} \frac1{(|x'-y'|+x_n+y_n+\sqrt{t})^n}\, dy'dy_n\\
&\lec \int_0^\infty { \frac{e^{-\frac{cy_n^2}t}}{(y_n+1)^a(x_n+y_n+\sqrt{t})}}\, dy_n \\
&\lec \frac1{x_n+\sqrt{t}} \int_0^{x_n+\sqrt{t}} \frac1{(y_n+1)^a}\, dy_n + \frac1{(x_n+\sqrt{t}+1)^a} \int_{x_n+\sqrt{t}}^\infty \frac{e^{-c\frac{y_n^2}t}}{y_n+\sqrt{t}}\, dy_n.
}
Using Lemma \ref{lem6-1} to bound the first integral, we have
\EQN{
J &\lec \frac1{x_n+\sqrt{t}}\, \frac{(x_n+\sqrt{t})\bke{1 + \de_{a1} \log_+(x_n+\sqrt{t}) } }{(1+x_n+\sqrt{t})^{\min(a,1)}} + \frac1{(x_n+\sqrt{t}+1)^a} \int_0^\infty \frac{e^{-u^2}}{u+1}\, du\\
&\lec \frac{1 + \de_{a1} \log_+(x_n+\sqrt{t})}{(x_n+\sqrt{t}+1)^{\min(a,1)}}.
}

When $a=1$,
we want to improve the above numerator $1 + \de_{a1} \log_+(x_n+\sqrt{t})$ to a function of $t$ independent of $x_n$. It suffices to consider the case $x_n > 10+\sqrt t$. In this case,
\EQN{
J&\lec \frac1{x_n} \int_{0}^\infty \frac{e^{-c\frac{y_n^2}t}}{y_n+1}\, dy_n
\lec  \frac1{x_n} \int_{0}^{\sqrt t} \frac{1}{y_n+1}\, dy_n+ \frac1{x_n} \int_{\sqrt t}^\infty  \frac{e^{-c\frac{y_n^2}t}}{y_n}\, dy_n
\\
&\lec  \frac1{x_n} \log(\sqrt t +1)+ \frac1{x_n}.
}
We conclude when $a=1$, either $x_n > 10+\sqrt t$ or not,
\[
J \lec  \frac {\log(2+\sqrt{t})}{x_n+\sqrt{t}+1}.
\]
Combining the above estimates of $u^{heat}$ and $J$, the estimate \eqref{weight-xn-10} is deduced.

\medskip

Next, we will show \eqref{weight-xn-20}.
 For $k=0$ and $l+q=1$, by Proposition \ref{prop2} with $k=k_i=0$ and $q=1$, we have
\EQ{
|\pd_{y'}^l \pd_{y_n}^q G_{ij}(x,y,t)|\lec\frac1{(|x-y|^2+t)^{\frac{n+1}2}}
+\frac{1}
{(|x^*-y|^2+t)^{\frac{n}2}(y_n^2+t)^{\frac{1}2} }.
}
It suffices to show, for $a>0$,
\[
I_1+I_2\lec t^{-1/2}  \frac 1{\bka{x_n}^{a}}
\]
where
\EQN{
I_1&= \int _{\R^n_+} \frac1{(|x-y|+\sqrt t)^{n+1} } \frac 1{\bka{y_n}^{2a}}\,dy,
\\
I_2 &= \int _{\R^n_+} \frac1{(|x^*-y|+\sqrt t)^n (y_n+\sqrt t)} \frac 1{\bka{y_n}^{2a}}\,dy .
}
Indeed, via  Lemma \ref{lemma2.2}, we have
\EQN{
I_1
&\lec \int_0^\infty \frac{1}{(y_n+1)^{2a}} \int_{\Si} \frac1{(|x-y|+\sqrt{t})^{n+1}}\, dy'dy_n\\
&\lec \int_0^\infty \frac{1}{(y_n+1)^{2a}(|x_n-y_n|+\sqrt{t})^{2}}\, dy_n\\
&\lec R^{-1-2a} + \de_{2a=1} R^{-2} \log R
+ \mathbbm 1_{2a>1} R^{-2}
+  R^{-2a} t^{-1/2}\\
&\lec  t^{-1/2} \frac 1{\bka{x_n}^{a}},
}
where $R=x_n+\sqrt{t}+1$. We have used $a\le 1$ to bound $\one_{2a>1} R^{-2}\lec  t^{-1/2} \frac 1{\bka{x_n}^{a}}$.
On the other hand,
\EQN{
I_2
&\lec \int_0^\infty \frac{1}{(y_n+1)^{2a}(y_n+\sqrt t)} \int_{\Si} \frac1{(|x^*-y|+\sqrt{t})^{n}}\, dy'dy_n\\
&\lec \int_0^\infty \frac{1}{(y_n+1)^{2a}(y_n+\sqrt t)(x_n+y_n+\sqrt{t})}\, dy_n.
}
If $x_n \le 1$, we have
\[
I_2\lec \int_0^1\frac1{(y_n+\sqrt t)^2}\, dy_n+  \int_1^\infty \frac{1}{y_n^{2a+1}\sqrt t}\, dy_n
\lec \frac1{\sqrt t} .
\]

If $x_n \ge 1$, using $0<a\le 1$ we have
\EQN{
I_2
&\lec
\int_0^\infty \frac{1}{(y_n+1)^{2a}(\sqrt t)\, x_n^a (y_n+1)^{1-a}}\, dy_n
\\
&= \frac1{x_n^a \sqrt t}\int_0^\infty \frac{1}{(y_n+1)^{1+a}}\, dy_n =  \frac c{x_n^a \sqrt t}.
}

Combining the above estimates of $I_1$ and $I_2$, we obtain \eqref{weight-xn-20}.
\end{proof}

\begin{remark}\label{rem9.4}
The restriction $a\le 1$ is used for both estimates of $u^{heat}$ and $J$ for \eqref{weight-xn-10} and
for both $I_1$ and $I_2$ for \eqref{weight-xn-20} in the above proof. In fact, $J$ has the lower bound for $t=1$ and all $a>0$,
\[
J(x,1) \gec \int_{0<y_n<1} \int_\Si \frac{dy'\,dy_n} {(|y'|+x_n+1)^n} \gec \frac 1{1+x_n}.
\]
\end{remark}

\subsection{Mild solutions in $L^q_\uloc$}
In this subsection we prove Lemma \ref{th9.6}.
The estimates in Lemma \ref{th9.6} are used by Maekawa, Miura and Prange to construct local in time mild solutions of \eqref{NS} in $L^q_\uloc(\R^n_+)$ in
\cite[Prop 7.1]{MMP1} for $n<q \le \infty$ and
\cite[Prop 7.2]{MMP1} for $q=n$.
Their same proofs give Theorem \ref{thm7}.

\begin{lem}\label{th9.6}
Let $n \ge 2$.
Let $1 \le p \le q \le \infty$.
For $u_0\in L^p_{\uloc,\si}$, %
\EQ{\label{S93eq1}
\norm{\sum_{j=1}^n \int _{\R^n_+}  \breve G_{ij}(x,y,t) u_{0,j}(y) dy }_{L^q_\uloc} \le C
 \bke{1+t^{-\frac{n}{2}(\frac{1}{p}-\frac{1}{q})}+  \one_{p=q=1} \ln_+ \frac {1}{ t}}
 \norm{u_0}_{L^p_\uloc}.
}
Let $F\in L^p_\uloc$, $a,b\in \NN_0$ and $1\le a+b$.
Assume $b\ge1$ and $n\ge3$ if $p=q=\infty$. Then
\EQ{\label{S93eq2}
\norm{\int _{\R^n_+} \pd_{x}^a\pd_{y}^b  G_{ij}(x,y,t) F(y) dy }_{L^q_\uloc} \le C t^{-\frac{a+b}{2}} \big(1+t^{-\frac{n}{2}(\frac{1}{p}-\frac{1}{q})}\big) \norm{F}_{L^p_\uloc}.
}
\end{lem}
These estimates correspond to \cite[Proposition 5.3]{MMP1} and \cite[Theorem 3]{MMP1}.
Their proof is based on resolvent estimates in \cite[Theorem 1]{MMP1}, which does not allow $q=1$. Thus our estimates for $p=q=1$ are new. Also note that we do not restrict $a,b\le 1$ as in \cite[Theorem 3]{MMP1}.

\begin{proof}
First consider \eqref{S93eq1}. The endpoint case $p=q=\infty$ follows from \eqref{E10.5}. Let $p<\infty$. The formula \eqref{E1.5} gives
\EQN{
\sum_{j=1}^n\int _{\R^n_+} \breve G_{ij}(x,y,t) u_{0,j}(y) dy
&=\int _{\R^n} \Ga(x-y,t) \one_{y_n>0}u_{0,i}(y) dy + \int _{\R^n_+} G_{ij}^*(x,y,t) u_{0,j}(y) dy\\
& = :
u^{heat}_i(x,t) + u_i^{*}(x,t).
}
Since $u^{heat}$ is a convolution with the heat kernel in $\R^n$, it satisfies the estimate in \eqref{S93eq1} by Maekawa-Terasawa \cite[(3.18)]{MaTe}.
It suffices now to show that $u^*(x,t)$ also satisfies the same estimate.
By \eqref{Solonnikov.est}, $u^*(x,t)$ is bounded by
\EQN{
J_t(x)&=  \int _{\R^n_+}\frac{e^{-\frac{cy_n^2}t}} {(|x'-y'|+x_n+y_n+\sqrt t)^n } |u_0(y)| \,dy
\\
&= \int_0^\infty \frac{1} {(|x'|+x_n+y_n+\sqrt t)^n }*_\Si  |u_0(x',y_n)|  e^{-\frac{cy_n^2}t} \,dy_n,
}
where $*_\Si$ indicates convolution over $\Si$. Denote
\[
Q= [-\tfrac12,\tfrac12]^{n-1} \subset \Si, \quad Q_k = k + Q, \quad  k \in \ZZ^{n-1}.
\]
Our goal is to bound
\[
\norm{J_t}_{L^q(Q_{j'} \times (j_n,j_n+1))}
\]
by the right side of  \eqref{S93eq1}, uniformly for all $j' \in \ZZ^{n-1}$ and $j_n \in \NN_0$. By translation, we may assume $j'=0$.
Decompose
\[
J_t(x) = \sum_{k,l\in \ZZ^{n-1}}  \int_0^\infty
\frac{\one_{Q_k}(x')} {(|x'|+x_n+y_n+\sqrt t)^n }*_{\Si_{x'}} \bke{\one_{Q_l}(x') |u_0(x',y_n)|}
e^{-\frac{cy_n^2}t} \,dy_n.
\]

By Minkowski and Young inequalities with $1+\frac{1}{q}=\frac{1}{p}+\frac{1}{r}$,
\EQN{
&\norm{J_t(\cdot,x_n)}_{L^q(Q)}\\
&\lec \sum_{k,l\in \ZZ^{n-1},\, k-l \in 3Q}  \int_0^\I \norm{
\frac{\one_{Q_k}(x')} {(|x'| + x_n+y_n+\sqrt t)^n} *_{\Si_{x'}}\bke{ \one_{Q_l}(x') |u_0(x',y_n)|}
}_{L^q_{x'}(Q)}  e^{-\frac{cy_n^2}t}\,dy_n
\\
&\lec  \sum_{k,l\in \ZZ^{n-1},\, k-l \in 3Q} \int_0^\I I_k \cdot
\norm{ u_0(\cdot,y_n)}_{L^p(Q_l)}  e^{-\frac{cy_n^2}t}\,dy_n,
}
where
\EQN{
I_k &= \norm{\frac{1} {(|x'| + x_n+y_n+\sqrt t)^n}}_{L^r_{x'}(Q_k)}.
}
We have $I_k \approx (1+|k|+ x_n+y_n+\sqrt t)^{-n}$ when $k \not =0$, and
$I_0 \lec (1+ x_n+y_n+\sqrt t)^{-\frac{n-1}r} ( x_n+y_n+\sqrt t)^{-n+\frac{n-1}r}$ by Lemma \ref{lem6-1} .

By Minkowski inequality again with $I=(j_n,j_n+1)$,
\[
\norm{J_t}_{L^q(Q\times I)} =\norm{ \norm{J_t(\cdot,x_n)}_{L^q(Q)}}_{L^q_{x_n}(I)}
\lec  \sum_{k\in \ZZ^{n-1}} \int_0^\I
\norm{I_k}_{L^q_{x_n}(I)}
\norm{ u_0(\cdot,y_n)}_{L^p(k+4Q)}  e^{-\frac{cy_n^2}t}\,dy_n .
\]
We have $\norm{I_k}_{L^q_{x_n}(I)} \lec \frac{1} {(1+|k| +y_n+\sqrt t)^n}$ except when $k=0$ and $y_n+\sqrt t<1$. For $k=0$,
\EQN{
\norm{I_0}_{L^q_{x_n}(I)} &\lec \norm{\frac 1{ (x_n+y_n+\sqrt t)^{n-\frac{n-1}r}}}_{L^q_{x_n}(0,1)}
\lec\frac 1{ (y_n+\sqrt t)^{1+\frac{n-1}p - \frac nq}} + \one_{p=q=1} \ln_+ \frac {1}{y_n+\sqrt t},
}
using $n-\frac{n-1}r = 1 + (n-1)(\frac{1}{p}-\frac{1}{q})\ge 1$.
 Thus
\EQN{
\norm{J_t}_{L^q(Q\times I)}
&\lec \sum_{k\in \ZZ^{n-1}} \sum_{j=0}^\infty
\int_j^{j+1} \frac{1} {(1+|k| +y_n+\sqrt t)^n}
\norm{ u_0(\cdot,y_n)}_{L^p(k+4Q)}  e^{-\frac{cy_n^2}t}\,dy_n+M,
}
where
\EQN{
M= \int_0^1 \norm{I_0}_{L^q_{x_n}(I)}  \norm{ u_0(\cdot,y_n)}_{L^p(4Q)}  e^{-\frac{cy_n^2}t}\,dy_n.
}

By H\"older inequality with $p'=\frac p{p-1}$,
\EQN{
\norm{J_t}_{L^q(Q\times I)}
&\lec  \sum_{k\in \ZZ^{n-1}} \sum_{j=0}^\infty
\norm{ u_0}_{L^p((k+4Q)\times (j,j+1))}\cdot
\norm{
\frac{1} {(1+|k| +y_n+\sqrt t)^n}
  e^{-\frac{cy_n^2}t} }_ {L^{p'}_{y_n} (j,j+1)}+M
\\
&\lec  \sum_{k\in \ZZ^{n-1}} \sum_{j=0}^\infty
\norm{ u_0}_{L^p_\uloc}
\frac{1} {(1+|k| +j+\sqrt t)^n}
  e^{-\frac{cj^2}t}+M
  \\
&\lec
\norm{ u_0}_{L^p_\uloc} \int_{\R^n_+} \frac {  e^{-\frac{cy_n^2}t}}
{(1+|y| +\sqrt t)^n} dy+M \lec \norm{ u_0}_{L^p_\uloc}+M.
}
Also by H\"older inequality, when $(p,q)\not=(1,1)$,
\[
M
\lec
\norm{ u_0}_{L^p_\uloc}\cdot
\norm{\frac 1{ (y_n+\sqrt t)^{1+\frac{n-1}p - \frac nq}}
  e^{-\frac{cy_n^2}t} }_ {L^{p'}_{y_n}(0,1)}
\lec  t^{-\frac{n}{2}(\frac{1}{p}-\frac{1}{q})}
\norm{ u_0}_{L^p_\uloc},
\]
while when $p=q=1$,
\[
M
\lec
\norm{ u_0}_{L^1_\uloc}\cdot
\norm{
\bke{1+ \ln_+ \frac {1}{y_n+\sqrt t}} e^{-\frac{cy_n^2}t} }_{L^\infty(0,1) }
\lec \bke{1+ \ln_+ \frac {1}{ t}} \norm{ u_0}_{L^1_\uloc}.
\]

We have shown
\EQ{\label{1007b}
\norm{J_t}_{L^q_\uloc} \lec \bke{t^{-\frac{n}{2}(\frac{1}{p}-\frac{1}{q})}+  \one_{p=q=1} \ln_+ \frac {1}{ t}} \norm{ u_0}_{L^p_\uloc}.
}

This proves \eqref{S93eq1}.

\medskip

For \eqref{S93eq2},
 denote
 $$
 w(x,t) = \int _{\R^n_+} \pd_{x}^a\pd_{y}^b  G_{ij}(x,y,t) F(y) dy ,\quad m=a+b.
 $$
By \eqref{1015a}, $|w(x,t)|\lec w_1(x,t) + w_2(x,t)$, where
$w_1(x,t) = \int_{\R^n} H_t^0(x-y) |F(y)|dy$ with $H_t^0(x)$ given by \eqref{1015b-0}, and 
$w_2(x,t) = (x_n^2 + t)^{a/2}\int_{\R^n} H_t(x-y^*) |F(y)| (y_n^2+t)^{b/2}\,dy$ with $H_t^0(x)$ given by \eqref{1015b}.

For $w_1(x,t)$, by Maekawa-Terasawa \cite[Theorem 3.1]{MaTe} with $\frac1q=\frac1r+\frac1p-1$,
\EQN{
\|w_1(\cdot,t)\|_{L^q_{\uloc}}
&\lec
t^{-\frac{m}{2}}\bke{ t^{\frac{n}{2}(\frac{1}{q}-\frac{1}{p})} \norm{H_1^0}_{L^r(\R^n)} +\norm{H_1^0}_{L^1(\R^n)} } \norm{F}_{L^p_\uloc}.
}

It remains to estimate $w_2(x,t)$. When $p=q=\infty$, noting that $L^\infty_{\uloc} = L^\infty$, \eqref{S93eq2} follows from \eqref{E10.6}.
For $p<q$, we drop the factors $(y_n^2+t)^{-\frac{b}2}$ and $(x_n^2+t)^{-\frac{a}2}$ in \eqref{1015a}, $H_t(x-y^*)$ by $H_t(x-y)$, and applying Maekawa-Terasawa \cite[Theorem 3.1]{MaTe} with $\frac1q=\frac1r+\frac1p-1$,
\EQN{
\|w_2(\cdot,t)\|_{L^q_{\uloc}}
&\lec
t^{-\frac{m}{2}}\bke{ t^{\frac{n}{2}(\frac{1}{q}-\frac{1}{p})} \norm{H_1}_{L^r(\R^n)} +\norm{H_1}_{L^1(\R^n)} } \norm{F}_{L^p_\uloc}.
}
Note that $H_1\in L^r$ since $r>1$ when $p<q$.
Thus, we get for $p<q$ that
\[
\norm{w_2(\cdot,t)}_{L^q_{\uloc}} \lec t^{-\frac{m}{2}}\bke{ t^{\frac{n}{2}(\frac{1}{q}-\frac{1}{p})}  + 1 } \norm{F}_{L^p_\uloc}.
\]
This shows \eqref{S93eq2}.
\end{proof}

\section*{Acknowledgments}
The research of KK was partially supported by NRF-2019R1A2C1084685.
The research of BL was partially supported by NSFC-11971148.
The research of both CL and TT was partially supported by the NSERC grant RGPIN-2018-04137.
The research of CL is supported in part by the Simons Foundation Math + X Investigator Award \#376319 (Michael I. Weinstein).

\addcontentsline{toc}{section}{\protect\numberline{}{References}}
\small

\def\cprime{$'$}

\end{document}